\documentclass[12pt,reqno]{amsart}%
\usepackage{amsmath,amsthm,amscd,amsthm,upref,indentfirst}
\usepackage{amsfonts,mathrsfs}
\usepackage{amssymb,amsbsy,bm}
\usepackage{graphicx}
\usepackage{amsmath}
\usepackage{amsfonts}
\usepackage{amssymb}%
\usepackage{hyperref}

\usepackage[all]{xy}
\xyoption{tips}
\SelectTips {xy} {12}
\usepackage[UglyObsolete,small,nohug,heads=LaTeX]{diagrams}
\diagramstyle[labelstyle=\scriptstyle]
\usepackage[usenames]{color}
\theoremstyle{plain}
\newtheorem{theorem}{Theorem}
\newtheorem{assertion}[theorem]{Assertion}
\newtheorem{lemma}[theorem]{Lemma}
\newtheorem{proposition}[theorem]{Proposition}

\theoremstyle{definition}
\newtheorem{definition}[theorem]{Definition}

\newtheorem{conjecture}[theorem]{Conjecture}
\newtheorem{corollary}[theorem]{Corollary}

\theoremstyle{remark}
\newtheorem{remark}[theorem]{Remark}

\newtheorem{example}[theorem]{Example}
\numberwithin{equation}{section}
\numberwithin{theorem}{section}
\renewcommand{\mathfrak}{\textsc}
\renewcommand{\mathbf}{\bm}
\renewcommand{\textit}{\textsf}
\normalfont\upshape
\numberwithin{equation}{section}
\numberwithin{theorem}{section}

\usepackage[normalem]{ulem}
\usepackage[square,comma]{natbib}

\usepackage{multirow}
\usepackage{lscape}

\author{Steven Duplij} 
\date{\textbf{March 29, 2017}}

\address{\noindent Mathematisches Institute\\
Universit\"at M\"unster\newline
\mbox{} \hskip 10pt Einsteinstr. 62\\
D-48149 M\"unster,
Deutschland}
\email{duplijs@math.uni-muenster.de, sduplij@gmail.com}
\urladdr{http://homepages.spa.umn.edu/\~{}duplij}
\title[Arity shape of polyadic algebraic structures 
]{\textbf{Arity shape of polyadic algebraic structures
}
}

\bibdata{sde.bib,quant.bib,cstar.bib,sinbooks-eng.bib,litold.bib,nary.bib,blg.bib,comp.bib}
\bibstyle{chicago1cc}

\begin{document}
\begin{abstract}
\noindent Concrete two-set (module-like and algebra-like) 
algebraic structures are investigated from the viewpoint 
that the initial arities of all operations are arbitrary. 
Relations between operations arising from the structure 
definitions, however, lead to restrictions which determine their 
possible arity shapes and lead us to the partial arity freedom principle.
In this manner, polyadic vector spaces and algebras, 
dual vector spaces, direct sums, tensor products and 
inner pairing spaces are reconsidered. 
As one application, elements of polyadic operator theory 
are outlined: multistars and polyadic analogs of adjoints, 
operator norms, isometries and projections are introduced, as well      
as polyadic $C^{*}$-algebras, Toeplitz algebras and 
Cuntz algebras represented by polyadic operators. 
Another application is connected with number theory, 
and it is shown that congruence classes 
are polyadic rings of a special kind. Polyadic 
numbers are introduced (see Definition 6.16), and  
Diophantine equations over these 
polyadic rings are then considered. 
Polyadic analogs of the Lander-Parkin-Selfridge 
conjecture and Fermat's last theorem are formulated. 
For nonderived polyadic ring operations (on polyadic numbers) neither of these statements holds, 
and counterexamples are given. Finally, a procedure for obtaining new 
solutions to the equal sums of like powers equation over 
polyadic rings by applying Frolov's theorem to 
the Tarry-Escott problem is presented.
\end{abstract}
\maketitle
\thispagestyle{empty}
\vskip-2.3cm
\begin{small}
\textsc{\tableofcontents}
\end{small}

\mbox{}

\bigskip

\section*{Introduction}

The study of polyadic (higher arity) algebraic structures has a two-century
long history, starting with works by Cayley, Sylvester, Kasner, Pr\"{u}fer,
D\"{o}rnte, Lehmer, Post, etc. They took a single set, closed under one (main)
binary operation having special properties (the so called group-like
structure), and \textquotedblleft generalized\textquotedblright\ it by
increasing the arity of that operation, which can then be called a
\textit{polyadic operation} and the corresponding algebraic structure
\textit{polyadic} as well\footnote{We use the term \textquotedblleft
polyadic\textquotedblright\ in this sense only, while there are other uses in
the literature (see, e.g., \cite{halmos}).}. An \textquotedblleft abstract
way\textquotedblright\ to study polyadic algebraic structures is via the use
of universal algebras defined as sets with different axioms (equational laws)
for polyadic operations \cite{cohn,grat2,bergman2}. However, in this language
some important algebraic structures cannot be described, e.g. ordered groups,
fields, etc. \cite{den/wis}. Therefore, another \textquotedblleft concrete
approach\textquotedblright\ is to study examples of binary algebraic
structures and then to \textquotedblleft polyadize\textquotedblright\ them
properly. This initiated the development of a corresponding theory of $n$-ary
quasigroups \cite{belousov}, $n$-ary semigroups \cite{mon/sio1,zup67} and
$n$-ary groups \cite{galmak1,rusakov1} (for a more recent review, see, e.g.,
\cite{dup2012} and a comprehensive list of references therein). The binary
algebraic structures with two operations (addition and multiplication) on one
set (the so-called ring-like structures) were later on generalized to $\left(
m,n\right)  $-rings \cite{cel,cro2,lee/but} and $\left(  m,n\right)  $-fields
\cite{ian/pop97}, while these were investigated mostly in a more restrictive
manner by considering particular cases: ternary rings (or $\left(  2,3\right)
$-rings) \cite{lis}, $\left(  m,2\right)  $-rings \cite{boc65,pop/pop02}, as
well as $\left(  3,2\right)  $-fields \cite{dup/wer}.

In the case of one set, speaking informally, the \textquotedblleft
polyadization\textquotedblright\ of two operations' \textquotedblleft
interaction\textquotedblright\ is straightforward, giving only polyadic
distributivity which does not connect or restrict their arities. However, when
the number of sets becomes greater than one, the \textquotedblleft
polyadization\textquotedblright\ turns out to be non-trivial, leading to
special relations between the operation arities, and also introduces additional (to
arities) parameters, allowing us to classify them. We call a selection
of such relations an \textit{arity shape} and formulate the \textit{arity
partial freedom principle} that not all arities of the operations that arise during
\textquotedblleft polyadization\textquotedblright\ of binary operations are possible.

In this paper we consider two-set algebraic structures in the
\textquotedblleft concrete way\textquotedblright\ and provide the consequent
\textquotedblleft polyadization\textquotedblright\ of binary operations on
them for the so-called module-like structures (vector spaces) and algebra-like
structures (algebras and inner product spaces). The \textquotedblleft
polyadization\textquotedblright\ of binary scalar multiplication is defined
in terms of the multiactions introduced in \cite{dup2012}, having special arity
shapes parametrized by the number of intact elements ($\ell
_{\operatorname*{id}}$) in the corresponding multiactions. We then
\textquotedblleft polyadize\textquotedblright\ related constructions, such as
dual vector spaces and direct sums, and also tensor products, and show that,
as opposed to the binary case, they can be implemented in spaces of different
arity signatures. The \textquotedblleft polyadization\textquotedblright\ of
inner product spaces and related norms gives additional arity shapes and
restrictions. In the resulting \textsc{Table \ref{T}} we present the arity
signatures and shapes of the polyadic algebraic structures under consideration.

In the application part we note some starting points for polyadic operator
theory by introducing multistars and polyadic analogs of adjoints, operator
norms, isometries and projections. It is proved (\textbf{Theorem
\ref{theo-sym}}) that, if the polyadic inner pairing (the analog of the inner
product) is symmetric, then all multistars coincide and all polyadic operators
are self-adjoint (in contrast to the binary case). The polyadic analogs of
$C^{\ast}$-algebras, Toeplitz algebras and Cuntz algebras are presented in
terms of the polyadic operators introduced here, and a ternary example is given.

Another application is connected with number theory: we show that the internal
structure of the congruence classes is described by a polyadic ring having a
special arity signature (\textsc{Table \ref{T1}}), and these we call polyadic
integers (or numbers) $\mathbb{Z}_{\left(  m,n\right)  }$ (\textbf{Definition
\ref{def-polint}}). They are classified by polyadic shape invariants, and the
relations between them which give the same arity signature are established. Also,
the limiting cases are analyzed, and it is shown that in one such case the
polyadic rings can be embedded into polyadic fields with binary
multiplication, which leads to the so-called polyadic rational numbers
\cite{cro/tim}. We then consider Diophantine equations over these polyadic
rings in a straightforward manner: we change only the arities of the operations
(\textquotedblleft additions\textquotedblright\ and \textquotedblleft
multiplications\textquotedblright), but save their mutual \textquotedblleft
interaction\textquotedblright. In this way we try to \textquotedblleft
polyadize\textquotedblright\ the equal sums of like powers equation and
formulate polyadic analogs of the Lander-Parkin-Selfridge conjecture and of
Fermat's last theorem \cite{lan/par/sel}. It is shown, that in the simplest
case, when the polyadic \textquotedblleft addition\textquotedblright\ and
\textquotedblleft multiplication\textquotedblright\ are nonderived (e.g.,
for polyadic numbers), neither
conjecture is valid, and counterexamples are presented. Finally, we apply
Frolov's theorem to the Tarry-Escott problem \cite{dor/bro,ngu2016} over
polyadic rings to obtain new solutions to the equal sums of like powers
equation for fixed congruence classes.

\section{One set polyadic \textquotedblleft linear\textquotedblright%
\ structures}

We use concise notations from our previous work on polyadic structures
\cite{dup2012,dup2016}. Take a non-empty set $A$, then $n$\textit{-tuple} (or
\textit{polyad}) consisting of the elements $\left(  a_{1},\ldots
,a_{n}\right)  $, $a_{i}\in A$, is denoted by bold letter $\left(
\mathbf{a}\right)  $ taking it values in the Cartesian product $A^{\times n}$
. If the number of elements in the $n$-tuple is important, we denote it
$\left(  \mathbf{a}^{\left(  n\right)  }\right)  $, and an
$n$-tuple with equal elements is denoted by $\left(  a^{n}\right)  $. On the
Cartesian product $A^{\times n}$ one can define a polyadic operation
$\mathbf{\mu}_{n}:A^{\times n}\rightarrow A$, and use the notation
$\mathbf{\mu}_{n}\left[  \mathbf{a}\right]  $. A \textit{polyadic structure}
$\mathcal{A}$ is a set $A$ which is closed under polyadic operations, and a
\textit{polyadic signature} is the selection of their arities. For formal
definitions, see, e.g., \cite{cohn}.

\subsection{Polyadic distributivity}

Let us consider a polyadic structure with two operations on the same set $A$:
the \textquotedblleft chief\textquotedblright\ (\textit{multiplication})
$n$-ary operation $\mathbf{\mu}_{n}:A^{n}\rightarrow A$ and the additional
$m$-ary operation $\mathbf{\nu}_{m}:A^{m}\rightarrow A$, that is $\left\langle
A\mid\mathbf{\mu}_{n},\mathbf{\nu}_{m}\right\rangle $. If there are no
relations between $\mathbf{\mu}_{n}$ and $\mathbf{\nu}_{m}$, then nothing new,
as compared with the polyadic structures having a single operation
$\left\langle A\mid\mathbf{\mu}_{n}\right\rangle $ or $\left\langle
A\mid\mathbf{\nu}_{m}\right\rangle $, can be said. Informally, the
\textquotedblleft interaction\textquotedblright\ between operations can be
described using the important relation of distributivity (an analog of
$a\cdot\left(  b+c\right)  =a\cdot b+a\cdot c$, $a,b,c\in A$ in the binary case).

\begin{definition}
\label{def-dis}The \textit{polyadic distributivity} for the operations
$\mathbf{\mu}_{n}$ and $\mathbf{\nu}_{m}$ (no additional properties are
implied for now) consists of $n$ relations%
\begin{align}
&  \mathbf{\mu}_{n}\left[  \mathbf{\nu}_{m}\left[  a_{1},\ldots a_{m}\right]
,b_{2},b_{3},\ldots b_{n}\right] \nonumber\\
&  =\mathbf{\nu}_{m}\left[  \mathbf{\mu}_{n}\left[  a_{1},b_{2},b_{3},\ldots
b_{n}\right]  ,\mathbf{\mu}_{n}\left[  a_{2},b_{2},b_{3},\ldots b_{n}\right]
,\ldots\mathbf{\mu}_{n}\left[  a_{m},b_{2},b_{3},\ldots b_{n}\right]  \right]
\label{dis1}\\
&  \mathbf{\mu}_{n}\left[  b_{1},\mathbf{\nu}_{m}\left[  a_{1},\ldots
a_{m}\right]  ,b_{3},\ldots b_{n}\right] \nonumber\\
&  =\mathbf{\nu}_{m}\left[  \mathbf{\mu}_{n}\left[  b_{1},a_{1},b_{3},\ldots
b_{n}\right]  ,\mathbf{\mu}_{n}\left[  b_{1},a_{2},b_{3},\ldots b_{n}\right]
,\ldots\mathbf{\mu}_{n}\left[  b_{1},a_{m},b_{3},\ldots b_{n}\right]  \right]
\label{dis2}\\
&  \vdots\nonumber\\
&  \mathbf{\mu}_{n}\left[  b_{1},b_{2},\ldots b_{n-1},\mathbf{\nu}_{m}\left[
a_{1},\ldots a_{m}\right]  \right] \nonumber\\
&  =\mathbf{\nu}_{m}\left[  \mathbf{\mu}_{n}\left[  b_{1},b_{2},\ldots
b_{n-1},a_{1}\right]  ,\mathbf{\mu}_{n}\left[  b_{1},b_{2},\ldots
b_{n-1},a_{2}\right]  ,\ldots\mathbf{\mu}_{n}\left[  b_{1},b_{2},\ldots
b_{n-1},a_{m}\right]  \right]  , \label{dis3}%
\end{align}
where $a_{i},b_{j}\in A$.
\end{definition}

It is seen that the operations $\mathbf{\mu}_{n}$ and $\mathbf{\nu}_{m}$ enter
into (\ref{dis1})-(\ref{dis3}) in a non-symmetric way, which allows us to
distinguish them: one of them ($\mathbf{\mu}_{n}$, the $n$-ary multiplication)
\textquotedblleft distributes\textquotedblright\ over the other one
$\mathbf{\nu}_{m}$, and therefore $\mathbf{\nu}_{m}$ is called the
\textit{addition}. If only some of the relations (\ref{dis1})-(\ref{dis3})
hold, then such distributivity is \textit{partial} (the analog of left and
right distributivity in the binary case). Obviously, the operations
$\mathbf{\mu}_{n}$ and $\mathbf{\nu}_{m}$ need have nothing to do with
ordinary multiplication (in the binary case denoted by $\mathbf{\mu}%
_{2}\Longrightarrow\left(  \cdot\right)  $) and addition (in the binary case
denoted by $\nu_{2}\Longrightarrow\left(  +\right)  $), as in the example below.

\begin{example}
Let $A=\mathbb{R}$, $n=2$, $m=3$, and $\mathbf{\mu}_{2}\left[  b_{1}%
,b_{2}\right]  =b_{1}^{b_{2}}$, $\mathbf{\nu}_{3}\left[  a_{1},a_{2}%
,a_{3}\right]  =a_{1}a_{2}a_{3}$ (product in $\mathbb{R}$). The partial
distributivity now is $\left(  a_{1}a_{2}a_{3}\right)  ^{b_{2}}=a_{1}^{b_{2}%
}a_{2}^{b_{2}}a_{3}^{b_{2}}$ (only the first relation (\ref{dis1}) holds).
\end{example}

\subsection{Polyadic rings and fields}

Here we briefly remind the reader of one-set (ring-like) polyadic structures
(informally). Let both operations $\mathbf{\mu}_{n}$ and $\mathbf{\nu}_{m}$ be
(totally) \textit{associative}, which (in our definition \cite{dup2012}) means
independence of the composition of two operations under placement of the
internal operations (there are $n$ and $m$ such placements and therefore
$\left(  n+m\right)  $ corresponding relations)%
\begin{align}
\mathbf{\mu}_{n}\left[  \mathbf{a},\mathbf{\mu}_{n}\left[  \mathbf{b}\right]
,\mathbf{c}\right]   &  =invariant,\label{as1}\\
\mathbf{\nu}_{m}\left[  \mathbf{d},\mathbf{\nu}_{m}\left[  \mathbf{e}\right]
,\mathbf{f}\right]   &  =invariant, \label{as2}%
\end{align}
where the polyads $\mathbf{a}$, $\mathbf{b}$, $\mathbf{c}$, $\mathbf{d}$,
$\mathbf{e}$, $\mathbf{f}$ have corresponding length, and then both
$\left\langle A\mid\mathbf{\mu}_{n}\mid assoc\right\rangle $ and $\left\langle
A\mid\mathbf{\nu}_{m}\mid assoc\right\rangle $ are \textit{polyadic
semigroups} $\mathcal{S}_{n}$ and $\mathcal{S}_{m}$. A \textit{commutative
semigroup} $\left\langle A\mid\mathbf{\nu}_{m}\mid assoc,comm\right\rangle $
is defined by $\nu_{m}\left[  \mathbf{a}\right]  =\nu_{m}\left[  \sigma
\circ\mathbf{a}\right]  $, for all $\sigma\in S_{n}$, where $S_{n}$ is the
symmetry group. If the equation $\nu_{m}\left[  \mathbf{a},x,\mathbf{b}%
\right]  =c$ is solvable for any place of $x$, then $\left\langle
A\mid\mathbf{\nu}_{m}\mid assoc,solv\right\rangle $ is a\ \textit{polyadic
group }$\mathcal{G}_{m}$, and such $x=\tilde{c}$ is called a (additive)
\textit{querelement} for $c$, which defines the (additive) \textit{unary
queroperation} $\tilde{\nu}_{1}$ by $\tilde{\nu}_{1}\left[  c\right]
=\tilde{c}$.

\begin{definition}
\label{def-ring}A \textit{polyadic }$\left(  m,n\right)  $-\textit{ring}
$\mathcal{R}_{m,n}$ is a set $A$ with two operations $\mathbf{\mu}_{n}%
:A^{n}\rightarrow A$ and $\mathbf{\nu}_{m}:A^{m}\rightarrow A$, such that: 1)
they are distributive (\ref{dis1})-(\ref{dis3}); 2) $\left\langle A\mid\mu
_{n}\mid assoc\right\rangle $ is a polyadic semigroup; 3) $\left\langle
A\mid\nu_{m}\mid assoc,comm,solv\right\rangle $ is a commutative polyadic group.
\end{definition}

It is obvious that a $(2,2)$-ring $\mathcal{R}_{2,2}$ is an ordinary (binary)
ring. Polyadic rings have much richer structure and can have unusual
properties \cite{cel,cro2,cup,lee/but}. If the multiplicative semigroup
$\left\langle A\mid\mathbf{\mu}_{n}\mid assoc\right\rangle $ is commutative,
$\mathbf{\mu}_{n}\left[  \mathbf{a}\right]  =\mathbf{\mu}_{n}\left[
\sigma\circ\mathbf{a}\right]  $, for all $\sigma\in S_{n}$, then
$\mathcal{R}_{m,n}$ is called a \textit{commutative polyadic ring}, and if it
contains the identity, then $\mathcal{R}_{m,n}$ is a (polyadic) $\left(
m,n\right)  $-\textit{semiring}. If the distributivity is only partial, then
$\mathcal{R}_{m,n}$ is called a \textit{polyadic near-ring}.

Introduce in $\mathcal{R}_{m,n}$ additive and multiplicative idempotent
elements by $\mathbf{\nu}_{m}\left[  a^{m}\right]  =a$ and $\mathbf{\mu}%
_{n}\left[  b^{n}\right]  =b$, respectively. A \textit{zero} $z$ of
$\mathcal{R}_{m,n}$ is defined by $\mathbf{\mu}_{n}\left[  z,\mathbf{a}%
\right]  =z$ for any $\mathbf{a}\in A^{n-1}$, where $z$ can be on any place.
Evidently, a zero (if it exists) is a multiplicative idempotent and is unique,
and, if a polyadic ring has an additive idempotent, it is a zero
\cite{lee/but}. Due to the distributivity (\ref{dis1})-(\ref{dis3}), there can
be at most one zero in a polyadic ring. If a zero $z$ exists, denote $A^{\ast
}=A\setminus\left\{  z\right\}  $, and observe that (in distinction to binary
rings) $\left\langle A^{\ast}\mid\mathbf{\mu}_{n}\mid assoc\right\rangle $ is
not a polyadic group, in general. In the case where $\left\langle A^{\ast}%
\mid\mathbf{\mu}_{n}\mid assoc\right\rangle $ is a commutative $n$-ary group,
such a polyadic ring is called a (\textit{polyadic}) $\left(  m,n\right)
$-\textit{field} and $\mathbb{K}_{m,n}$ (\textquotedblleft polyadic
scalars\textquotedblright) (see \cite{lee/but,ian/pop97}).

A multiplicative \textit{identity} $e$ in $\mathcal{R}_{m,n}$ is a
distinguished element $e$ such that%
\begin{equation}
\mu_{n}\left[  a,\left(  e^{n-1}\right)  \right]  =a, \label{e}%
\end{equation}
for any $a\in A$ and where $a$ can be on any place. In binary rings the
identity is the only neutral element, while in polyadic rings there can exist
many \textit{neutral }$\left(  n-1\right)  $\textit{-polyads} $\mathbf{e}$
satisfying%
\begin{equation}
\mu_{n}\left[  a,\mathbf{e}\right]  =a, \label{e1}%
\end{equation}
for any $a\in A$ which can also be on any place. The neutral polyads
$\mathbf{e}$ are not determined uniquely. Obviously, the polyad $\left(
e^{n-1}\right)  $ is neutral. There exist exotic polyadic rings which have no
zero, no identity, and no additive idempotents at all (see, e.g.,
\cite{cro2}), but, if $m=2$, then a zero always exists \cite{lee/but}.

\begin{example}
\label{exam-ring}Let us consider a polyadic ring $\mathcal{R}_{3,4}$ generated
by 2 elements $a$, $b$ and the relations%
\begin{align}
\mu_{4}\left[  a^{4}\right]   &  =a,\ \ \ \ \mu_{4}\left[  a^{3},b\right]
=b,\ \ \ \ \mu_{4}\left[  a^{2},b^{2}\right]  =a,\ \ \ \ \mu_{4}\left[
a,b^{3}\right]  =b,\ \ \ \ \mu_{4}\left[  b^{4}\right]  =a,\\
\nu_{3}\left[  a^{3}\right]   &  =b,\ \ \ \ \nu_{3}\left[  a^{2},b\right]
=a,\ \ \ \text{\ }\nu_{3}\left[  a,b^{2}\right]  =b,\ \ \ \ \nu_{3}\left[
b^{3}\right]  =a,
\end{align}
which has a multiplicative idempotent $a$ only, but has no zero and no identity.
\end{example}

\begin{proposition}
In the case of polyadic structures with two operations on one set there are no
conditions between arities of operations which could follow from
distributivity (\ref{dis1})-(\ref{dis3}) or the other relations above, and
therefore they have no arity shape.
\end{proposition}

Such conditions will appear below, when we consider more complicated universal
algebraic structures with two or more sets with operations and relations.

\section{Two set polyadic structures}

\subsection{Polyadic vector spaces}

Let us consider a polyadic field $\mathbb{K}_{m_{K},n_{K}}=\left\langle
K\mid\mathbf{\sigma}_{m_{K}},\mathbf{\kappa}_{n_{K}}\right\rangle $
(\textquotedblleft polyadic scalars\textquotedblright), having $m_{K}$-ary
addition $\mathbf{\sigma}_{m_{K}}:K^{m_{K}}\rightarrow K$ and\ $n_{K}$-ary
multiplication $\mathbf{\kappa}_{n_{K}}:K^{n_{K}}\rightarrow K$, and the
identity $e_{K}\in K$, a neutral element with respect to multiplication
$\mathbf{\kappa}_{n_{K}}\left[  e_{K}^{n_{K}-1},\lambda\right]  =\lambda$, for
all $\lambda\in K$. In polyadic structures, one can introduce a neutral
$\left(  n_{K}-1\right)  $-polyad (\textit{identity polyad} for
\textquotedblleft scalars\textquotedblright) $\mathbf{e}_{K}\in K^{n_{K}-1}$
by%
\begin{equation}
\mathbf{\kappa}_{n_{K}}\left[  \mathbf{e}_{K},\lambda\right]  =\lambda.
\end{equation}
where $\lambda\in K$ can be on any place.

Next, take a $m_{V}$-ary commutative (abelian) group $\left\langle
\mathsf{V}\mid\mathbf{\nu}_{m_{V}}\right\rangle $, which can be treated as
\textquotedblleft polyadic vectors\textquotedblright\ with $m_{V}$-ary
addition $\mathbf{\nu}_{m_{V}}:\mathsf{V}^{m_{V}}\rightarrow\mathsf{V}$.
Define in $\left\langle \mathsf{V}\mid\mathbf{\nu}_{m_{V}}\right\rangle $ an
additive neutral element (zero) $\mathsf{z}_{V}\in\mathsf{V}$ by%
\begin{equation}
\mathbf{\nu}_{m_{V}}\left[  \mathsf{z}_{V}^{m_{V}-1},\mathsf{v}\right]
=\mathsf{v} \label{z}%
\end{equation}
for any $\mathsf{v}\in\mathsf{V}$, and a \textquotedblleft negative
vector\textquotedblright\ $\mathsf{\bar{v}}\in\mathsf{V}$ as its querelement%
\begin{equation}
\mathbf{\nu}_{m_{V}}\left[  \mathbf{a}_{V},\mathsf{\bar{v}},\mathbf{b}%
_{V}\right]  =\mathsf{v},
\end{equation}
where $\mathsf{\bar{v}}$ can be on any place in the l.h.s., and $\mathbf{a_{V}%
,b_{V}}$ are polyads in $\mathsf{V}$. Here, instead of one neutral element we
can also introduce the $\left(  m_{V}-1\right)  $-polyad $\mathsf{z}_{V}$
(which may not be unique), and so, for a \textit{zero polyad} (for
\textquotedblleft vectors\textquotedblright) we have%
\begin{equation}
\mathbf{\nu}_{m_{V}}\left[  \mathsf{z}_{V},\mathsf{v}\right]  =\mathsf{v}%
,\ \ \ \forall\mathsf{v}\in\mathsf{V,}%
\end{equation}
where $\mathsf{v}\in\mathsf{V}$ can be on any place. The \textquotedblleft
interaction\textquotedblright\ between \textquotedblleft polyadic
scalars\textquotedblright\ and \textquotedblleft polyadic
vectors\textquotedblright\ (the analog of binary multiplication by a scalar
$\lambda\mathsf{v}$) can be defined as a multiaction ($k_{\rho}$-place action)
introduced in \cite{dup2012}%
\begin{equation}
\mathbf{\rho}_{k_{\rho}}:K^{k_{\rho}}\times\mathsf{V}\longrightarrow
\mathsf{V}. \label{r}%
\end{equation}
The set of all multiactions form a $n_{\rho}$-ary semigroup $\mathcal{S}%
_{\rho}$ under composition. We can \textquotedblleft
normalize\textquotedblright\ the multiactions in a similar way, as multiplace
representations \cite{dup2012}, by (an analog of $1\mathsf{v}=\mathsf{v}$,
$\mathsf{v}\in\mathsf{V}$, $1\in K$)
\begin{equation}
\mathbf{\rho}_{k_{\rho}}\left\{  \left.
\begin{array}
[c]{c}%
e_{K}\\
\vdots\\
e_{K}%
\end{array}
\right\vert \mathsf{v}\right\}  =\mathsf{v}, \label{re}%
\end{equation}
for all $\mathsf{v}\in\mathsf{V}$, where $e_{K}$ is the identity of
$\mathbb{K}_{m_{K},n_{K}}$. In the case of an (ordinary) $1$-place (left)
action (as an external binary operation) $\mathbf{\rho}_{1}:K\times
\mathsf{V}\rightarrow\mathsf{V}$, its consistency with the polyadic field
multiplication $\mathbf{\kappa}_{n_{K}}$ under composition of the binary
operations $\mathbf{\rho}_{1}\left\{  \lambda|a\right\}  $ gives a product of
the same arity%
\begin{equation}
n_{\rho}=n_{K},
\end{equation}
that is (a polyadic analog of $\lambda\left(  \mu\mathsf{v}\right)  =\left(
\lambda\mu\right)  \mathsf{v}$, $\mathsf{v}\in\mathsf{V}$, $\lambda,\mu\in K$)%
\begin{equation}
\mathbf{\rho}_{1}\left\{  \lambda_{1}|\mathbf{\rho}_{1}\left\{  \lambda
_{2}|\ldots|\mathbf{\rho}_{1}\left\{  \lambda_{n_{K}}|\mathsf{v}\right\}
\right\}  \ldots\right\}  =\mathbf{\rho}_{1}\left\{  \mathbf{\kappa}_{n_{K}%
}\left[  \lambda_{1},\lambda_{2},\ldots\lambda_{n_{K}}\right]  |\mathsf{v}%
\right\}  ,\ \ \ \lambda_{1},\ldots,\lambda_{n}\in K,\ \mathsf{v}\in
\mathsf{V}. \label{rm}%
\end{equation}

In the general case of $k_{\rho}$-place actions, the multiplication in the
$n_{\rho}$-ary semigroup $\mathcal{S}_{\rho}$ can be defined by the
\textit{changing arity formula} \cite{dup2012} (schematically)%
\begin{equation}
\mathbf{\rho}_{k_{\rho}}\overset{n_{\rho}}{\overbrace{\left\{  \left.
\begin{array}
[c]{c}%
\lambda_{1}\\
\vdots\\
\lambda_{k_{\rho}}%
\end{array}
\right\vert \left.
\begin{array}
[c]{c}%
\ \\
\ldots\\
\
\end{array}
\right\vert \mathbf{\rho}_{k\rho}\left\{  \left.
\begin{array}
[c]{c}%
\lambda_{k_{\rho}\left(  n_{\rho}-1\right)  }\\
\vdots\\
\lambda_{k_{\rho}n_{\rho}}%
\end{array}
\right\vert \mathsf{v}\right\}  \ldots\right\}  }}=\mathbf{\rho}_{k_{\rho}%
}\left\{  \left.
\genfrac{}{}{0pt}{}{\left.
\begin{array}
[c]{c}%
\mathbf{\kappa}_{n_{K}}\left[  \lambda_{1},\ldots\lambda_{n_{K}}\right]  ,\\
\vdots\\
\mathbf{\kappa}_{n_{K}}\left[  \lambda_{n_{K}\left(  \ell_{\mu}-1\right)
},\ldots\lambda_{n_{K}\ell_{\mu}}\right]
\end{array}
\right\}  \ell_{\mu}}{\left.
\begin{array}
[c]{c}%
\lambda_{n_{K}\ell_{\mu}+1},\\
\vdots\\
\lambda_{n_{K}\ell_{\mu}+\ell_{\operatorname*{id}}}%
\end{array}
\right\}  \ell_{\operatorname*{id}}}%
\right\vert \mathsf{v}\right\}  , \label{rrk}%
\end{equation}
where $\ell_{\mu}$ and $\ell_{\operatorname*{id}}$ are both integers. The
associativity of (\ref{rrk}) in each concrete case can be achieved by applying
the \textit{associativity quiver} concept from \cite{dup2012}.

\begin{definition}
The $\ell$\textit{-shape} is a pair $\left(  \ell_{\mu},\ell
_{\operatorname*{id}}\right)  $ , where $\ell_{\mu}$ is the number of
multiplications and $\ell_{\operatorname*{id}}$ is the number of
\textit{intact elements} in the composition of operations.
\end{definition}

It follows from (\ref{rrk}),

\begin{proposition}
The arities of the polyadic field $\mathbb{K}_{m_{K},n_{K}}$, the arity
$n_{\rho}$ of the multiaction semigroup $\mathfrak{S}_{\rho}$ and the $\ell
$-shape of the composition satisfy%
\begin{align}
k_{\rho}n_{\rho}  &  =n_{K}\ell_{\mu}+\ell_{\operatorname*{id}},\label{k1}\\
k_{\rho}  &  =\ell_{\mu}+\ell_{\operatorname*{id}}. \label{kl}%
\end{align}

\end{proposition}

We can exclude $\ell_{\mu}$ or $\ell_{\operatorname*{id}}$ and obtain%

\begin{equation}
n_{\rho}=n_{K}-\dfrac{n_{K}-1}{k_{\rho}}\ell_{\operatorname*{id}%
},\ \ \ n_{\rho}=\dfrac{n_{K}-1}{k_{\rho}}\ell_{\mu}+1, \label{n1}%
\end{equation}
respectively, where $\tfrac{n_{K}-1}{k_{\rho}}\ell_{\operatorname*{id}}\geq1$
and $\tfrac{n_{K}-1}{k_{\rho}}\ell_{\mu}\geq1$ are integers. The following
inequalities hold%
\begin{equation}
1\leq\ell_{\mu}\leq k_{\rho},\ \ \ 0\leq\ell_{\operatorname*{id}}\leq k_{\rho
}-1,\ \ \ \ell_{\mu}\leq k_{\rho}\leq\left(  n_{K}-1\right)  \ell_{\mu
},\ \ \ 2\leq n_{\rho}\leq n_{K}. \label{l1}%
\end{equation}

\begin{remark}
The formulas (\ref{n1}) coincide with the \textit{arity changing formulas} for
heteromorphisms \cite{dup2012} applied to (\ref{rrk}).
\end{remark}

It follows from (\ref{k1}), that the $\ell$-shape is determined by the arities
and number of places $k_{\rho}$ by
\begin{equation}
\ell_{\mu}=\dfrac{k_{\rho}\left(  n_{\rho}-1\right)  }{n_{K}-1}%
,\ \ \ \ \ \ \ \ \ell_{\operatorname*{id}}=\dfrac{k_{\rho}\left(
n_{K}-n_{\rho}\right)  }{n_{K}-1}. \label{l1n}%
\end{equation}

Because we have two polyadic \textquotedblleft additions\textquotedblright%
\ $\mathbf{\nu}_{m_{V}}$ and $\mathbf{\sigma}_{m_{K}}$, we need to consider,
how the multiaction $\mathbf{\rho}_{k_{\rho}}$ \textquotedblleft
distributes\textquotedblright\ between each of them. First, consider
distributivity of the multiaction $\mathbf{\rho}_{k_{\rho}}$ with respect to
\textquotedblleft vector addition\textquotedblright\ $\mathbf{\nu}_{m_{V}}$ (a
polyadic analog of the binary $\lambda\left(  \mathsf{v}+\mathsf{u}\right)
=\lambda\mathsf{v}+\lambda\mathsf{u}$, $\mathsf{v},\mathsf{u}\in\mathsf{V}$,
$\lambda,\mu\in K$)%
\begin{equation}
\mathbf{\rho}_{k_{\rho}}\left\{  \left.
\begin{array}
[c]{c}%
\lambda_{1}\\
\vdots\\
\lambda_{k_{\rho}}%
\end{array}
\right\vert \mathbf{\nu}_{m_{V}}\left[  \mathsf{v}_{1},\ldots,\mathsf{v}%
_{m_{V}}\right]  \right\}  =\mathbf{\nu}_{m_{V}}\left[  \mathbf{\rho}%
_{k_{\rho}}\left\{  \left.
\begin{array}
[c]{c}%
\lambda_{1}\\
\vdots\\
\lambda_{k_{\rho}}%
\end{array}
\right\vert \mathsf{v}_{1}\right\}  ,\ldots,\mathbf{\rho}_{k_{\rho}}\left\{
\left.
\begin{array}
[c]{c}%
\lambda_{1}\\
\vdots\\
\lambda_{k_{\rho}}%
\end{array}
\right\vert \mathsf{v}_{m_{V}}\right\}  \right]  . \label{disv}%
\end{equation}
Observe that here, in distinction to (\ref{rrk}), there is no connection
between the arities $m_{V}$ and $k_{\rho}$.

Secondly, the distributivity of the multiaction $\mathbf{\rho}_{k_{\rho}}$
(\textquotedblleft multiplication by scalar\textquotedblright) with respect to
the \textquotedblleft field addition\textquotedblright\ (a polyadic analog of
$\left(  \lambda+\mu\right)  \mathsf{v}=\lambda\mathsf{v}+\mu\mathsf{v}$,
$\mathsf{v}\in A$, $\lambda,\mu\in K$) has a form similar to (\ref{rrk})
(which can be obtained from the changing arity formula \cite{dup2012})%
\begin{equation}
\mathbf{\rho}_{k_{\rho}}\overset{n_{\rho}}{\overbrace{\left\{  \left.
\begin{array}
[c]{c}%
\lambda_{1}\\
\vdots\\
\lambda_{k_{\rho}}%
\end{array}
\right\vert \left.
\begin{array}
[c]{c}%
\ \\
\ldots\\
\
\end{array}
\right\vert \mathbf{\rho}_{k\rho}\left\{  \left.
\begin{array}
[c]{c}%
\lambda_{k_{\rho}\left(  n_{\rho}-1\right)  }\\
\vdots\\
\lambda_{k_{\rho}n_{\rho}}%
\end{array}
\right\vert \mathsf{v}\right\}  \ldots\right\}  }}=\mathbf{\rho}_{k_{\rho}%
}\left\{  \left.
\genfrac{}{}{0pt}{}{\left.
\begin{array}
[c]{c}%
\mathbf{\sigma}_{m_{K}}\left[  \lambda_{1},\ldots\lambda_{m_{K}}\right]  ,\\
\vdots\\
\mathbf{\sigma}_{m_{K}}\left[  \lambda_{m_{K}\left(  \ell_{\mu}^{\prime
}-1\right)  },\ldots\lambda_{m_{K}\ell_{\mu}^{\prime}}\right]
\end{array}
\right\}  \ell_{\mu}^{\prime}}{\left.
\begin{array}
[c]{c}%
\lambda_{m_{K}\ell_{\mu}^{\prime}+1},\\
\vdots\\
\lambda_{m_{K}\ell_{\mu}^{\prime}+\ell_{\operatorname*{id}}}%
\end{array}
\right\}  \ell_{\operatorname*{id}}^{\prime}}%
\right\vert \mathsf{v}\right\}  , \label{disk}%
\end{equation}
where $\ell_{\rho}^{\prime}$ and $\ell_{\operatorname*{id}}^{\prime}$ are the
numbers of multiplications and \textit{intact elements} in the resulting
multiaction, respectively. Here the arities are not independent as in
(\ref{disv}), and so we have

\begin{proposition}
The arities of the polyadic field $\mathbb{K}_{m_{K},n_{K}}$, the arity
$n_{\rho}$ of the multiaction semigroup $\mathcal{S}_{\rho}$ and the $\ell
$-shape of the distributivity satisfy%
\begin{align}
k_{\rho}n_{\rho}  &  =m_{K}\ell_{\mu}^{\prime}+\ell_{\operatorname*{id}%
}^{\prime},\label{k1a}\\
k_{\rho}  &  =\ell_{\mu}^{\prime}+\ell_{\operatorname*{id}}^{\prime}.
\label{k1a1}%
\end{align}

\end{proposition}

It follows from (\ref{k1a})--(\ref{k1a1})
%

\begin{equation}
n_{\rho}=m_{K}-\dfrac{m_{K}-1}{k_{\rho}}\ell_{\operatorname*{id}}^{\prime
},\ \ \ \ \ \ n_{\rho}=\dfrac{m_{K}-1}{k_{\rho}}\ell_{\mu}^{\prime}+1.
\end{equation}
Here $\tfrac{m_{K}-1}{k_{\rho}}\ell_{\operatorname*{id}}^{\prime}\geq1$ and
$\tfrac{m_{K}-1}{k_{\rho}}\ell_{\mu}^{\prime}\geq1$ are integers, and we have
the inequalities%
\begin{equation}
1\leq\ell_{\mu}^{\prime}\leq k_{\rho},\ \ \ 0\leq\ell_{\operatorname*{id}%
}^{\prime}\leq k_{\rho}-1,\ \ \ \ell_{\mu}^{\prime}\leq k_{\rho}\leq\left(
m_{K}-1\right)  \ell_{\mu}^{\prime},\ \ \ 2\leq n_{\rho}\leq m_{K}.
\label{mnk}%
\end{equation}
Now, the $\ell$-shape of the distributivity is fully determined from the
arities and number of places $k_{\rho}$ by the arity shape formulas%
\begin{equation}
\ell_{\rho}^{\prime}=\dfrac{k_{\rho}\left(  n_{\rho}-1\right)  }{m_{K}%
-1},\ \ \ \ \ \ \ell_{\operatorname*{id}}^{\prime}=\dfrac{k_{\rho}\left(
m_{K}-n_{\rho}\right)  }{m_{K}-1}. \label{l2m}%
\end{equation}
It follows from (\ref{mnk}) that:

\begin{corollary}
The arity $n_{\rho}$ of the multiaction semigroup $\mathcal{S}_{\rho}$ is
\textsl{less than or equal} to the arity of the field addition $m_{K}$.
\end{corollary}

\begin{definition}
\label{def-vect}A polyadic ($\mathbb{K}$)-vector (\textquotedblleft
linear\textquotedblright) space over a polyadic field is the 2-set 4-operation
algebraic structure%
\begin{equation}
\mathcal{V}_{m_{K},n_{K},m_{V},k_{\rho}}=\left\langle K;\mathsf{V}%
\mid\mathbf{\sigma}_{m_{K}},\mathbf{\kappa}_{n_{K}};\mathbf{\nu}_{m_{V}}%
\mid\mathbf{\rho}_{k_{\rho}}\right\rangle , \label{v}%
\end{equation}
such that the following axioms hold:

1) $\left\langle K\mid\mathbf{\sigma}_{m_{K}},\mathbf{\kappa}_{n_{K}%
}\right\rangle $ is a polyadic $\left(  m_{K},n_{K}\right)  $-field
$\mathbb{K}_{m_{K},n_{K}}$;

2) $\left\langle A\mid\mathbf{\nu}_{m_{V}}\right\rangle $ is a commutative
$m_{V}$-ary group;

3) $\left\langle \mathbf{\rho}_{k_{\rho}}\mid composition\right\rangle $ is a
$n_{\rho}$-ary semigroup $\mathfrak{S}_{\rho}$;

4) Distributivity of the multiaction $\mathbf{\rho}_{k_{\rho}}$ with respect
to\ the \textquotedblleft vector addition\textquotedblright\ $\mathbf{\nu
}_{m_{V}}$ (\ref{disv});

5) Distributivity of $\mathbf{\rho}_{k_{\rho}}$ with respect to\ the
\textquotedblleft scalar addition\textquotedblright\ $\mathbf{\sigma}_{m_{K}}$
(\ref{disk});

6) Compatibility of $\mathbf{\rho}_{k_{\rho}}$ with the \textquotedblleft
scalar multiplication\textquotedblright\ $\mathbf{\kappa}_{n_{K}}$ (\ref{rrk});

7) Normalization of the multiaction $\mathbf{\rho}_{k_{\rho}}$ (\ref{re}).
\end{definition}

All of the arities in (\ref{v}) are independent and can be chosen arbitrarily,
but they fix the $\ell$-shape of the multiaction composition (\ref{rrk}) and
the distributivity (\ref{disk}) by (\ref{l1n}) and (\ref{l2m}), respectively.
Note that the main distinction from the binary case is the possibility for the
arity $n_{\rho}$ of the multiaction semigroup $\mathcal{S}_{\rho}$ to be arbitrary.

\begin{definition}
A polyadic $\mathbb{K}$-vector subspace is%
\begin{equation}
\mathcal{V}_{m_{K},n_{K},m_{V},k_{\rho}}^{sub}=\left\langle K;\mathsf{V}%
^{sub}\mid\mathbf{\sigma}_{m_{K}},\mathbf{\kappa}_{n_{K}};\mathbf{\nu}_{m_{V}%
}\mid\mathbf{\rho}_{k_{\rho}}\right\rangle , \label{vsub}%
\end{equation}
where the subset $\mathsf{V}^{sub}\subset\mathsf{V}$ is closed under all
operations $\mathbf{\sigma}_{m_{K}},\mathbf{\kappa}_{n_{K}},\mathbf{\nu
}_{m_{V}},\mathbf{\rho}_{k_{\rho}}$ and the axioms 1)-7).
\end{definition}

Let us consider a subset $\mathsf{S}=\left\{  \mathsf{v}_{1},\ldots
,\mathsf{v}_{d_{V}}\right\}  \subseteq\mathsf{V}$ (of $d_{V}$
\textquotedblleft vectors\textquotedblright), then a \textit{polyadic span} of
$\mathsf{S}$ is (a \textquotedblleft linear combination\textquotedblright)%
\begin{align}
&  \operatorname*{Span}\nolimits_{pol}^{\lambda}\left(  \mathsf{v}_{1}%
,\ldots,\mathsf{v}_{d_{V}}\right)  =\left\{  w\right\}  ,\label{span}\\
&  w=\mathbf{\nu}_{m_{V}}^{\ell_{\nu}}\left[  \mathbf{\rho}_{k_{\rho}}\left\{
\left.
\begin{array}
[c]{c}%
\lambda_{1}\\
\vdots\\
\lambda_{k_{\rho}}%
\end{array}
\right\vert \mathsf{v}_{1}\right\}  ,\ldots,\mathbf{\rho}_{k_{\rho}}\left\{
\left.
\begin{array}
[c]{c}%
\lambda_{\left(  d_{V}-1\right)  k_{\rho}}\\
\vdots\\
\lambda_{d_{V}k_{\rho}}%
\end{array}
\right\vert \mathsf{v}_{s}\right\}  \right]  ,
\end{align}
where $\left(  d_{V}\cdot k_{\rho}\right)  $ \textquotedblleft
scalars\textquotedblright\ play the role of coefficients (or coordinates as
columns consisting of $k_{\rho}$ elements from the polyadic field
$\mathbb{K}_{m_{K},n_{K}}$), and the number of \textquotedblleft
vectors\textquotedblright\ $s$ is connected with the \textquotedblleft number
of $m_{V}$-ary additions\textquotedblright\ $\ell_{\nu}$ by%
\begin{equation}
d_{V}=\ell_{\nu}\left(  m_{V}-1\right)  +1,
\end{equation}
while $\operatorname*{Span}\nolimits_{pol}^{\lambda}\mathsf{S}$ is the set of
all \textquotedblleft vectors\textquotedblright\ of this form (\ref{span}) (we
consider here finite \textquotedblleft sums\textquotedblright\ only).

\begin{definition}
A polyadic span $\mathsf{S}=\left\{  \mathsf{v}_{1},\ldots,\mathsf{v}_{d_{V}%
}\right\}  \subseteq\mathsf{V}$ is \textit{nontrivial}, if at least one
multiaction $\mathbf{\rho}_{k_{\rho}}$ in (\ref{span}) is nonzero.
\end{definition}

Since polyadic fields and groups may not contain zeroes, we need to redefine
the basic notions of equivalences. Let us take two different spans of the same
set $\mathsf{S}$.

\begin{definition}
A set $\left\{  \mathsf{v}_{1},\ldots,\mathsf{v}_{d_{V}}\right\}  $ is called
\textquotedblleft linear\textquotedblright\ \textit{polyadic independent}, if
from the equality of nontrivial spans, as $\operatorname*{Span}\nolimits_{pol}%
^{\lambda}\left(  \mathsf{v}_{1},\ldots,\mathsf{v}_{d_{V}}\right)
=\operatorname*{Span}\nolimits_{pol}^{\lambda^{\prime}}\left(  \mathsf{v}%
_{1},\ldots,\mathsf{v}_{d_{V}}\right)  $, it follows that all $\lambda
_{i}=\lambda_{i}^{\prime}$, $i=1,\ldots,d_{V}k_{\rho}$.
\end{definition}

\begin{definition}
A set $\left\{  \mathsf{v}_{1},\ldots,\mathsf{v}_{d_{V}}\right\}  $ is called
a \textit{polyadic basis} of a polyadic vector space $\mathcal{V}_{m_{K}%
n_{K}m_{V}k_{\rho}}$, if it spans the whole space $\operatorname*{Span}%
\nolimits_{pol}^{\lambda}\left(  \mathsf{v}_{1},\ldots,\mathsf{v}_{d_{V}%
}\right)  =\mathsf{V}$.
\end{definition}

In other words, any element of $\mathsf{V}$ can be uniquely presented in the
form of the polyadic \textquotedblleft linear combination\textquotedblright%
\ (\ref{span}). If a polyadic vector space $\mathcal{V}_{m_{K}n_{K}%
m_{V}k_{\rho}}$ has a finite basis $\left\{  \mathsf{v}_{1},\ldots
,\mathsf{v}_{d_{V}}\right\}  $, then any another basis $\left\{
\mathsf{v}_{1}^{\prime},\ldots,\mathsf{v}_{d_{V}}^{\prime}\right\}  $ has the
same number of elements.

\begin{definition}
\label{def-dim}The number of elements in the polyadic basis $\left\{
\mathsf{v}_{1},\ldots,\mathsf{v}_{d_{V}}\right\}  $ is called the
\textit{polyadic dimension} of $\mathcal{V}_{m_{K},n_{K},m_{V},k_{\rho}}$.
\end{definition}

\begin{remark}
The so-called $3$-vector space introduced and studied in \cite{dup/wer},
corresponds to $\mathcal{V}_{m_{K}=3,n_{K}=2,m_{V}=3,k_{\rho}=1}$.
\end{remark}

\subsection{One-set polyadic vector space}

A particular polyadic vector space is important: consider $\mathsf{V}=K$,
$\mathbf{\nu}_{m_{V}}=\mathbf{\sigma}_{m_{K}}$ and $m_{V}=m_{K}$, which gives
the following one-set \textquotedblleft linear\textquotedblright\ algebraic
structure (we call it a \textit{one-set polyadic vector space})%
\begin{equation}
\mathcal{K}_{m_{K},n_{K},k_{\rho}}=\left\langle K\mid\mathbf{\sigma}_{m_{K}%
},\mathbf{\kappa}_{n_{K}}\mid\mathbf{\rho}_{k_{\rho}}^{\lambda}\right\rangle ,
\end{equation}
where now the multiaction%
\begin{equation}
\mathbf{\rho}_{k_{\rho}}^{\lambda}\left\{  \left.
\begin{array}
[c]{c}%
\lambda_{1}\\
\vdots\\
\lambda_{k_{\rho}}%
\end{array}
\right\vert \lambda\right\}  ,\lambda,\lambda_{i}\in K,
\end{equation}
acts on $K$ itself (in some special way), and therefore can be called a
\textit{regular multiaction}. In the binary case $n_{K}=m_{K}=2$, the only
possibility for the regular action is the multiplication (by \textquotedblleft
scalars\textquotedblright) in the field $\mathbf{\rho}_{1}^{\lambda}\left\{
\left.  \lambda_{1}\right\vert \lambda\right\}  =\mathbf{\kappa}_{2}\left[
\lambda_{1}\lambda\right]  \left(  \equiv\lambda_{1}\lambda\right)  $, which
obviously satisfies the axioms 4)-7) of a vector space in \textbf{Definition
\ref{def-vect}}. In this way we arrive at the definition of the binary field
$\mathbb{K}\equiv\mathbb{K}_{2,2}=\left\langle K\mid\mathbf{\sigma}%
_{2},\mathbf{\kappa}_{2}\right\rangle $, and so a one-set \textsl{binary
vector space} coincides with the underlying field $\mathcal{K}_{m_{K}%
=2,n_{K}=2,k_{\rho}=1}=\mathbb{K}$, or as it is said \textquotedblleft a field
is a (one-dimensional) vector space over itself\textquotedblright.

\begin{remark}
In the polyadic case, the regular multiaction $\mathbf{\rho}_{k_{\rho}%
}^{\lambda}$ can be chosen, as any (additional to $\mathbf{\sigma}_{m_{K}}$,
$\mathbf{\kappa}_{n_{K}}$) function satisfying axioms 4)-7) of a polyadic
vector space and the number of places $k_{\rho}$ and the arity of the
semigroup of multiactions $\mathcal{S}_{\rho}$ can be \textsl{arbitrary}, in
general. Also, $\mathbf{\rho}_{k_{\rho}}^{\lambda}$ can be taken as a some
nontrivial combination of $\mathbf{\sigma}_{m_{K}}$, $\mathbf{\kappa}_{n_{K}}$
satisfying axioms 4)-7) (which admits a nontrivial \textquotedblleft
multiplication by scalars\textquotedblright).
\end{remark}

In the simplest regular (similar to the binary) case,%
\begin{equation}
\mathbf{\rho}_{k_{\rho}}^{\lambda}\left\{  \left.
\begin{array}
[c]{c}%
\lambda_{1}\\
\vdots\\
\lambda_{k_{\rho}}%
\end{array}
\right\vert \lambda\right\}  =\mathbf{\kappa}_{n_{K}}^{\ell_{\kappa}}\left[
\lambda_{1},\ldots,\lambda_{k_{\rho}},\lambda\right]  , \label{rkl}%
\end{equation}
where $\ell_{\kappa}$ is the number of multiplications $\mathbf{\kappa}%
_{n_{K}}$, and the number of places $k_{\rho}$ is now fixed by%
\begin{equation}
k_{\rho}=\ell_{\kappa}\left(  n_{K}-1\right)  ,
\end{equation}
while $\lambda$ in (\ref{rkl}) can be on any place due the commutativity of
the field multiplication $\mathbf{\kappa}_{n_{K}}$.

\begin{remark}
In general, the one-set polyadic vector space need not coincide with the
underlying polyadic field, $\mathcal{K}_{m_{K},n_{K},k_{\rho}}\neq
\mathbb{K}_{n_{K}m_{K}}$ (as \textsl{opposed} to the binary case), but can
have a more complicated structure which is determined by an additional
multiplace function, the multiaction $\mathbf{\rho}_{k_{\rho}}^{\lambda}$.
\end{remark}

\subsection{Polyadic algebras}

By analogy with the binary case, introducing an additional operation on
vectors, a multiplication which is distributive and \textquotedblleft
linear\textquotedblright\ with respect to \textquotedblleft
scalars\textquotedblright, leads to a polyadic generalization of the
(associative) algebra notion \cite{car4}. Here, we denote the second (except
for the 'scalars' $K$) set by $\mathsf{A}$ (instead of $\mathsf{V}$ above), on
which we define \textsl{two} operations: $m_{A}$-ary \textquotedblleft
addition\textquotedblright\ $\mathbf{\nu}_{m_{A}}:\mathsf{A}^{\times m_{A}%
}\rightarrow\mathsf{A}$ and $n_{A}$-ary \textquotedblleft
multiplication\textquotedblright\ $\mathbf{\mu}_{n_{A}}:\mathsf{A}^{\times
n_{A}}\rightarrow\mathsf{A}$. To interpret the $n_{A}$-ary operation as a true
multiplication, the operations $\mathbf{\mu}_{n_{A}}$ and $\mathbf{\nu}%
_{m_{A}}$ should satisfy polyadic distributivity (\ref{dis1})--(\ref{dis3})
(an analog of $\left(  \mathsf{a}+\mathsf{b}\right)  \cdot c=\mathsf{a}%
\cdot\mathsf{c}+\mathsf{b}\cdot\mathsf{c}$, with $\mathsf{a},\mathsf{b}%
,\mathsf{c}\in\mathsf{A}$). Then we should consider the \textquotedblleft
interaction\textquotedblright\ of this new operation $\mathbf{\mu}_{n_{A}}$
with the multiaction $\mathbf{\rho}_{k_{\rho}}$ (an analog of the
\textquotedblleft compatibility with scalars\textquotedblright\ $\left(
\lambda\mathsf{a}\right)  \cdot\left(  \mu\mathsf{b}\right)  =\left(
\lambda\mu\right)  \mathsf{a}\cdot\mathsf{b}$, $\mathsf{a},\mathsf{b}%
\in\mathsf{A}$, $\lambda,\mu\in K$). In the most general case, when all
arities are arbitrary, we have the \textit{polyadic compatibility} of
$\mathbf{\mu}_{n_{A}}$ with the field multiplication $\kappa_{n_{K}}$ as
follows%
\begin{align}
&  \mathbf{\mu}_{n_{A}}\left[  \mathbf{\rho}_{k_{\rho}}\left\{  \left.
\begin{array}
[c]{c}%
\lambda_{1}\\
\vdots\\
\lambda_{k_{\rho}}%
\end{array}
\right\vert \mathsf{a}_{1}\right\}  ,\ldots,\mathbf{\rho}_{k_{\rho}}\left\{
\left.
\begin{array}
[c]{c}%
\lambda_{k_{\rho}\left(  n_{A}-1\right)  }\\
\vdots\\
\lambda_{k_{\rho}n_{A}}%
\end{array}
\right\vert \mathsf{a}_{n_{A}}\right\}  \right] \nonumber\\
&  =\mathbf{\rho}_{k_{\rho}}\left\{  \left.
\genfrac{}{}{0pt}{}{\left.
\begin{array}
[c]{c}%
\mathbf{\kappa}_{n_{K}}\left[  \lambda_{1},\ldots,\lambda_{n_{K}}\right]  ,\\
\vdots\\
\mathbf{\kappa}_{n_{K}}\left[  \lambda_{n_{K}\left(  \ell_{\mu}^{\prime\prime
}-1\right)  },\ldots,\lambda_{n_{K}\ell_{\mu}^{\prime\prime}}\right]
\end{array}
\right\}  \ell_{\mu}^{\prime\prime}}{\left.
\begin{array}
[c]{c}%
\lambda_{n_{K}\ell_{\mu}^{\prime\prime}+1},\\
\vdots\\
\lambda_{n_{K}\ell_{\mu}^{\prime\prime}+\ell_{\operatorname*{id}}%
^{\prime\prime}}%
\end{array}
\right\}  \ell_{\operatorname*{id}}^{\prime\prime}}%
\right\vert \mathbf{\mu}_{n_{A}}\left[  \mathsf{a}_{1}\ldots\mathsf{a}_{n_{A}%
}\right]  \right\}  , \label{mr}%
\end{align}
where $\ell_{\mu}^{\prime\prime}$ and $\ell_{\operatorname*{id}}^{\prime
\prime}$ are the numbers of multiplications and intact elements in the
resulting multiaction, respectively.

\begin{proposition}
The arities of the polyadic field $\mathbb{K}_{m_{K},n_{K}}$, the arity
$n_{\rho}$ of the multiaction semigroup $\mathfrak{S}_{\rho}$ and the $\ell
$-shape of the polyadic compatibility (\ref{mr}) satisfy%
\begin{equation}
k_{\rho}n_{A}=n_{K}\ell_{\mu}^{\prime\prime}+\ell_{\operatorname*{id}}%
^{\prime\prime},\ \ \ \ \ \ \ k_{\rho}=\ell_{\mu}^{\prime\prime}%
+\ell_{\operatorname*{id}}^{\prime\prime}. \label{k1b}%
\end{equation}

\end{proposition}

We can exclude from (\ref{k1b}) $\ell_{\rho}^{\prime\prime}$ or $\ell
_{\operatorname*{id}}^{\prime\prime}$ and obtain%

\begin{equation}
n_{A}=n_{K}-\dfrac{n_{K}-1}{k_{\rho}}\ell_{\operatorname*{id}}^{\prime\prime
},\ \ \ \ n_{A}=\dfrac{n_{K}-1}{k_{\rho}}\ell_{\mu}^{\prime\prime}+1,
\end{equation}
where $\tfrac{n_{K}-1}{k_{\rho}}\ell_{\operatorname*{id}}^{\prime\prime}\geq1$
and $\tfrac{n_{K}-1}{k_{\rho}}\ell_{\mu}^{\prime\prime}\geq1$ are integer, and
the inequalities hold%
\begin{equation}
1\leq\ell_{\mu}^{\prime\prime}\leq k_{\rho},\ \ 0\leq\ell_{\operatorname*{id}%
}^{\prime\prime}\leq k_{\rho}-1,\ \ \ell_{\mu}^{\prime\prime}\leq k_{\rho}%
\leq\left(  n_{K}-1\right)  \ell_{\mu}^{\prime\prime},\ \ 2\leq n_{A}\leq
n_{K}. \label{nak}%
\end{equation}
It follows from (\ref{k1b}), that the $\ell$-shape is determined by the
arities and number of places $k_{\rho}$ as
\begin{equation}
\ell_{\mu}^{\prime\prime}=\dfrac{k_{\rho}\left(  n_{A}-1\right)  }{n_{K}%
-1},\ \ \ \ \ell_{\operatorname*{id}}^{\prime\prime}=\dfrac{k_{\rho}\left(
n_{K}-n_{A}\right)  }{n_{K}-1}. \label{l1r}%
\end{equation}

\begin{definition}
\label{def-alg}A polyadic (\textquotedblleft linear\textquotedblright) algebra
over a polyadic field is the 2-set 5-operation algebraic structure%
\begin{equation}
\mathcal{A}_{m_{K},n_{K},m_{A},n_{A},k_{\rho}}=\left\langle K;\mathsf{A}%
\mid\mathbf{\sigma}_{m_{K}},\mathbf{\kappa}_{n_{K}};\mathbf{\nu}_{m_{A}%
},\mathbf{\mu}_{n_{A}}\mid\mathbf{\rho}_{k_{\rho}}\right\rangle ,
\end{equation}
such that the following axioms hold:

1) $\left\langle K;\mathsf{A}\mid\mathbf{\sigma}_{m_{K}},\mathbf{\kappa
}_{n_{K}};\mathbf{\nu}_{m_{A}}\mid\mathbf{\rho}_{k_{\rho}}\right\rangle $ is a
polyadic vector space over a polyadic field $\mathbb{K}_{m_{K},n_{K}}$;

2) The algebra multiplication $\mathbf{\mu}_{n_{A}}$ and the algebra addition
$\mathbf{\nu}_{m_{A}}$ satisfy the polyadic distributivity (\ref{dis1}%
)--(\ref{dis3});

3) The multiplications in the algebra $\mathbf{\mu}_{n_{A}}$ and in the field
$\mathbf{\kappa}_{n_{K}}$ are compatible by (\ref{mr}).
\end{definition}

In the case where the algebra multiplication $\mathbf{\mu}_{n_{A}}$ is
associative (\ref{as1}), then $\mathcal{A}_{m_{K},n_{K},m_{A},n_{A},k_{\rho}}$
is an \textit{associative polyadic algebra} (for $k_{\rho}=1$ see
\cite{car4}). If $\mathbf{\mu}_{n_{A}}$ is commutative, $\mu_{n_{A}}\left[
\mathbf{a}_{A}\right]  =\mu_{n_{A}}\left[  \sigma\circ\mathbf{a}_{A}\right]
$, for any polyad in algebra $\mathbf{a}_{A}\in\mathsf{A}^{\times n_{A}}$ for
all permutations $\sigma\in S_{n}$, where $S_{n}$ is the symmetry group, then
$\mathcal{A}_{m_{K},n_{K},m_{A},n_{A},k_{\rho}}$ is called a
\textit{commutative polyadic algebra}. As in the $n$-ary (semi)group theory,
for polyadic algebras one can introduce special kinds of associativity and
partial commutativity. If the multiplication $\mathbf{\mu}_{n_{A}}$ contains
the identity $\mathsf{e}_{A}$ (\ref{e}) or a neutral polyad for any element,
such a polyadic algebra is called \textit{unital} or \textit{neutral-unital},
respectively. It follows from (\ref{nak}) that:

\begin{corollary}
In a polyadic (\textquotedblleft linear\textquotedblright) algebra the arity
of the algebra multiplication $n_{A}$ is \textsl{less than or equal} to the
arity of the field multiplication $n_{K}$.
\end{corollary}

\begin{proposition}
It all the operation $\ell$-shapes in (\ref{rrk}), (\ref{disk}) and (\ref{mr})
coincide%
\begin{equation}
\ell_{\mu}^{\prime\prime}=\ell_{\mu}^{\prime}=\ell_{\mu},\ \ \ \ \ \ \ \ell
_{\operatorname*{id}}^{\prime\prime}=\ell_{\operatorname*{id}}^{\prime}%
=\ell_{\operatorname*{id}},
\end{equation}
then, we obtain the conditions for the arities%
\begin{equation}
n_{K}=m_{K},\ \ \ \ \ \ \ n_{\rho}=n_{A},
\end{equation}
while $m_{A}$ and $k_{\rho}$ are not connected.
\end{proposition}

\begin{proof}
Use (\ref{l1n}) and (\ref{l1r}).
\end{proof}

\begin{proposition}
In the case of equal $\ell$-shapes the multiplication and addition of the
ground polyadic field (\textquotedblleft scalars\textquotedblright) should
coincide, while the arity $n_{\rho}$ of the multiaction semigroup
$\mathcal{S}_{\rho}$ should be the same as of the algebra multiplication
$n_{A}$, while the arity of the algebra addition $m_{A}$ and number of places
$k_{\rho}$ remain arbitrary.
\end{proposition}

\begin{remark}
The above $\ell$-shapes (\ref{l1n}), (\ref{l2m}), and (\ref{l1r}) are defined
by a pair of \textsl{integers}, and therefore the arities in them are not
arbitrary, but should be \textquotedblleft quantized\textquotedblright\ in the
same manner as the arities of heteromorphisms in \cite{dup2012}.
\end{remark}

Therefore, numerically the \textquotedblleft quantization\textquotedblright%
\ rules for the $\ell$-shapes (\ref{l1n}), (\ref{l2m}), and (\ref{l1r})
coincide and given in \textsc{Table \ref{T3}}.

\begin{table}[h]
\caption{\textquotedblleft Quantization\textquotedblright\ of arity $\ell
$-shapes}%
\label{T3}
\begin{center}%
\begin{tabular}
[c]{|c|c|c|c|}\hline
$k_{\rho}$ & $\ell_{\mu}\mid\ell_{\mu}^{\prime}\mid\ell_{\mu}^{\prime\prime}$
& $\ell_{\operatorname*{id}}\mid\ell_{\operatorname*{id}}^{\prime}\mid
\ell_{\operatorname*{id}}^{\prime\prime}$ & $%
\begin{array}
[c]{c}%
n_{K}\\
n_{\rho}%
\end{array}
\mid%
\begin{array}
[c]{c}%
m_{K}\\
n_{\rho}%
\end{array}
\mid%
\begin{array}
[c]{c}%
n_{K}\\
n_{A}%
\end{array}
$\\\hline\hline
$2$ & $1$ & $1$ & $%
\begin{array}
[c]{cccc}%
3, & 5, & 7, & \ldots\\
2, & 3, & 4, & \ldots
\end{array}
$\\\hline
$3$ & $1$ & $2$ & $%
\begin{array}
[c]{cccc}%
4, & 7, & 10, & \ldots\\
2, & 3, & 4, & \ldots
\end{array}
$\\\hline
$3$ & $2$ & $1$ & $%
\begin{array}
[c]{cccc}%
4, & 7, & 10, & \ldots\\
3, & 5, & 7, & \ldots
\end{array}
$\\\hline
$4$ & $1$ & $3$ & $%
\begin{array}
[c]{cccc}%
5, & 9, & 13, & \ldots\\
2, & 3, & 4, & \ldots
\end{array}
$\\\hline
$4$ & $2$ & $2$ & $%
\begin{array}
[c]{cccc}%
3, & 5, & 7, & \ldots\\
2, & 3, & 4, & \ldots
\end{array}
$\\\hline
$4$ & $3$ & $1$ & $%
\begin{array}
[c]{cccc}%
5, & 9, & 13, & \ldots\\
4, & 7, & 10, & \ldots
\end{array}
$\\\hline
\end{tabular}
\end{center}
\end{table}

Thus, we arrive at the following

\begin{theorem}
[\textit{The arity partial freedom principle}]The basic two-set polyadic
algebraic structures have non-free underlying operation arities which are
\textquotedblleft quantized\textquotedblright\ in such a way that their $\ell
$-shape is given by integers.
\end{theorem}

The above definitions can be generalized, as in the binary case by considering
a polyadic ring $\mathcal{R}_{m_{K},n_{K}}$ instead of a polyadic field
$\mathbb{K}_{m_{K},n_{K}}$. In this way a polyadic vector space becomes a
\textit{polyadic module over a ring} or \textit{polyadic }$\mathcal{R}%
$\textit{-module}, while a polyadic algebra over a polyadic field becomes a
\textit{polyadic algebra over a ring} or \textit{polyadic }$\mathcal{R}%
$-\textit{algebra}. All the axioms and relations between arities in the
\textbf{Definition \ref{def-vect}} and \textbf{Definition \ref{def-alg}}
remain the same. However, one should take into account that the ring
multiplication $\mathbf{\kappa}_{n_{K}}$ can be noncommutative, and therefore
for\textit{ }polyadic $\mathcal{R}$-modules and $\mathcal{R}$-algebras it is
necessary to consider many different kinds of multiactions $\mathbf{\rho
}_{k_{\rho}}$ (all of them are described in (\ref{rrk})). For instance, in the
ternary case this corresponds to left, right and central ternary modules, and
tri-modules \cite{car5,baz/bor/ker}.

\section{Mappings between polyadic algebraic structures}

Let us consider $D_{V}$ different polyadic vector spaces over the same
polyadic field $\mathbb{K}_{m_{K},n_{K}}$, as
\begin{equation}
\mathcal{V}_{m_{K},n_{K},m_{V}^{\left(  i\right)  },k_{\rho}^{\left(
i\right)  }}^{\left(  i\right)  }=\left\langle K;\mathsf{V}^{\left(  i\right)
}\mid\mathbf{\sigma}_{m_{K}},\mathbf{\kappa}_{n_{K}};\mathbf{\nu}%
_{m_{V}^{\left(  i\right)  }}^{\left(  i\right)  }\mid\mathbf{\rho}_{k_{\rho
}^{\left(  i\right)  }}^{\left(  i\right)  }\right\rangle
,\ \ \ \ \ \ i=1,\ldots,D_{V}<\infty.
\end{equation}

Here we define a polyadic analog of a \textquotedblleft
linear\textquotedblright\ mapping for polyadic vector spaces which
\textquotedblleft commutes\textquotedblleft\ with the \textquotedblleft vector
addition\textquotedblright\ and the \textquotedblleft multiplication by
scalar\textquotedblright\ (an analog of the additivity $\mathbf{F}\left(
\mathsf{v}+\mathsf{u}\right)  =\mathbf{F}\left(  \mathsf{v}\right)
+\mathbf{F}\left(  \mathsf{u}\right)  $, and the homogeneity of degree one
$\mathbf{F}\left(  \lambda\mathsf{v}\right)  =\lambda\mathbf{F}\left(
\mathsf{v}\right)  $, $\mathsf{v},\mathsf{u}\in\mathsf{V}$, $\lambda\in K$).

\begin{definition}
\label{def-lin1}A \textit{1-place (\textquotedblleft}$\mathbb{K}%
$\textit{-linear\textquotedblright)\ mapping} between polyadic vector spaces
$\mathcal{V}_{m_{K},n_{K},m_{V},k_{\rho}}=\left\langle K;\mathsf{V}%
\mid\mathbf{\sigma}_{m_{K}},\mathbf{\kappa}_{n_{K}};\mathbf{\nu}_{m_{V}}%
\mid\mathbf{\rho}_{k_{\rho}}\right\rangle $ and $\mathcal{V}_{m_{K}%
,n_{K},m_{V},k_{\rho}}=\left\langle K;\mathsf{V}^{\prime}\mid\mathbf{\sigma
}_{m_{K}},\mathbf{\kappa}_{n_{K}};\mathbf{\nu}_{m_{V}}^{\prime}\mid
\mathbf{\rho}_{k_{\rho}}^{\prime}\right\rangle $ over the same polyadic field
$\mathbb{K}_{m_{K},n_{K}}=\left\langle K\mid\mathbf{\sigma}_{m_{K}%
},\mathbf{\kappa}_{n_{K}}\right\rangle $ is $\mathbf{F}_{1}:\mathsf{V}%
\rightarrow\mathsf{V}^{\prime}$, such that%
\begin{align}
\mathbf{F}_{1}\left(  \mathbf{\nu}_{m_{V}}\left[  \mathsf{v}_{1}%
,\ldots,\mathsf{v}_{m_{V}}\right]  \right)   &  =\mathbf{\nu}_{m_{V}}^{\prime
}\left[  \mathbf{F}_{1}\left(  \mathsf{v}_{1}\right)  ,\ldots,\mathbf{F}%
_{1}\left(  \mathsf{v}_{m_{V}}\right)  \right]  ,\label{fv}\\
\mathbf{F}_{1}\left(  \mathbf{\rho}_{k_{\rho}}\left\{  \left.
\begin{array}
[c]{c}%
\lambda_{1}\\
\vdots\\
\lambda_{k_{\rho}}%
\end{array}
\right\vert \mathsf{v}\right\}  \right)   &  =\mathbf{\rho}_{k_{\rho}}%
^{\prime}\left\{  \left.
\begin{array}
[c]{c}%
\lambda_{1}\\
\vdots\\
\lambda_{k_{\rho}}%
\end{array}
\right\vert \mathbf{F}_{1}\left(  \mathsf{v}\right)  \right\}  , \label{fr}%
\end{align}
where $\mathsf{v}_{1},\ldots,\mathsf{v}_{m_{V}},\mathsf{v}\in\mathsf{V}$,
$\lambda_{1},\ldots,\lambda_{k_{\rho}}\in K$.
\end{definition}

If $\mathsf{z}_{V}$ is a \textquotedblleft zero vector\textquotedblright\ in
$\mathsf{V}$ and $\mathsf{z}_{V^{\prime}}$ is a \textquotedblleft zero
vector\textquotedblright\ in $\mathsf{V}^{\prime}$ (see (\ref{z})), then it
follows from (\ref{fv})--(\ref{fr}), that $\mathbf{F}_{1}\left(
\mathsf{z}_{V}\right)  =\mathsf{z}_{V^{\prime}}$.

The initial and final arities of $\mathbf{\nu}_{m_{V}}$ (\textquotedblleft
vector addition\textquotedblright) and the multiaction $\mathbf{\rho}%
_{k_{\rho}}$ (\textquotedblleft multiplication by scalar\textquotedblright)
coincide, because $\mathbf{F}_{1}$ is a 1-place mapping (a linear
homomorphism). In \cite{dup2012} \textit{multiplace mappings} and
corresponding \textit{heteromorphisms} were introduced. The latter allow us to
change arities ($m_{V}\rightarrow m_{V}^{\prime}$, $k_{\rho}\rightarrow
k_{\rho}^{\prime}$), which is the main difference between binary and polyadic mappings.

\begin{definition}
\label{def-link}A $k_{F}$\textit{-place (\textquotedblleft}$\mathbb{K}%
$\textit{-linear\textquotedblright)\ mapping} between two polyadic vector
spaces $\mathcal{V}_{m_{K},n_{K},m_{V},k_{\rho}}=\left\langle K;\mathsf{V}%
\mid\mathbf{\sigma}_{m_{K}},\mathbf{\kappa}_{n_{K}};\mathbf{\nu}_{m_{V}}%
\mid\mathbf{\rho}_{k_{\rho}}\right\rangle $ and $\mathcal{V}_{m_{K}%
,n_{K},m_{V},k_{\rho}}=\left\langle K;\mathsf{V}^{\prime}\mid\mathbf{\sigma
}_{m_{K}},\mathbf{\kappa}_{n_{K}};\mathbf{\nu}_{m_{V}^{\prime}}^{\prime}%
\mid\mathbf{\rho}_{k_{\rho}^{\prime}}^{\prime}\right\rangle $ over the same
polyadic field $\mathbb{K}_{m_{K},n_{K}}=\left\langle K\mid\mathbf{\sigma
}_{m_{K}},\mathbf{\kappa}_{n_{K}}\right\rangle $ is defined, if there exists
$\mathbf{F}_{k_{F}}:\mathsf{V}^{\times k_{F}}\rightarrow\mathsf{V}^{\prime}$,
such that%
\begin{equation}
\mathbf{F}_{k_{F}}\left(
\genfrac{}{}{0pt}{}{\left.
\begin{array}
[c]{c}%
\mathbf{\nu}_{m_{V}}\left[  \mathsf{v}_{1},\ldots,\mathsf{v}_{m_{V}}\right] \\
\vdots\\
\mathbf{\nu}_{m_{V}}\left[  \mathsf{v}_{m_{V}\left(  \ell_{\mu}^{k}-1\right)
},\ldots\mathsf{v}_{m_{V}\ell_{\mu}^{k}}\right]
\end{array}
\right\}  \ell_{\mu}^{k}}{\left.
\begin{array}
[c]{c}%
\mathsf{v}_{m_{V}\ell_{\mu}^{k}+1},\\
\vdots\\
\mathsf{v}_{m_{V}\ell_{\mu}^{k}+\ell_{\operatorname*{id}}^{k}}%
\end{array}
\right\}  \ell_{\operatorname*{id}}^{k}}%
\right)  =\mathbf{\nu}_{m_{V}^{\prime}}^{\prime}\left[  \mathbf{F}_{k_{F}%
}\left(
\begin{array}
[c]{c}%
\mathsf{v}_{1}\\
\vdots\\
\mathsf{v}_{k_{F}}%
\end{array}
\right)  ,\ldots,\mathbf{F}_{k_{F}}\left(
\begin{array}
[c]{c}%
\mathsf{v}_{k_{F}\left(  m_{V}^{\prime}-1\right)  }\\
\vdots\\
\mathsf{v}_{k_{F}m_{V}^{\prime}}%
\end{array}
\right)  \right]  , \label{fv1}%
\end{equation}%
\begin{equation}
\mathbf{F}_{k_{F}}\left(
\begin{array}
[c]{c}%
\left.
\begin{array}
[c]{c}%
\mathbf{\rho}_{k_{\rho}}\left\{  \left.
\begin{array}
[c]{c}%
\lambda_{1}\\
\vdots\\
\lambda_{k_{\rho}}%
\end{array}
\right\vert \mathsf{v}_{1}\right\} \\
\vdots\\
\mathbf{\rho}_{k_{\rho}}\left\{  \left.
\begin{array}
[c]{c}%
\lambda_{k_{\rho}\left(  \ell_{\mu}^{f}-1\right)  }\\
\vdots\\
\lambda_{k_{\rho}\ell_{\mu}^{f}}%
\end{array}
\right\vert \mathsf{v}_{\ell_{\mu}^{f}}\right\}
\end{array}
\right\}  \ell_{\mu}^{f}\\
\left.
\begin{array}
[c]{c}%
\mathsf{v}_{\ell_{\mu}^{f}+1}\\
\vdots\\
\mathsf{v}_{k_{F}}%
\end{array}
\right\}  \ell_{\operatorname*{id}}^{f}%
\end{array}
\right)  =\mathbf{\rho}_{k_{\rho}^{\prime}}^{\prime}\left\{  \left.
\begin{array}
[c]{c}%
\lambda_{1}\\
\vdots\\
\lambda_{k_{\rho}^{\prime}}%
\end{array}
\right\vert \mathbf{F}_{k_{F}}\left(
\begin{array}
[c]{c}%
\mathsf{v}_{1}\\
\vdots\\
\mathsf{v}_{k_{F}}%
\end{array}
\right)  \right\}  , \label{fr1}%
\end{equation}
where $\mathsf{v}_{1},\ldots,\mathsf{v}_{m_{V}},\mathsf{v}\in\mathsf{V}$,
$\lambda_{1},\ldots,\lambda_{k_{\rho}}\in K$, and the four integers
$\ell_{\rho}^{k}$, $\ell_{\operatorname*{id}}^{k}$, $\ell_{\rho}^{f}$,
$\ell_{\operatorname*{id}}^{f}$ define the $\ell$-shape of the mapping.
\end{definition}

It follows from (\ref{fv1})--(\ref{fr1}), that the arities satisfy%
\begin{equation}
k_{F}m_{V}^{\prime}=m_{V}\ell_{\mu}^{k}+\ell_{\operatorname*{id}}%
^{k},\ \ \ k_{F}=\ell_{\mu}^{k}+\ell_{\operatorname*{id}}^{k},\ \ \ k_{F}%
=\ell_{\mu}^{f}+\ell_{\operatorname*{id}}^{f},\ \ \ k_{\mu}^{\prime}=k_{\rho
}\ell_{\mu}^{f}.
\end{equation}

The following inequalities hold%
\begin{equation}
1\leq\ell_{\mu}^{k}\leq k_{F},\ 0\leq\ell_{\operatorname*{id}}^{k}\leq
k_{F}-1,\ \ell_{\mu}^{k}\leq k_{F}\leq\left(  m_{V}-1\right)  \ell_{\mu}%
^{k},\ 2\leq m_{V}^{\prime}\leq m_{V},\ 2\leq k_{\rho}\leq k_{\rho}^{\prime}.
\end{equation}

Thus, the $\ell$-shape of the $k_{F}$-place mapping between polyadic vector
spaces is determined by%
\begin{equation}
\ell_{\mu}^{k}=\dfrac{k_{F}\left(  m_{V}-1\right)  }{m_{V}-1},\ \ \ell
_{\operatorname*{id}}^{k}=\dfrac{k_{F}\left(  m_{V}-m_{V}^{\prime}\right)
}{m_{V}-1},\ \ \ell_{\mu}^{f}=\dfrac{k_{\rho}}{k_{\rho}^{\prime}}%
,\ \ \ell_{\operatorname*{id}}^{f}=k_{F}-\dfrac{k_{\rho}}{k_{\rho}^{\prime}}.
\end{equation}

\subsection{\label{sub-polfun}Polyadic functionals and dual polyadic vector
spaces}

An important particular case of the $k_{F}$-place mapping can be considered,
where the final polyadic vector space coincides with the underlying field
(analog of a \textquotedblleft linear functional\textquotedblright).

\begin{definition}
\label{def-linh}A\textit{ }\textquotedblleft linear\textquotedblright\textit{
polyadic functional} (or \textit{polyadic dual vector}, \textit{polyadic
covector}) is a $k_{L}$-place mapping of a polyadic vector space
$\mathcal{V}_{m_{K},n_{K},m_{V},k_{\rho}}=\left\langle K;\mathsf{V}%
\mid\mathbf{\sigma}_{m_{K}},\mathbf{\kappa}_{n_{K}};\mathbf{\nu}_{m_{V}}%
\mid\mathbf{\rho}_{k_{\rho}}\right\rangle $ into its polyadic field
$\mathbb{K}_{m_{K},n_{K}}=\left\langle K\mid\mathbf{\sigma}_{m_{K}%
},\mathbf{\kappa}_{n_{K}}\right\rangle $, such that there exists \textbf{$L$%
}$_{k_{L}}:\mathsf{V}^{\times k_{L}}\rightarrow K$, and%
\begin{equation}
\mathbf{L}_{k_{L}}\left(
\genfrac{}{}{0pt}{}{\left.
\begin{array}
[c]{c}%
\mathbf{\nu}_{m_{V}}\left[  \mathsf{v}_{1},\ldots,\mathsf{v}_{m_{V}}\right] \\
\vdots\\
\mathbf{\nu}_{m_{V}}\left[  \mathsf{v}_{m_{V}\left(  \ell_{\nu}^{k}-1\right)
},\ldots\mathsf{v}_{m_{V}\ell_{\nu}^{k}}\right]
\end{array}
\right\}  \ell_{\nu}^{k}}{\left.
\begin{array}
[c]{c}%
\mathsf{v}_{m_{V}\ell_{\nu}^{k}+1},\\
\vdots\\
\mathsf{v}_{n_{K}\ell_{\nu}^{k}+\ell_{\operatorname*{id}}^{\nu}}%
\end{array}
\right\}  \ell_{\operatorname*{id}}^{\nu}}%
\right)  =\mathbf{\sigma}_{m_{K}}\left[  \mathbf{L}_{k_{L}}\left(
\begin{array}
[c]{c}%
\mathsf{v}_{1}\\
\vdots\\
\mathsf{v}_{k_{L}}%
\end{array}
\right)  ,\ldots,\mathbf{L}_{k_{L}}\left(
\begin{array}
[c]{c}%
\mathsf{v}_{k_{L}\left(  m_{K}-1\right)  }\\
\vdots\\
\mathsf{v}_{k_{L}m_{K}}%
\end{array}
\right)  \right]  , \label{h1}%
\end{equation}%
\begin{equation}
\mathbf{L}_{k_{L}}\left(
\begin{array}
[c]{c}%
\left.
\begin{array}
[c]{c}%
\mathbf{\rho}_{k_{\rho}}\left\{  \left.
\begin{array}
[c]{c}%
\lambda_{1}\\
\vdots\\
\lambda_{k_{\rho}}%
\end{array}
\right\vert \mathsf{v}_{1}\right\} \\
\vdots\\
\mathbf{\rho}_{k_{\rho}}\left\{  \left.
\begin{array}
[c]{c}%
\lambda_{k_{\rho}\left(  \ell_{\mu}^{L}-1\right)  }\\
\vdots\\
\lambda_{k_{\rho}\ell_{\mu}^{L}}%
\end{array}
\right\vert \mathsf{v}_{\ell_{\mu}^{L}}\right\}
\end{array}
\right\}  \ell_{\mu}^{L}\\
\left.
\begin{array}
[c]{c}%
\mathsf{v}_{\ell_{\mu}^{L}+1}\\
\vdots\\
\mathsf{v}_{k_{L}}%
\end{array}
\right\}  \ell_{\operatorname*{id}}^{L}%
\end{array}
\right)  =\mathbf{\kappa}_{n_{K}}\left[  \lambda_{1},\ldots,\lambda_{n_{K}%
-1},\mathbf{L}_{k_{L}}\left(
\begin{array}
[c]{c}%
\mathsf{v}_{1}\\
\vdots\\
\mathsf{v}_{k_{L}}%
\end{array}
\right)  \right]  , \label{h2}%
\end{equation}
where $\mathsf{v}_{1},\ldots,\mathsf{v}_{m_{V}},\mathsf{v}\in V$, $\lambda
_{1},\ldots,\lambda_{n_{K}}\in K$, and the integers $\ell_{\nu}^{k}$,
$\ell_{\operatorname*{id}}^{\nu}$, $\ell_{\mu}^{L}$, $\ell_{\operatorname*{id}%
}^{L}$ define the $\ell$-shape of \textbf{$L$}$_{k_{L}}$.
\end{definition}

It follows from (\ref{fv1})--(\ref{fr1}), that the arities satisfy%
\begin{equation}
k_{L}m_{K}=m_{V}\ell_{\nu}^{k}+\ell_{\operatorname*{id}}^{\nu},\ \ \ k_{L}%
=\ell_{\nu}^{k}+\ell_{\operatorname*{id}}^{\nu},\ \ \ k_{L}=\ell_{\mu}%
^{h}+\ell_{\operatorname*{id}}^{h},\ \ \ n_{K}-1=k_{\rho}\ell_{\mu}^{h},
\end{equation}
and for them%
\begin{equation}
1\leq\ell_{\nu}^{k}\leq k_{L},\ 0\leq\ell_{\operatorname*{id}}^{\nu}\leq
k_{L}-1,\ \ell_{\nu}^{k}\leq k_{L}\leq\left(  m_{V}-1\right)  \ell_{\nu}%
^{k},\ 2\leq m_{K}\leq m_{V},\ 2\leq k_{\rho}\leq n_{K}-1.
\end{equation}

Thus, the $\ell$-shape of the polyadic functional is determined by%
\begin{equation}
\ell_{\nu}^{k}=\dfrac{k_{L}\left(  m_{K}-1\right)  }{m_{V}-1},\ \ \ell
_{\operatorname*{id}}^{\nu}=\dfrac{k_{L}\left(  m_{V}-m_{K}\right)  }{m_{V}%
-1},\ \ \ell_{\mu}^{h}=\dfrac{k_{\rho}}{n_{K}-1},\ \ \ell_{\operatorname*{id}%
}^{h}=k_{L}-\dfrac{k_{\rho}}{n_{K}-1}.
\end{equation}

In the binary case, because the dual vectors (linear functionals) take their
values in the underlying field, new operations between them, such that the
dual vector \textquotedblleft addition\textquotedblright\ $\left(  +^{\ast
}\right)  $ and the \textquotedblleft multiplication by a
scalar\textquotedblright\ $\left(  \bullet^{\ast}\right)  $ can be naturally
introduced by $\left(  L^{\left(  1\right)  }+^{\ast}L^{\left(  2\right)
}\right)  \left(  \mathsf{v}\right)  =L^{\left(  1\right)  }\left(
\mathsf{v}\right)  +L^{\left(  2\right)  }\left(  \mathsf{v}\right)  $,
$\left(  \lambda\bullet^{\ast}L\right)  \left(  \mathsf{v}\right)
=\lambda\bullet L\left(  \mathsf{v}\right)  $, which leads to another vector
space structure, called a dual vector space. Note that operations $+^{\ast}$
and $+$, $\bullet^{\ast}$ and $\bullet$ are different, because $+$ and
$\bullet$ are the operations in the underlying field $\mathbb{K}$. In the
polyadic case, we have more complicated arity changing formulas, and here we
consider finite-dimensional spaces only. The arities of operations between
dual vectors can be different from ones in the underlying polyadic field
$\mathbb{K}_{m_{K}n_{K}}$, in general. In this way, we arrive to the following

\begin{definition}
A \textit{polyadic dual vector space} over a polyadic field $\mathbb{K}%
_{m_{K},n_{K}}$ is%
\begin{equation}
\mathcal{V}_{m_{K},n_{K},m_{V}^{\ast},k_{\rho}^{\ast}}^{\ast}=\left\langle
K;\left\{  \mathbf{L}_{k_{L}}^{\left(  i\right)  }\right\}  \mid
\mathbf{\sigma}_{m_{K}},\mathbf{\kappa}_{n_{K}};\mathbf{\nu}_{m_{L}}^{\ast
}\mid\mathbf{\rho}_{k_{L}}^{\ast}\right\rangle ,
\end{equation}
and the axioms are:

1) $\left\langle K\mid\mathbf{\sigma}_{m_{K}},\mathbf{\kappa}_{n_{K}%
}\right\rangle $ is a polyadic $\left(  m_{K},n_{K}\right)  $-field
$\mathbb{K}_{m_{K},n_{K}}$;

2) $\left\langle \left\{  \mathbf{L}_{k_{L}}^{\left(  i\right)  }\right\}
\mid\mathbf{\nu}_{m_{L}}^{\ast},i=1,\ldots,D_{L}\right\rangle $ is a
commutative $m_{L}$-ary group (which is finite, if $D_{L}<\infty$);

3) The \textquotedblleft dual vector addition\textquotedblright\ $\mathbf{\nu
}_{m_{L}}^{\ast}$ is compatible with the polyadic field addition
$\mathbf{\sigma}_{m_{K}}$ by%
\begin{equation}
\mathbf{\nu}_{m_{L}}^{\ast}\left[  \mathbf{L}_{k_{L}}^{\left(  1\right)
},\ldots,\mathbf{L}_{k_{L}}^{\left(  m_{L}\right)  }\right]  \left(
\mathbf{a}^{\left(  k_{L}\right)  }\right)  =\mathbf{\sigma}_{m_{K}}\left[
\mathbf{L}_{k_{L}}^{\left(  1\right)  }\left(  \mathbf{a}^{\left(
k_{L}\right)  }\right)  ,\ldots,\mathbf{L}_{k_{L}}^{\left(  m_{K}\right)
}\left(  \mathbf{v}^{\left(  k_{L}\right)  }\right)  \right]  ,
\end{equation}
where $\mathbf{v}^{\left(  k_{L}\right)  }=\left(
\begin{array}
[c]{c}%
\mathsf{v}_{1}\\
\vdots\\
\mathsf{v}_{k_{L}}%
\end{array}
\right)  $, $\mathsf{v}_{1},\ldots,\mathsf{v}_{k_{L}}\in\mathsf{V}$, and it
follows that%
\begin{equation}
m_{L}=m_{K}.
\end{equation}

4) The compatibility of $\mathbf{\rho}_{k_{L}}^{\ast}$ with the
\textquotedblleft multiplication by a scalar\textquotedblright\ in the
underlying polyadic field%
\begin{equation}
\mathbf{\rho}_{k_{L}}^{\ast}\left\{  \left.
\begin{array}
[c]{c}%
\lambda_{1}\\
\vdots\\
\lambda_{k_{L}}%
\end{array}
\right\vert \mathbf{L}_{k_{L}}\right\}  \left(  \mathbf{v}^{\left(
k_{L}\right)  }\right)  =\mathbf{\kappa}_{n_{K}}\left[  \lambda_{1}%
,\ldots,\lambda_{n_{K}-1},\mathbf{L}_{k_{L}}\left(  \mathbf{v}^{\left(
k_{L}\right)  }\right)  \right]  ,
\end{equation}
and then%
\begin{equation}
k_{L}=n_{K}-1 \label{kn}%
\end{equation}

5) $\left\langle \left\{  \mathbf{\rho}_{k_{L}}^{\ast}\right\}  \mid
composition\right\rangle $ is a $n_{L}$-ary semigroup $\mathcal{S}_{L}$
(similar to (\ref{rrk}))%
\begin{align}
&  \mathbf{\rho}_{k_{L}}^{\ast}\overset{n_{L}}{\overbrace{\left\{  \left.
\begin{array}
[c]{c}%
\lambda_{1}\\
\vdots\\
\lambda_{k_{L}}%
\end{array}
\right\vert \left.
\begin{array}
[c]{c}%
\ \\
\ldots\\
\
\end{array}
\right\vert \mathbf{\rho}_{k_{L}}^{\ast}\left\{  \left.
\begin{array}
[c]{c}%
\lambda_{k_{L}\left(  n_{L}-1\right)  }\\
\vdots\\
\lambda_{k_{L}n_{L}}%
\end{array}
\right\vert \mathbf{L}_{k_{L}}\right\}  \ldots\right\}  }}\left(
\mathbf{v}^{\left(  k_{L}\right)  }\right) \\
&  =\mathbf{\rho}_{k_{L}}^{\ast}\left\{  \left.
\genfrac{}{}{0pt}{}{\left.
\begin{array}
[c]{c}%
\mathbf{\kappa}_{n_{K}}\left[  \lambda_{1},\ldots\lambda_{n_{K}}\right]  ,\\
\vdots\\
\mathbf{\kappa}_{n_{K}}\left[  \lambda_{n_{K}\left(  \ell_{\mu}^{L}-1\right)
},\ldots\lambda_{n_{K}\ell_{\mu}^{L}}\right]
\end{array}
\right\}  \ell_{\mu}^{L}}{\left.
\begin{array}
[c]{c}%
\lambda_{n_{K}\ell_{\mu}^{L}+1},\\
\vdots\\
\lambda_{n_{K}\ell_{\mu}^{L}+\ell_{\operatorname*{id}}^{L}}%
\end{array}
\right\}  \ell_{\operatorname*{id}}^{L}}%
\right\vert \mathbf{L}_{k_{L}}\right\}  \left(  \mathbf{v}^{\left(
k_{L}\right)  }\right)  ,
\end{align}
where the $\ell$-shape is determined by the system%
\begin{equation}
k_{L}n_{L}=n_{K}\ell_{\mu}^{L}+\ell_{\operatorname*{id}}^{L},\ \ \ \ \ k_{L}%
=\ell_{\mu}^{L}+\ell_{\operatorname*{id}}^{L}. \label{knn}%
\end{equation}

\end{definition}

Using (\ref{kn}) and (\ref{knn}), we obtain the $\ell$-shape as%
\begin{equation}
\ell_{\mu}^{L}=n_{L}-1,\ \ \ \ \ell_{\operatorname*{id}}^{L}=n_{K}-n_{L}.
\label{ln}%
\end{equation}

\begin{corollary}
The arity $n_{L}$ of the semigroup $\mathcal{S}_{L}$ is less than or equal to
the arity $n_{K}$ of the underlying polyadic field $n_{L}\leq n_{K}$.
\end{corollary}

\subsection{Polyadic direct sum and tensor product}

The Cartesian product of $D_{V}$ polyadic vector spaces $\times\Pi
_{i=1}^{m_{V}}\mathcal{V}_{m_{K}n_{K}m_{V}^{\left(  i\right)  }k_{\rho
}^{\left(  i\right)  }}^{\left(  i\right)  }$(sometimes we use the concise
notation $\times\Pi\mathcal{V}^{\left(  i\right)  }$), $i=1,\ldots,D_{V}$ is
given by the $D_{V}$-ples (an analog of the Cartesian pair $\left(
\mathsf{v},\mathsf{u}\right)  $, $\mathsf{v}\in\mathsf{V}^{\left(  1\right)
}$, $\mathsf{u}\in\mathsf{V}^{\left(  2\right)  }$)%
\begin{equation}
\left(
\begin{array}
[c]{c}%
\mathsf{v}^{\left(  1\right)  }\\
\vdots\\
\mathsf{v}^{\left(  D_{V}\right)  }%
\end{array}
\right)  \equiv\left(  \mathbf{\mathsf{v}}^{\left(  D_{V}\right)  }\right)
\in\mathsf{V}^{\times D_{V}}. \label{a}%
\end{equation}

We introduce a polyadic generalization of the direct sum and tensor product of
vector spaces by considering \textquotedblleft linear\textquotedblright%
\ operations on the $D_{V}$-ples (\ref{a}).

In the first case, to endow $\times\Pi\mathcal{V}^{\left(  i\right)  }$ with
the structure of a vector space we need to define a new operation between the
$D_{V}$-ples (\ref{a}) (similar to vector addition, but between elements from
\textsl{different} spaces) and a rule, specifying how they are
\textquotedblleft multiplied by scalars\textquotedblright\ (analogs of
$\left(  \mathsf{v}_{1},\mathsf{v}_{2}\right)  +\left(  \mathsf{u}%
_{1},\mathsf{u}_{2}\right)  =\left(  \mathsf{v}_{1}+\mathsf{u}_{1}%
,\mathsf{v}_{2}+\mathsf{u}_{2}\right)  $ and $\lambda\left(  \mathsf{v}%
_{1},\mathsf{v}_{2}\right)  =\left(  \lambda\mathsf{v}_{1},\lambda
\mathsf{v}_{2}\right)  $ ). In the binary case, a formal summation is used,
but it can be different from the addition in the initial vector spaces.
Therefore, we can define on the set of the $D_{V}$-ples (\ref{a}) new
operations $\mathbf{\chi}_{m_{V}}$ (\textquotedblleft addition of vectors from
\textsl{different spaces}\textquotedblright) and \textquotedblleft
multiplication by a scalar\textquotedblright\ $\mathbf{\tau}_{k_{\rho}}$,
which \textsl{does not} need to coincide with the corresponding
operations\ $\mathbf{\nu}_{m_{V}^{\left(  i\right)  }}^{\left(  i\right)  }$
and $\mathbf{\rho}_{k_{\rho}^{\left(  i\right)  }}^{\left(  i\right)  }$ of
the initial polyadic vector spaces $\mathcal{V}_{m_{K},n_{K},m_{V}^{\left(
i\right)  },k_{\rho}^{\left(  i\right)  }}^{\left(  i\right)  }$.

If all $D_{V}$-ples (\ref{a}) are of fixed length, then we can define their
\textquotedblleft addition\textquotedblright\ $\mathbf{\chi}_{m_{V}}$ in the
standard way, if all the arities $m_{V}^{\left(  i\right)  }$ coincide and
equal the arity of the resulting vector space%
\begin{equation}
m_{V}=m_{V}^{\left(  1\right)  }=\ldots=m_{V}^{\left(  D_{V}\right)  },
\label{ma}%
\end{equation}
while the operations (\textquotedblleft additions\textquotedblright)
themselves $\mathbf{\nu}_{m_{V}}^{\left(  i\right)  }$ between vectors in
different spaces can be still different. Thus, a new commutative $m_{V}$-ary
operation (\textquotedblleft addition\textquotedblright) $\mathbf{\chi}%
_{m_{V}}$ of the $D_{V}$-ples of the same length is defined by%
\begin{equation}
\mathbf{\chi}_{m_{V}}\left[  \left(
\begin{array}
[c]{c}%
\mathsf{v}_{1}^{\left(  1\right)  }\\
\vdots\\
\mathsf{v}_{1}^{\left(  D_{V}\right)  }%
\end{array}
\right)  ,\ldots,\left(
\begin{array}
[c]{c}%
\mathsf{v}_{m_{V}}^{\left(  1\right)  }\\
\vdots\\
\mathsf{v}_{m_{V}}^{\left(  D_{V}\right)  }%
\end{array}
\right)  \right]  =\left(
\begin{array}
[c]{c}%
\mathbf{\nu}_{m_{V}}^{\left(  1\right)  }\left[  \mathsf{v}_{1}^{\left(
1\right)  },\ldots,\mathsf{v}_{m_{V}}^{\left(  1\right)  }\right] \\
\vdots\\
\mathbf{\nu}_{m_{V}}^{\left(  D_{V}\right)  }\left[  \mathsf{v}_{1}^{\left(
D_{V}\right)  },\ldots,\mathsf{v}_{m_{V}}^{\left(  D_{V}\right)  }\right]
\end{array}
\right)  , \label{va}%
\end{equation}
where $D_{V}\neq m_{V}$, in general. However, it is also possible to add
$D_{V}$-ples of \textsl{different length} such that the operation (\ref{va})
is compatible with all arities $m_{V}^{\left(  i\right)  }$, $i=1,\ldots
,m_{V}$. For instance, if $m_{V}=3$, $m_{V}^{\left(  1\right)  }%
=m_{V}^{\left(  2\right)  }=3$, $m_{V}^{\left(  3\right)  }=2$, then%
\begin{equation}
\mathbf{\chi}_{3}\left[  \left(
\begin{array}
[c]{c}%
\mathsf{v}_{1}^{\left(  1\right)  }\\
\mathsf{v}_{1}^{\left(  2\right)  }\\
\mathsf{v}_{1}^{\left(  3\right)  }%
\end{array}
\right)  ,\left(
\begin{array}
[c]{c}%
\mathsf{v}_{2}^{\left(  1\right)  }\\
\mathsf{v}_{2}^{\left(  2\right)  }\\
\mathsf{v}_{2}^{\left(  3\right)  }%
\end{array}
\right)  ,\left(
\begin{array}
[c]{c}%
\mathsf{v}_{3}^{\left(  1\right)  }\\
\mathsf{v}_{3}^{\left(  2\right)  }\\
\
\end{array}
\right)  \right]  =\left(
\begin{array}
[c]{c}%
\nu_{3}^{\left(  1\right)  }\left[  \mathsf{v}_{1}^{\left(  1\right)
},\mathsf{v}_{2}^{\left(  1\right)  },\mathsf{v}_{3}^{\left(  1\right)
}\right] \\
\nu_{3}^{\left(  2\right)  }\left[  \mathsf{v}_{1}^{\left(  2\right)
},\mathsf{v}_{2}^{\left(  2\right)  },\mathsf{v}_{3}^{\left(  2\right)
}\right] \\
\nu_{2}^{\left(  3\right)  }\left[  \mathsf{v}_{1}^{\left(  3\right)
},\mathsf{v}_{2}^{\left(  3\right)  }\right]
\end{array}
\right)  . \label{vv}%
\end{equation}

\begin{assertion}
In the polyadic case, a direct sum of polyadic vector spaces having
\textit{different arities} of \textquotedblleft vector
addition\textquotedblright\ $m_{V}^{\left(  i\right)  }$ can be defined.
\end{assertion}

Let us introduce the multiaction\ $\mathbf{\tau}_{k_{\rho}}$
(\textquotedblleft multiplication by a scalar\textquotedblright) acting on
$D_{V}$-ple $\left(  \mathbf{\mathsf{v}}^{\left(  m_{V}\right)  }\right)  $,
then%
\begin{equation}
\mathbf{\rho}_{k_{\rho}}\left\{  \left.
\begin{array}
[c]{c}%
\lambda_{1}\\
\vdots\\
\lambda_{k_{\rho}}%
\end{array}
\right\vert \left(
\begin{array}
[c]{c}%
\mathsf{v}^{\left(  1\right)  }\\
\vdots\\
\mathsf{v}^{\left(  D_{V}\right)  }%
\end{array}
\right)  \right\}  =\left(
\begin{array}
[c]{c}%
\mathbf{\rho}_{k_{\rho}^{\left(  1\right)  }}^{\left(  1\right)  }\left\{
\left.
\begin{array}
[c]{c}%
\lambda_{1}\\
\vdots\\
\lambda_{k_{\rho}^{\left(  1\right)  }}%
\end{array}
\right\vert \mathsf{v}^{\left(  1\right)  }\right\} \\
\vdots\\
\mathbf{\rho}_{k_{\rho}^{\left(  D_{V}\right)  }}^{\left(  m_{V}\right)
}\left\{  \left.
\begin{array}
[c]{c}%
\lambda_{k_{\rho}^{\left(  1\right)  }+\ldots+k_{\rho}^{\left(  D_{V}%
-1\right)  }+1}\\
\vdots\\
\lambda_{k_{\rho}^{\left(  1\right)  }+\ldots+k_{\rho}^{\left(  D_{V}\right)
}}%
\end{array}
\right\vert \mathsf{v}^{\left(  D_{V}\right)  }\right\}
\end{array}
\right)  , \label{ra}%
\end{equation}
where
\begin{equation}
k_{\rho}^{\left(  1\right)  }+\ldots+k_{\rho}^{\left(  D_{V}\right)  }%
=k_{\rho}. \label{k}%
\end{equation}

\begin{definition}
A \textit{polyadic direct sum} of $m_{V}$ polyadic vector spaces is their
Cartesian product equipped with the $m_{V}$-ary addition $\mathbf{\chi}%
_{m_{V}}$ and the $k_{\rho}$-place multiaction $\mathbf{\tau}_{k_{\rho}}$,
satisfying (\ref{va}) and (\ref{ra}) respectively%
\begin{equation}
\oplus\Pi_{i=1}^{D_{V}}\mathcal{V}_{m_{K},n_{K},m_{V}^{\left(  i\right)
},k_{\rho}^{\left(  i\right)  }}^{\left(  i\right)  }=\left\{  \times\Pi
_{i=1}^{D_{V}}\mathcal{V}_{m_{K},n_{K},m_{V}^{\left(  i\right)  },k_{\rho
}^{\left(  i\right)  }}^{\left(  i\right)  }\mid\mathbf{\chi}_{m_{V}%
},\mathbf{\tau}_{k_{\rho}}\right\}  .
\end{equation}

\end{definition}

Let us consider another way to define a vector space structure on the $D_{V}%
$-ples from the Cartesian product $\times\Pi\mathcal{V}^{\left(  i\right)  }$.
Remember that in the binary case, the concept of bilinearity is used, which
means \textquotedblleft distributivity\textquotedblright\ and
\textquotedblleft multiplicativity by scalars\textquotedblright\ on
\textit{each place} separately in the Cartesian pair $\left(  \mathsf{v}%
_{1},\mathsf{v}_{2}\right)  \in\mathsf{V}^{\left(  1\right)  }\times
\mathsf{V}^{\left(  2\right)  }$ (\textit{as opposed to} the direct sum, where
these relations hold on all places simultaneously, see (\ref{va}) and
(\ref{ra})) such that%
\begin{align}
\left(  \mathsf{v}_{1}+\mathsf{u}_{1},\mathsf{v}_{2}\right)   &  =\left(
\mathsf{v}_{1},\mathsf{v}_{2}\right)  +\left(  \mathsf{u}_{1},\mathsf{v}%
_{2}\right)  ,\ \ \ \ \ \left(  \mathsf{v}_{1},\mathsf{v}_{2}+\mathsf{u}%
_{2}\right)  =\left(  \mathsf{v}_{1},\mathsf{v}_{2}\right)  +\left(
\mathsf{v}_{1},\mathsf{u}_{2}\right)  ,\label{a1}\\
\lambda\left(  \mathsf{v}_{1},\mathsf{v}_{2}\right)   &  =\left(
\lambda\mathsf{v}_{1},\mathsf{v}_{2}\right)  =\left(  \mathsf{v}_{1}%
,\lambda\mathsf{v}_{2}\right)  , \label{a2}%
\end{align}
respectively. If we denote the ideal corresponding to the relations
(\ref{a1})--(\ref{a2}) by $\mathfrak{B}_{2}$, then the binary tensor product
of the vector spaces can be defined as their Cartesian product by factoring
out this ideal, as $\mathcal{V}^{\left(  1\right)  }\otimes\mathcal{V}%
^{\left(  2\right)  }\overset{def}{=}\mathcal{V}^{\left(  1\right)  }%
\times\mathcal{V}^{\left(  2\right)  }\diagup\mathfrak{B}_{2}$. Note first,
that the additions and multiplications by a scalar on both sides of
(\ref{a1})--(\ref{a2}) \textquotedblleft work\textquotedblright\ in different
spaces, which sometimes can be concealed by adding the word \textquotedblleft
formal\textquotedblright\ to them. Second, all these operations have the
\textit{same} arity (binary ones), which need not be the case when considering
polyadic structures.

As in the case of the polyadic direct sum, we first define a new operation
$\mathbf{\tilde{\chi}}_{m_{V}}$ (\textquotedblleft addition\textquotedblright)
of the $D_{V}$-ples of fixed length (different from $\mathbf{\chi}_{m_{V}}$ in
(\ref{va})), when all the arities $m_{V}^{\left(  i\right)  }$ coincide and
equal to $m_{V}$ (\ref{ma}). Then, a straightforward generalization of
(\ref{a1}) can be defined for $m_{V}$-ples similar to the polyadic
distributivity (\ref{dis1})--(\ref{dis3}), as in the following $m_{V}$
relations%
\begin{equation}
\left(
\begin{array}
[c]{c}%
\mathbf{\nu}_{m_{V}}^{\left(  1\right)  }\left[  \mathsf{v}_{1}^{\left(
1\right)  },\ldots,\mathsf{v}_{m_{V}}^{\left(  1\right)  }\right] \\
\mathsf{u}_{2}\\
\vdots\\
\mathsf{u}_{D_{V}}%
\end{array}
\right)  =\mathbf{\tilde{\chi}}_{m_{V}}\left[  \left(
\begin{array}
[c]{c}%
\mathsf{v}_{1}^{\left(  1\right)  }\\
\mathsf{u}_{2}\\
\vdots\\
\mathsf{u}_{D_{V}}%
\end{array}
\right)  ,\ldots,\left(
\begin{array}
[c]{c}%
\mathsf{v}_{m_{V}}^{\left(  1\right)  }\\
\mathsf{u}_{2}\\
\vdots\\
\mathsf{u}_{D_{V}}%
\end{array}
\right)  \right]  , \label{v1}%
\end{equation}%
\begin{align}
\left(
\begin{array}
[c]{c}%
\mathsf{u}_{1}\\
\mathbf{\nu}_{m_{V}}^{\left(  2\right)  }\left[  \mathsf{v}_{1}^{\left(
2\right)  },\ldots,\mathsf{v}_{m_{V}}^{\left(  2\right)  }\right] \\
\vdots\\
\mathsf{u}_{m_{V}}%
\end{array}
\right)   &  =\mathbf{\tilde{\chi}}_{m_{V}}\left[  \left(
\begin{array}
[c]{c}%
\mathsf{u}_{1}\\
\mathsf{v}_{1}^{\left(  2\right)  }\\
\vdots\\
\mathsf{u}_{m_{V}}%
\end{array}
\right)  ,\ldots,\left(
\begin{array}
[c]{c}%
\mathsf{u}_{1}\\
\mathsf{v}_{m_{V}}^{\left(  2\right)  }\\
\vdots\\
\mathsf{u}_{m_{V}}%
\end{array}
\right)  \right]  ,\\
&  \vdots\nonumber
\end{align}%
\begin{equation}
\left(
\begin{array}
[c]{c}%
\mathsf{u}_{1}\\
\mathsf{u}_{2}\\
\vdots\\
\mathbf{\nu}_{m_{V}}^{\left(  m_{V}\right)  }\left[  \mathsf{v}_{1}^{\left(
D_{V}\right)  },\ldots,\mathsf{v}_{m_{V}}^{\left(  D_{V}\right)  }\right]
\end{array}
\right)  =\mathbf{\tilde{\chi}}_{m_{V}}\left[  \left(
\begin{array}
[c]{c}%
\mathsf{u}_{1}\\
\mathsf{u}_{2}\\
\vdots\\
\mathsf{v}_{1}^{\left(  D_{V}\right)  }%
\end{array}
\right)  ,\ldots,\left(
\begin{array}
[c]{c}%
\mathsf{u}_{1}\\
\mathsf{u}_{2}\\
\vdots\\
\mathsf{v}_{m_{V}}^{\left(  D_{V}\right)  }%
\end{array}
\right)  \right]  . \label{v3}%
\end{equation}

By analogy, if all $k_{\rho}^{\left(  i\right)  }$ are equal we can define a
new multiaction $\mathbf{\tilde{\tau}}_{k_{\rho}}$ (different from
$\mathbf{\tau}_{k_{\rho}}$ (\ref{ra})) with the same number of places%
\begin{equation}
k_{\rho}=k_{\rho}^{\left(  1\right)  }=\ldots=k_{\rho}^{\left(  D_{V}\right)
} \label{kk}%
\end{equation}
as the $D_{V}$ relations (an analog of (\ref{a2}))%
\begin{equation}
\mathbf{\rho}_{k_{\rho}}^{\prime}\left\{  \left.
\begin{array}
[c]{c}%
\lambda_{1}\\
\vdots\\
\lambda_{k_{\rho}}%
\end{array}
\right\vert \left(
\begin{array}
[c]{c}%
\mathsf{v}^{\left(  1\right)  }\\
\vdots\\
\mathsf{v}^{\left(  D_{V}\right)  }%
\end{array}
\right)  \right\}  =\left(
\begin{array}
[c]{c}%
\mathbf{\rho}_{k_{\rho}}^{\left(  1\right)  }\left\{  \left.
\begin{array}
[c]{c}%
\lambda_{1}\\
\vdots\\
\lambda_{k_{\rho}}%
\end{array}
\right\vert \mathsf{v}^{\left(  1\right)  }\right\} \\
\mathsf{v}^{\left(  2\right)  }\\
\vdots\\
\mathsf{v}^{\left(  D_{V}\right)  }%
\end{array}
\right)  \label{r1}%
\end{equation}%
\begin{align}
&  =\left(
\begin{array}
[c]{c}%
\mathsf{v}^{\left(  1\right)  }\\
\mathbf{\rho}_{k_{\rho}}^{\left(  2\right)  }\left\{  \left.
\begin{array}
[c]{c}%
\lambda_{1}\\
\vdots\\
\lambda_{k_{\rho}}%
\end{array}
\right\vert \mathsf{v}^{\left(  2\right)  }\right\} \\
\vdots\\
\mathsf{v}^{\left(  D_{V}\right)  }%
\end{array}
\right) \\
&  \vdots\nonumber
\end{align}%
\begin{equation}
=\left(
\begin{array}
[c]{c}%
\mathsf{v}^{\left(  1\right)  }\\
\mathsf{v}^{\left(  2\right)  }\\
\vdots\\
\mathbf{\rho}_{k_{\rho}}^{\left(  D_{V}\right)  }\left\{  \left.
\begin{array}
[c]{c}%
\lambda_{1}\\
\vdots\\
\lambda_{k_{\rho}}%
\end{array}
\right\vert \mathsf{v}^{\left(  D_{V}\right)  }\right\}
\end{array}
\right)  . \label{r3}%
\end{equation}

Let us denote the ideal corresponding to the relations (\ref{v1})--(\ref{v3}),
(\ref{r1})--(\ref{r3}) by $\mathfrak{B}_{D_{V}}$.

\begin{definition}
A \textit{polyadic tensor product} of $D_{V}$ polyadic vector spaces
$\mathcal{V}_{m_{K},n_{K},m_{V}^{\left(  i\right)  },k_{\rho}^{\left(
i\right)  }}^{\left(  i\right)  }$ is their Cartesian product equipped with
the $m_{V}$-ary addition $\mathbf{\tilde{\chi}}_{m_{V}}$ (of $D_{V}$-ples) and
the $k_{\rho}$-place multiaction $\mathbf{\tilde{\tau}}_{k_{\rho}}$,
satisfying (\ref{v1})--(\ref{v3}) and (\ref{r1})--(\ref{r3}), respectively, by
factoring out the ideal $\mathfrak{B}_{D_{V}}$%
\begin{equation}
\otimes\Pi_{i=1}^{m_{V}}\mathcal{V}_{m_{K},n_{K},m_{V}^{\left(  i\right)
},k_{\rho}^{\left(  i\right)  }}^{\left(  i\right)  }=\left\{  \times\Pi
_{i=1}^{m_{V}}\mathcal{V}_{m_{K},n_{K},m_{V}^{\left(  i\right)  },k_{\rho
}^{\left(  i\right)  }}^{\left(  i\right)  }\mid\mathbf{\tilde{\chi}}_{m_{V}%
},\mathbf{\tilde{\tau}}_{k_{\rho}}\right\}  \diagup\mathfrak{B}_{D_{V}}.
\end{equation}

\end{definition}

As in the case of the polyadic direct sum, we can consider the distributivity
for $D_{V}$-ples of different length. In a similar example (\ref{vv}), if
$m_{V}=3$, $m_{V}^{\left(  1\right)  }=m_{V}^{\left(  2\right)  }=3$,
$m_{V}^{\left(  3\right)  }=2$, we have%
\begin{equation}
\left(
\begin{array}
[c]{c}%
\mathbf{\nu}_{3}^{\left(  1\right)  }\left[  \mathsf{v}_{1}^{\left(  1\right)
},\mathsf{v}_{2}^{\left(  1\right)  },\mathsf{v}_{3}^{\left(  1\right)
}\right] \\
\mathsf{u}_{2}\\
\mathsf{u}_{3}%
\end{array}
\right)  =\mathbf{\tilde{\chi}}_{3}\left[  \left(
\begin{array}
[c]{c}%
\mathsf{v}_{1}^{\left(  1\right)  }\\
\mathsf{u}_{2}\\
\mathsf{u}_{3}%
\end{array}
\right)  ,\left(
\begin{array}
[c]{c}%
\mathsf{v}_{2}^{\left(  1\right)  }\\
\mathsf{u}_{2}\\
\mathsf{u}_{3}%
\end{array}
\right)  ,\left(
\begin{array}
[c]{c}%
\mathsf{v}_{3}^{\left(  1\right)  }\\
\mathsf{u}_{2}\\
\
\end{array}
\right)  \right]  ,
\end{equation}%
\begin{equation}
\left(
\begin{array}
[c]{c}%
\mathsf{u}_{1}\\
\mathbf{\nu}_{3}^{\left(  2\right)  }\left[  \mathsf{v}_{1}^{\left(  2\right)
},\mathsf{v}_{2}^{\left(  2\right)  },\mathsf{v}_{3}^{\left(  2\right)
}\right] \\
\mathsf{u}_{3}%
\end{array}
\right)  =\mathbf{\tilde{\chi}}_{3}\left[  \left(
\begin{array}
[c]{c}%
\mathsf{u}_{1}\\
\mathsf{v}_{1}^{\left(  2\right)  }\\
\mathsf{u}_{3}%
\end{array}
\right)  ,\left(
\begin{array}
[c]{c}%
\mathsf{u}_{1}\\
\mathsf{v}_{2}^{\left(  2\right)  }\\
\mathsf{u}_{3}%
\end{array}
\right)  ,\left(
\begin{array}
[c]{c}%
\mathsf{u}_{1}\\
\mathsf{v}_{3}^{\left(  2\right)  }\\
\
\end{array}
\right)  \right]  ,
\end{equation}%
\begin{equation}
\left(
\begin{array}
[c]{c}%
\mathsf{u}_{1}\\
\mathsf{u}_{2}\\
\mathbf{\nu}_{2}^{\left(  3\right)  }\left[  \mathsf{v}_{1}^{\left(  3\right)
},\mathsf{v}_{2}^{\left(  3\right)  }\right]
\end{array}
\right)  =\mathbf{\tilde{\chi}}_{3}\left[  \left(
\begin{array}
[c]{c}%
\mathsf{u}_{1}\\
\mathsf{u}_{2}\\
\mathsf{v}_{1}^{\left(  3\right)  }%
\end{array}
\right)  ,\left(
\begin{array}
[c]{c}%
\mathsf{u}_{1}\\
\mathsf{u}_{2}\\
\mathsf{v}_{2}^{\left(  3\right)  }%
\end{array}
\right)  ,\left(
\begin{array}
[c]{c}%
\mathsf{u}_{1}\\
\mathsf{u}_{2}\\
\
\end{array}
\right)  \right]  .
\end{equation}

\begin{assertion}
A tensor product of polyadic vector spaces having \textit{different arities}
of the \textquotedblleft vector addition\textquotedblright\ $m_{V}^{\left(
i\right)  }$ can be defined.
\end{assertion}

In the case of modules over a polyadic ring, the formulas connecting arities
and $\ell$-shapes similar to those above hold, while concrete properties
(noncommutativity, mediality, etc.) should be taken into account.

\section{Polyadic inner pairing spaces and norms}

Here we introduce the next important operation: a polyadic analog of the inner
product for polyadic vector spaces - a polyadic inner pairing\footnote{Note
that this concept is different from the $n$-inner product spaces considered in
[Misiak, et al].}. Let $\mathcal{V}_{m_{K},n_{K},m_{V},k_{\rho}}=\left\langle
K;\mathsf{V}\mid\mathbf{\sigma}_{m_{K}},\mathbf{\kappa}_{n_{K}};\mathbf{\nu
}_{m_{V}}\mid\mathbf{\rho}_{k_{\rho}}\right\rangle $ be a polyadic vector
space over the polyadic field $\mathbb{K}_{m_{K},n_{K}}$ (\ref{v}). By analogy
with the binary inner product, we introduce its polyadic counterpart and study
its arity shape.

\begin{definition}
\textit{A polyadic }$N$-place\textit{ inner pairing }(an analog of the inner
product) is a mapping%
\begin{equation}
\overset{N}{\overbrace{\left\langle \left\langle \mathsf{\bullet
}|\mathsf{\bullet}|\ldots|\mathsf{\bullet}\right\rangle \right\rangle }%
}:\mathsf{V}^{\times N}\rightarrow K, \label{aa}%
\end{equation}
satisfying the following conditions:

1) Polyadic \textquotedblleft linearity\textquotedblright\ (\ref{rrk}) (for
first argument):%
\begin{equation}
\left\langle \left\langle \mathbf{\rho}_{k_{\rho}}\left\{  \left.
\begin{array}
[c]{c}%
\lambda_{1}\\
\vdots\\
\lambda_{k_{\rho}}%
\end{array}
\right\vert \mathsf{v}_{1}\right\}  |\mathsf{v}_{2}|\ldots|\mathsf{v}%
_{N}\right\rangle \right\rangle =\mathbf{\kappa}_{n_{K}}\left[  \lambda
_{1},\ldots,\lambda_{k_{\rho}},\left\langle \left\langle \mathsf{v}%
_{1}|\mathsf{v}_{2}|\ldots|\mathsf{v}_{N}\right\rangle \right\rangle \right]
. \label{rl}%
\end{equation}

2) Polyadic \textquotedblleft distributivity\textquotedblright\ (\ref{dis1}%
)--(\ref{dis3}) (on each place):%
\begin{align}
&  \left\langle \left\langle \mathbf{\nu}_{m_{V}}\left[  \mathsf{v}%
_{1},\mathsf{u}_{1},\ldots\mathsf{u}_{m_{V}-1}\right]  |\mathsf{v}_{2}%
|\ldots|\mathsf{v}_{N}\right\rangle \right\rangle \nonumber\\
&  =\mathbf{\sigma}_{m_{K}}\left[  \left\langle \left\langle \mathsf{v}%
_{1}|\mathsf{v}_{2}|\ldots|\mathsf{v}_{N}\right\rangle \right\rangle
,\left\langle \left\langle \mathsf{u}_{1}|\mathsf{v}_{2}|\ldots|\mathsf{v}%
_{N}\right\rangle \right\rangle \ldots\left\langle \left\langle \mathsf{u}%
_{m_{V}-1}|\mathsf{v}_{2}|\ldots|\mathsf{v}_{N}\right\rangle \right\rangle
\right]  . \label{vab}%
\end{align}

If the polyadic field $\mathbb{K}_{m_{K},n_{K}}$ contains the zero $z_{K}$ and
$\left\langle \mathsf{V}\mid m_{V}\right\rangle $ has the zero
\textquotedblleft vector\textquotedblright\ $\mathsf{z}_{V}$\textbf{ }(which
is not always the case in the polyadic case), we have the additional axiom:

3) The polyadic inner pairing vanishes $\left\langle \left\langle
\mathsf{v}_{1}|\mathsf{v}_{2}|\ldots|\mathsf{v}_{N}\right\rangle \right\rangle
=z_{K}$, iff any of the \textquotedblleft vectors\textquotedblright\ vanishes,
$\exists i\in1,\ldots,N$, such that $\mathsf{v}_{i}=\mathsf{z}_{V}$.

If the standard binary ordering on $\mathbb{K}_{m_{K},n_{K}}$ can be defined,
then the polyadic inner pairing satisfies:

4) The positivity condition%
\begin{equation}
\overset{N}{\overbrace{\left\langle \left\langle \mathsf{v}|\mathsf{v}%
|\ldots|\mathsf{v}\right\rangle \right\rangle }}\geq z_{K},
\end{equation}

5) The polyadic \textit{Cauchy-Schwarz inequality (\textquotedblleft
triangle\textquotedblright\ inequality)}%
\begin{align}
&  \mathbf{\kappa}_{n_{K}}\left[  \overset{n_{K}}{\overbrace{\overset
{N}{\overbrace{\left\langle \left\langle \mathsf{v}_{1}|\mathsf{v}_{1}%
|\ldots|\mathsf{v}_{1}\right\rangle \right\rangle }},\overset{N}%
{\overbrace{\left\langle \left\langle \mathsf{v}_{2}|\mathsf{v}_{2}%
|\ldots|\mathsf{v}_{2}\right\rangle \right\rangle }}\ldots\overset
{N}{\overbrace{\left\langle \left\langle \mathsf{v}_{n_{K}}|\mathsf{v}_{n_{K}%
}|\ldots|\mathsf{v}_{n_{K}}\right\rangle \right\rangle }}}}\right] \nonumber\\
&  \geq\mathbf{\kappa}_{n_{K}}\left[  \overset{n_{K}}{\overbrace{\left\langle
\left\langle \mathsf{v}_{1}|\mathsf{v}_{2}|\ldots|\mathsf{v}_{N}\right\rangle
\right\rangle ,\left\langle \left\langle \mathsf{v}_{1}|\mathsf{v}_{2}%
|\ldots|\mathsf{v}_{N}\right\rangle \right\rangle \ldots,\left\langle
\left\langle \mathsf{v}_{1}|\mathsf{v}_{2}|\ldots|\mathsf{v}_{N}\right\rangle
\right\rangle }}\right]  . \label{ka}%
\end{align}

\end{definition}

To make the above relations consistent, the arity shapes should be fixed.

\begin{definition}
If the inner pairing is fully symmetric under permutations it is called a
\textit{polyadic inner product}.
\end{definition}

\begin{proposition}
The number of places in the multiaction $\mathbf{\rho}_{k_{\rho}}$ differs by
$1$ from the multiplication arity of the polyadic field%
\begin{equation}
n_{K}-k_{\rho}=1. \label{nkk}%
\end{equation}

\end{proposition}

\begin{proof}
It follows from the polyadic \textquotedblleft linearity\textquotedblright%
\ (\ref{rl}).
\end{proof}

\begin{proposition}
The arities of \textquotedblleft vector addition\textquotedblright\ and
\textquotedblleft field addition\textquotedblright\ coincide%
\begin{equation}
m_{V}=m_{K}. \label{mvm}%
\end{equation}

\end{proposition}

\begin{proof}
Implied by the polyadic \textquotedblleft distributivity\textquotedblright%
\ (\ref{vab}).
\end{proof}

\begin{proposition}
The arity of the \textquotedblleft field multiplication\textquotedblright\ is
equal to the arity of the polyadic inner pairing space%
\begin{equation}
n_{K}=N. \label{nkn}%
\end{equation}

\end{proposition}

\begin{proof}
This follows from the polyadic Cauchy-Schwarz inequality (\ref{ka}).
\end{proof}

\begin{definition}
\label{def-inpar}The polyadic vector space $\mathcal{V}_{m_{K},n_{K}%
,m_{V},k_{\rho}}$ equipped with the polyadic inner pairing

$\overset{N}{\overbrace{\left\langle \left\langle \mathsf{\bullet
}|\mathsf{\bullet}|\ldots|\mathsf{\bullet}\right\rangle \right\rangle }}$
$:\mathsf{V}^{\times N}\rightarrow K$ is called a \textit{polyadic inner
pairing space }$\mathcal{H}_{m_{K},n_{K},m_{V},k_{\rho},N}$.
\end{definition}

A polyadic analog of the binary norm $\left\Vert \bullet\right\Vert
:\mathsf{V}\rightarrow K$ can be induced by the inner pairing similarly to the
binary case for the inner product (we use the form $\left\Vert \mathsf{v}%
\right\Vert ^{2}=\left\langle \left\langle \mathsf{v}|\mathsf{v}\right\rangle
\right\rangle $).

\begin{definition}
\label{def-pnorm}A polyadic norm of a \textquotedblleft
vector\textquotedblright\ $\mathsf{v}$ in the polyadic inner pairing space
$\mathcal{H}_{m_{K},n_{K},m_{V},k_{\rho},N}$ is a mapping $\left\Vert
\bullet\right\Vert _{N}:\mathsf{V}\rightarrow K$, such that%
\begin{align}
\mathbf{\kappa}_{n_{K}}\left[  \overset{n_{K}}{\overbrace{\left\Vert
\mathsf{v}\right\Vert _{N},\left\Vert \mathsf{v}\right\Vert _{N}%
,\ldots,\left\Vert \mathsf{v}\right\Vert _{N}}}\right]   &  =\overset
{N}{\overbrace{\left\langle \left\langle \mathsf{v}|\mathsf{v}|\ldots
|\mathsf{v}\right\rangle \right\rangle }},\label{kv}\\
n_{K}  &  =N,
\end{align}
and the following axioms apply:

1) The polyadic \textquotedblleft linearity\textquotedblright%
\begin{align}
\left\Vert \mathbf{\rho}_{k_{\rho}}\left\{  \left.
\begin{array}
[c]{c}%
\lambda_{1}\\
\vdots\\
\lambda_{k_{\rho}}%
\end{array}
\right\vert \mathsf{v}\right\}  \right\Vert _{N}  &  =\mathbf{\kappa}_{n_{K}%
}\left[  \lambda_{1},\ldots,\lambda_{k_{\rho}},\left\Vert \mathsf{v}%
\right\Vert _{N}\right]  ,\\
n_{K}-k_{\rho}  &  =1.
\end{align}

If the polyadic field $\mathbb{K}_{m_{K},n_{K}}$ contains the zero $z_{K}$ and
$\left\langle \mathsf{V}\mid m_{V}\right\rangle $ has a zero \textquotedblleft
vector\textquotedblright\ $\mathsf{z}_{V}$, then:

2) The polyadic norm vanishes $\left\Vert \mathsf{v}\right\Vert _{N}=z_{K}$,
iff $\mathsf{v}=\mathsf{z}_{V}$.

If the binary ordering on $\left\langle \mathsf{V}\mid m_{V}\right\rangle $
can be defined, then:

3) The polyadic norm is positive $\left\Vert \mathsf{v}\right\Vert _{N}\geq
z_{K}$.

4) The polyadic\textquotedblleft triangle\textquotedblright\ inequality holds%
\begin{align}
\mathbf{\sigma}_{m_{K}}\left[  \overset{m_{K}}{\overbrace{\left\Vert
\mathsf{v}_{1}\right\Vert _{N},\left\Vert \mathsf{v}_{2}\right\Vert
_{N},\ldots,\left\Vert \mathsf{v}_{N}\right\Vert _{N}}}\right]   &
\geq\left\Vert \mathbf{\nu}_{m_{V}}\left[  \overset{m_{V}}{\overbrace
{\left\Vert \mathsf{v}_{1}\right\Vert _{N},\left\Vert \mathsf{v}%
_{2}\right\Vert _{N},\ldots,\left\Vert \mathsf{v}_{N}\right\Vert _{N}}%
}\right]  \right\Vert ,\\
m_{K}  &  =m_{V}=N.
\end{align}

\end{definition}

\begin{definition}
The polyadic inner pairing space $\mathcal{H}_{m_{K},n_{K},m_{V},k_{\rho},N}$
equipped with the polyadic norm $\left\Vert \mathsf{v}\right\Vert _{N}$ is
called a \textit{polyadic normed space}.
\end{definition}

Recall that in the binary vector space $\mathsf{V}$ over the field
$\mathbb{K}$ equipped with the inner product $\left\langle \left\langle
\mathsf{\bullet}|\mathsf{\bullet}\right\rangle \right\rangle $ and the norm
$\left\Vert \mathsf{\bullet}\right\Vert $, one can introduce the angle between
vectors $\left\Vert \mathsf{v}_{1}\right\Vert \cdot\left\Vert \mathsf{v}%
_{2}\right\Vert \cdot\cos\theta=\left\langle \left\langle \mathsf{v}%
_{1}|\mathsf{v}_{2}\right\rangle \right\rangle $, where on l.h.s. there are
\textsl{two} binary multiplications $\left(  \cdot\right)  $.

\begin{definition}
A \textit{polyadic angle} between $N$ vectors $\mathsf{v}_{1},\mathsf{v}%
_{2},\ldots,\mathsf{v}_{n_{K}}$ of the polyadic inner pairing space
$\mathcal{H}_{m_{K},n_{K},m_{V},k_{\rho},N}$ is defined as a \textit{set} of
angles $\mathbf{\vartheta}=\left\{  \left\{  \theta_{i}\right\}  \mid
i=1,2,\ldots,n_{K}-1\right\}  $ satisfying%
\begin{equation}
\mathbf{\kappa}_{n_{K}}^{\left(  2\right)  }\left[  \left\Vert \mathsf{v}%
_{1}\right\Vert _{N},\left\Vert \mathsf{v}_{2}\right\Vert _{N},\ldots
,\left\Vert \mathsf{v}_{n_{K}}\right\Vert _{N},\cos\theta_{1},\cos\theta
_{2},\ldots,\cos\theta_{n_{K}-1}\right]  =\left\langle \left\langle
\mathsf{v}_{1}|\mathsf{v}_{2}|\ldots|\mathsf{v}_{n_{K}}\right\rangle
\right\rangle ,
\end{equation}
where $\mathbf{\kappa}_{n_{K}}^{\left(  2\right)  }$ is a long product of two
$n_{K}$-ary multiplications, which consists of $2\left(  n_{K}-1\right)  +1$ terms.
\end{definition}

We will not consider the completion with respect to the above norm (to obtain
a polyadic analog of Hilbert space) and corresponding limits and boundedness
questions, because it will not give additional arity shapes, in which we are
mostly interested here. Instead, below we turn to some applications and new
general constructions which appear from the above polyadic structures.


\begin{table}[h]
\caption{The arity signature and arity shape of polyadic algebraic structures.
}%
\label{T}
\begin{center}
{\tiny
\begin{tabular}
[c]{||c||c|c||c|c|c|c|c|c|c||}\hline\hline
\textbf{\multirow{2}*{Structures}} & \multicolumn{2}{|c||}{\textbf{Sets}} &
\multicolumn{6}{|c|}{\textbf{Operations and arities}} & \textbf{Arity}%
\\\cline{2-9}
& N & Name & N & \multicolumn{2}{|c|}{Multiplications} &
\multicolumn{2}{|c|}{Additions} & Multiactions & \textbf{shape}\\\hline\hline
\multicolumn{10}{|l|}{\textsf{Group-like polyadic algebraic structures}%
}\\\hline\hline
$%
\begin{array}
[c]{c}%
n\text{-ary magma}\\
\text{(or groupoid)}%
\end{array}
$ & \textbf{1} & ${M}$ & \textbf{1} & \multicolumn{2}{|c|}{$%
\begin{array}
[c]{c}%
\mu_{n}:\\
{M}^{n}\rightarrow{M}%
\end{array}
$} & \multicolumn{2}{|c|}{} &  & \\\hline
$%
\begin{array}
[c]{c}%
n\text{-ary semigroup}\\
\text{(and monoid)}%
\end{array}
$ & \textbf{1} & ${S}$ & \textbf{1} & \multicolumn{2}{|c|}{$%
\begin{array}
[c]{c}%
\mu_{n}:\\
{S}^{n}\rightarrow{S}%
\end{array}
$} & \multicolumn{2}{|c|}{} &  & \\\hline
$%
\begin{array}
[c]{c}%
n\text{-ary quasigroup}\\
\text{(and loop)}%
\end{array}
$ & \textbf{1} & ${Q}$ & \textbf{1} & \multicolumn{2}{|c|}{$%
\begin{array}
[c]{c}%
\mu_{n}:\\
{Q}^{n}\rightarrow{Q}%
\end{array}
$} & \multicolumn{2}{|c|}{} &  & \\\hline
$n$-ary group & \textbf{1} & ${G}$ & \textbf{1} & \multicolumn{2}{|c|}{$%
\begin{array}
[c]{c}%
\mu_{n}:\\
{G}^{n}\rightarrow{G}%
\end{array}
$} & \multicolumn{2}{|c|}{} &  & \\\hline\hline
\multicolumn{10}{|l|}{\textsf{Ring-like polyadic algebraic structures}%
}\\\hline\hline
$\left(  m,n\right)  $-ary ring & \textbf{1} & ${R}$ & \textbf{2} &
\multicolumn{2}{|c|}{$%
\begin{array}
[c]{c}%
\mu_{n}:\\
{R}^{n}\rightarrow{R}%
\end{array}
$} & \multicolumn{2}{|c|}{$%
\begin{array}
[c]{c}%
\nu_{m}:\\
{R}^{m}\rightarrow{R}%
\end{array}
$} &  & \\\hline
$\left(  m,n\right)  $-ary field & \textbf{1} & ${K}$ & \textbf{2} &
\multicolumn{2}{|c|}{$%
\begin{array}
[c]{c}%
\mu_{n}:\\
{K}^{n}\rightarrow{K}%
\end{array}
$} & \multicolumn{2}{|c|}{$%
\begin{array}
[c]{c}%
\nu_{m}:\\
{K}^{m}\rightarrow{K}%
\end{array}
$} &  & \\\hline\hline
\multicolumn{10}{|l|}{\textsf{Module-like polyadic algebraic structures}%
}\\\hline\hline
$%
\begin{array}
[c]{c}%
\text{Module}\\
\text{over }\left(  m,n\right)  \text{-ring}%
\end{array}
$ & \textbf{2} & $R,\mathsf{M}$ & \textbf{4} & \multicolumn{2}{|c|}{$%
\begin{array}
[c]{c}%
\sigma_{n}:\\
{R}^{n}\rightarrow{R}%
\end{array}
$} & $%
\begin{array}
[c]{c}%
\kappa_{m}:\\
{R}^{m}\rightarrow{R}%
\end{array}
$ & $%
\begin{array}
[c]{c}%
\nu_{m_{M}}:\\
\mathsf{M}^{m_{M}}\rightarrow\mathsf{M}%
\end{array}
$ & $%
\begin{array}
[c]{c}%
\rho_{k_{\rho}}:\\
{R}^{k_{\rho}}\times\mathsf{M}\rightarrow\mathsf{M}%
\end{array}
$ & \\\hline
$%
\begin{array}
[c]{c}%
\text{Vector space}\\
\text{over }\left(  m_{K},n_{K}\right)  \text{-field}%
\end{array}
$ & \textbf{2} & $K,\mathsf{V}$ & \textbf{4} & \multicolumn{2}{|c|}{$%
\begin{array}
[c]{c}%
\sigma_{n_{K}}:\\
{K}^{n_{K}}\rightarrow{K}%
\end{array}
$} & $%
\begin{array}
[c]{c}%
\kappa_{m_{K}}:\\
{K}^{m_{K}}\rightarrow{K}%
\end{array}
$ & $%
\begin{array}
[c]{c}%
\nu_{m_{V}}:\\
\mathsf{V}^{m_{V}}\rightarrow\mathsf{V}%
\end{array}
$ & $%
\begin{array}
[c]{c}%
\rho_{k_{\rho}}:\\
{K}^{k_{\rho}}\times\mathsf{V}\rightarrow\mathsf{V}%
\end{array}
$ & $%
\begin{array}
[c]{c}%
\text{(\ref{l1n})}\\
\text{(\ref{l2m})}%
\end{array}
$\\\hline\hline
\multicolumn{10}{|l|}{\textsf{Algebra-like polyadic algebraic structures}%
}\\\hline\hline
$%
\begin{array}
[c]{c}%
\text{Inner pairing space}\\
\text{over }\left(  m_{K},n_{K}\right)  \text{-field}%
\end{array}
$ & \textbf{2} & $K,\mathsf{V}$ & \textbf{5} & $%
\begin{array}
[c]{c}%
\sigma_{n_{K}}:\\
{K}^{n_{K}}\rightarrow{K}%
\end{array}
$ & $%
\begin{array}
[c]{c}%
N\text{-Form}\\
\left\langle \left\langle \bullet..\bullet\right\rangle \right\rangle :\\
\mathsf{V}^{N}\rightarrow{K}%
\end{array}
$ & $%
\begin{array}
[c]{c}%
\kappa_{m_{K}}:\\
{K}^{m_{K}}\rightarrow{K}%
\end{array}
$ & $%
\begin{array}
[c]{c}%
\nu_{m_{V}}:\\
\mathsf{V}^{m_{V}}\rightarrow\mathsf{V}%
\end{array}
$ & $%
\begin{array}
[c]{c}%
\rho_{k_{\rho}}:\\
{K}^{k_{\rho}}\times\mathsf{V}\rightarrow\mathsf{V}%
\end{array}
$ & $%
\begin{array}
[c]{c}%
\text{(\ref{nkk})}\\
\text{(\ref{mvm})}\\
\text{(\ref{nkn})}%
\end{array}
$\\\hline
$%
\begin{array}
[c]{c}%
\left(  m_{A},n_{A}\right)  \text{-algebra}\\
\text{over }\left(  m_{K},n_{K}\right)  \text{-field}%
\end{array}
$ & \textbf{2} & $K,\mathsf{A}$ & \textbf{5} & $%
\begin{array}
[c]{c}%
\sigma_{n_{K}}:\\
{K}^{n_{K}}\rightarrow{K}%
\end{array}
$ & $%
\begin{array}
[c]{c}%
\mu_{n_{A}}:\\
\mathsf{A}^{n}\rightarrow\mathsf{A}%
\end{array}
$ & $%
\begin{array}
[c]{c}%
\kappa_{m_{K}}:\\
{K}^{m_{K}}\rightarrow{K}%
\end{array}
$ & $%
\begin{array}
[c]{c}%
\nu_{m_{A}}:\\
\mathsf{A}^{m_{M}}\rightarrow\mathsf{A}%
\end{array}
$ & $%
\begin{array}
[c]{c}%
\rho_{k_{\rho}}:\\
{K}^{k_{\rho}}\times\mathsf{A}\rightarrow\mathsf{A}%
\end{array}
$ & $\text{(\ref{l1r})}$\\\hline\hline
\end{tabular}
}
\end{center}
\end{table}


To conclude, we present the resulting \textsc{Table \ref{T}} in which the
polyadic algebraic structures are listed together with their arity shapes.

\section*{\textbf{Applications}}

\section{Elements of polyadic operator theory}

Here we consider the $1$-place polyadic operators $\mathbf{T}=\mathbf{F}%
_{k_{F}=1}$ (the case $k_{F}=1$ of the mapping $\mathbf{F}_{k_{F}}$ in
\textbf{Definition \ref{def-link}}) on polyadic inner pairing spaces and
structurally generalize the adjointness and involution concepts.

\begin{remark}
\label{rem-k=1}A polyadic operator is a complicated mapping between polyadic
vector spaces having nontrivial arity shape (\ref{fv1}) which is actually an
action on a set of \textquotedblleft vectors\textquotedblright. However, only
for $k_{F}=1$ it can be written in a formal way multiplicatively, as it is
always done in the binary case.
\end{remark}

Recall (to fix notations and observe analogies) the informal standard
introduction of the operator algebra and the adjoint operator on a binary
pre-Hilbert space $\mathcal{H}$ ($\equiv\mathcal{H}_{m_{K}=2,n_{K}%
=2,m_{V}=2,k_{\rho}=1,N=2}$) over a binary field $\mathbb{K}$ ($\equiv
\mathbb{K}_{m_{K}=2,n_{K}=2}$) (having the underlying set $\left\{
K;\mathsf{V}\right\}  $). For the operator norm $\left\Vert \bullet\right\Vert
_{T}:\left\{  \mathbf{T}\right\}  \rightarrow K$ we use (among many others)
the following definition%
\begin{equation}
\left\Vert \mathbf{T}\right\Vert _{T}=\inf\left\{  M\in K\mid\left\Vert
\mathbf{T}\mathsf{v}\right\Vert \leq M\left\Vert \mathsf{v}\right\Vert
,\forall\mathsf{v}\in\mathsf{V}\right\}  , \label{tm}%
\end{equation}
which is convenient for further polyadic generalization. \textit{Bounded}
operators have $M<\infty$. If on the set of operators $\left\{  \mathbf{T}%
\right\}  $ (as $1$-place mappings $\mathsf{V}\rightarrow\mathsf{V}$) one
defines the addition $\left(  +_{T}\right)  $, product $\left(  \circ
_{T}\right)  $ and scalar multiplication $\left(  \cdot_{T}\right)  $ in the
standard way%
\begin{align}
\left(  \mathbf{T}_{1}+_{T}\mathbf{T}_{2}\right)  \left(  \mathsf{v}\right)
&  =\mathbf{T}_{1}\mathsf{v}+\mathbf{T}_{2}\mathsf{v},\\
\left(  \mathbf{T}_{1}\circ_{T}\mathbf{T}_{2}\right)  \left(  \mathsf{v}%
\right)   &  =\mathbf{T}_{1}\left(  \mathbf{T}_{2}\mathsf{v}\right)  ,\\
\left(  \lambda\cdot_{T}\mathbf{T}\right)  \left(  \mathsf{v}\right)   &
=\lambda\left(  \mathbf{T}\mathsf{v}\right)  ,\ \ \ \lambda\in K,\mathbb{\ \ }%
\mathsf{v}\in\mathsf{V},
\end{align}
then $\left\langle \left\{  \mathbf{T}\right\}  \mid+_{T},\circ_{T}|\cdot
_{T}\right\rangle $ becomes an operator algebra $\mathcal{A}_{T}$
(associativity and distributivity are obvious). The unity $\mathbf{I}$ and
zero $\mathbf{Z}$ of $\mathcal{A}_{T}$ (if they exist), satisfy%
\begin{align}
\mathbf{I}\mathsf{v}  &  =\mathsf{v},\label{iv}\\
\mathbf{Z}\mathsf{v}  &  =\mathsf{z}_{V},\ \ \ \ \ \forall\mathsf{v}%
\in\mathsf{V}, \label{zv}%
\end{align}
respectively, where $\mathsf{z}_{V}\in\mathsf{V}$ is the polyadic
\textquotedblleft zero-vector\textquotedblright.

The connection between operators, linear functionals and inner products is
given by the Riesz representation theorem. Informally, it states that in a
binary pre-Hilbert space $\mathcal{H=}\left\{  K;\mathsf{V}\right\}  $ a
(bounded) linear functional (sesquilinear form) $\mathbf{L}:\mathsf{V}%
\times\mathsf{V}\rightarrow K$ can be \textsl{uniquely} represented as%
\begin{equation}
\mathbf{L}\left(  \mathsf{v}_{1},\mathsf{v}_{2}\right)  =\left\langle
\left\langle \mathbf{T}\mathsf{v}_{1}|\mathsf{v}_{2}\right\rangle
\right\rangle _{sym},\ \ \ \ \ \ \forall\mathsf{v}_{1},\mathsf{v}_{2}%
\in\mathsf{V}, \label{lv}%
\end{equation}
where $\left\langle \left\langle \bullet|\bullet\right\rangle \right\rangle
_{sym}:\mathsf{V}\times\mathsf{V}\rightarrow K$ is a (binary) inner product
with standard properties and $\mathbf{T}:\mathsf{V}\rightarrow\mathsf{V}$ is a
bounded linear operator, such that the norms of $\mathbf{L}$ and $\mathbf{T}$
coincide. Because the linear functionals form a dual space (see
\textbf{Subsection }\ref{sub-polfun}), the relation (\ref{lv}) fixes the shape
of its elements. The main consequence of the Riesz representation theorem is
the existence of the adjoint: for any (bounded) linear operator $\mathbf{T}%
:\mathsf{V}\rightarrow\mathsf{V}$ there exists a (unique bounded)
\textit{adjoint operator} $\mathbf{T}^{\ast}:\mathsf{V}\rightarrow\mathsf{V}$
satisfying%
\begin{equation}
\mathbf{L}\left(  \mathsf{v}_{1},\mathsf{v}_{2}\right)  =\left\langle
\left\langle \mathbf{T}\mathsf{v}_{1}|\mathsf{v}_{2}\right\rangle
\right\rangle _{sym}=\left\langle \left\langle \mathsf{v}_{1}|\mathbf{T}%
^{\ast}\mathsf{v}_{2}\right\rangle \right\rangle _{sym},\ \ \ \ \ \forall
\mathsf{v}_{1},\mathsf{v}_{2}\in\mathsf{V}, \label{tt}%
\end{equation}
and the norms of $\mathbf{T}$ and $\mathbf{T}^{\ast}$ are equal. It follows
from the conjugation symmetry of the standard binary inner product, that
(\ref{tt}) coincides with%
\begin{equation}
\left\langle \left\langle \mathsf{v}_{1}|\mathbf{T}\mathsf{v}_{2}\right\rangle
\right\rangle _{sym}=\left\langle \left\langle \mathbf{T}^{\ast}\mathsf{v}%
_{1}|\mathsf{v}_{2}\right\rangle \right\rangle _{sym},\ \ \ \ \ \forall
\mathsf{v}_{1},\mathsf{v}_{2}\in\mathsf{V}. \label{t1}%
\end{equation}

However, when $\left\langle \left\langle \bullet|\bullet\right\rangle
\right\rangle $ has no symmetry (permutation, conjugation, etc., see, e.g.
\cite{mig76}), it becomes the binary ($N=2$) inner pairing (\ref{aa}), the
binary adjoint consists of 2 operators $\left(  \mathbf{T}^{\star_{12}%
}\right)  \neq\left(  \mathbf{T}^{\star_{21}}\right)  $, $\mathbf{T}%
^{\star_{ij}}:\mathsf{V}\rightarrow\mathsf{V}$, which should be defined by 2
equations%
\begin{align}
\left\langle \left\langle \mathbf{T}\mathsf{v}_{1}|\mathsf{v}_{2}\right\rangle
\right\rangle  &  =\left\langle \left\langle \mathsf{v}_{1}|\mathbf{T}%
^{\star_{12}}\mathsf{v}_{2}\right\rangle \right\rangle ,\\
\left\langle \left\langle \mathsf{v}_{1}|\mathbf{T}\mathsf{v}_{2}\right\rangle
\right\rangle  &  =\left\langle \left\langle \mathbf{T}^{\star_{21}}%
\mathsf{v}_{1}|\mathsf{v}_{2}\right\rangle \right\rangle ,
\end{align}
where $\left(  \star_{12}\right)  \neq\left(  \star_{21}\right)  $ are 2
different star operations satisfying 2 relations%
\begin{align}
\mathbf{T}^{\star_{12}\star_{21}}  &  =\mathbf{T},\label{t1t}\\
\mathbf{T}^{\star_{21}\star_{12}}  &  =\mathbf{T.} \label{t2t}%
\end{align}
If $\left\langle \left\langle \bullet|\bullet\right\rangle \right\rangle
=\left\langle \left\langle \bullet|\bullet\right\rangle \right\rangle _{sym}$
is symmetric, it becomes the inner product in the pre-Hilbert space
$\mathcal{H}$ and the equations (\ref{t1t})--(\ref{t2t}) coincide, while the
operation $\left(  \ast\right)  =\left(  \star_{12}\right)  =\left(
\star_{21}\right)  $ stands for the standard involution%
\begin{equation}
\mathbf{T}^{\ast\ast}=\mathbf{T}. \label{tt1}%
\end{equation}

\subsection{Multistars and polyadic adjoints}

Consider now a special case of the polyadic inner pairing space (see
\textbf{Definition \ref{def-inpar}})%
\begin{equation}
\mathcal{H}_{m_{K},n_{K},m_{V},k_{\rho}=1,N}=\left\langle K;\mathsf{V}%
\mid\mathbf{\sigma}_{m_{K}},\mathbf{\kappa}_{n_{K}};\mathbf{\nu}_{m_{V}}%
\mid\mathbf{\rho}_{k_{\rho}=1}\mid\overset{N}{\overbrace{\left\langle
\left\langle \mathsf{\bullet}|\ldots|\mathsf{\bullet}\right\rangle
\right\rangle }}\right\rangle
\end{equation}
with $1$-place multiaction $\mathbf{\rho}_{k_{\rho}=1}$.

\begin{definition}
The set of $1$-place operators $\mathbf{T}:\mathsf{V}\rightarrow\mathsf{V}$
together with the set of \textquotedblleft scalars\textquotedblright\ $K$
become a polyadic operator algebra $\mathcal{A}_{T}=\left\langle K;\left\{
\mathbf{T}\right\}  \mid\mathbf{\sigma}_{m_{K}},\mathbf{\kappa}_{n_{K}%
};\mathbf{\eta}_{m_{T}},\mathbf{\omega}_{n_{T}}\mid\mathbf{\theta}_{k_{F}%
=1}\right\rangle $, if the operations $\mathbf{\eta}_{m_{T}},\mathbf{\omega
}_{n_{T}},\mathbf{\theta}_{k_{F}=1}$ to define by%
\begin{align}
\mathbf{\eta}_{m_{T}}\left[  \mathbf{T}_{1},\mathbf{T}_{2},\ldots
,\mathbf{T}_{m_{T}}\right]  \left(  \mathsf{v}\right)   &  =\mathbf{\nu
}_{m_{V}}\left[  \mathbf{T}_{1}\mathsf{v},\mathbf{T}_{2}\mathsf{v}%
,\ldots,\mathbf{T}_{m_{T}}\mathsf{v}\right]  ,\label{et}\\
\mathbf{\omega}_{n_{T}}\left[  \mathbf{T}_{1},\mathbf{T}_{2},\ldots
,\mathbf{T}_{n_{T}}\right]  \left(  \mathsf{v}\right)   &  =\mathbf{T}%
_{1}\left(  \mathbf{T}_{2}\ldots\left(  \mathbf{T}_{n_{T}}\mathsf{v}\right)
\right)  ,\\
\mathbf{\theta}_{k_{F}=1}\left\{  \lambda\mid\mathbf{T}\right\}  \left(
\mathsf{v}\right)   &  =\mathbf{\rho}_{k_{\rho}=1}\left\{  \lambda
\mid\mathbf{T}\mathsf{v}\right\}  ,\ \ \forall\lambda\in K,\ \forall
\mathsf{v}\in\mathsf{V}. \label{tl}%
\end{align}

\end{definition}

The arity shape is fixed by

\begin{proposition}
In the polyadic algebra $\mathcal{A}_{T}$ the arity of the operator addition
$m_{T}$ coincides with the \textquotedblleft vector\textquotedblright%
\ addition of the inner pairing space $m_{V}$, i.e.%
\begin{equation}
m_{T}=m_{V}. \label{mtm}%
\end{equation}

\end{proposition}

\begin{proof}
This follows from (\ref{et}).
\end{proof}

To get relations between operators we assume (as in the binary case)
uniqueness: for any $\mathbf{T}_{1},\mathbf{T}_{2}:\mathsf{V}\rightarrow
\mathsf{V}$ it follows from%
\begin{equation}
\left\langle \left\langle \mathsf{v}_{1}|\mathsf{v}_{2}|\ldots|\mathbf{T}%
_{1}\mathsf{v}_{i}|\ldots\mathsf{v}_{N-1}|\mathsf{v}_{N}\right\rangle
\right\rangle =\left\langle \left\langle \mathsf{v}_{1}|\mathsf{v}_{2}%
|\ldots|\mathbf{T}_{2}\mathsf{v}_{i}|\ldots\mathsf{v}_{N-1}|\mathsf{v}%
_{N}\right\rangle \right\rangle , \label{t1t2}%
\end{equation}
that $\mathbf{T}_{1}=\mathbf{T}_{2}$ on any place $i=1,\ldots,N$.

First, by analogy with the binary adjoint (\ref{tt}) we define $N$ different
adjoints for each operator $\mathbf{T}$.

\begin{definition}
Given a polyadic operator $\mathbf{T}:\mathsf{V}\rightarrow\mathsf{V}$ on the
polyadic inner pairing space $\mathcal{H}_{m_{K},n_{K},m_{V},k_{\rho}=1,N}$ we
define a \textit{polyadic adjoint} as the set $\left\{  \mathbf{T}^{\star
_{ij}}\right\}  $ of $N$ operators $\mathbf{T}^{\star_{ij}}$ satisfying the
following $N$ equations%
\begin{align}
&  \left\langle \left\langle \mathbf{T}\mathsf{v}_{1}|\mathsf{v}%
_{2}|\mathsf{v}_{3}|\ldots|\mathsf{v}_{N}\right\rangle \right\rangle
=\left\langle \left\langle \mathsf{v}_{1}|\mathbf{T}^{\star_{12}}%
\mathsf{v}_{2}|\mathsf{v}_{3}|\ldots|\mathsf{v}_{N}\right\rangle \right\rangle
,\nonumber\\
&  \left\langle \left\langle \mathsf{v}_{1}|\mathbf{T}\mathsf{v}%
_{2}|\mathsf{v}_{3}|\ldots|\mathsf{v}_{N}\right\rangle \right\rangle
=\left\langle \left\langle \mathsf{v}_{1}|\mathsf{v}_{2}|\mathbf{T}%
^{\star_{23}}\mathsf{v}_{3}|\ldots|\mathsf{v}_{N}\right\rangle \right\rangle
,\nonumber\\
&  \vdots\nonumber\\
&  \left\langle \left\langle \mathsf{v}_{1}|\mathsf{v}_{2}|\mathsf{v}%
_{3}|\ldots\mathbf{T}\mathsf{v}_{N-1}|\mathsf{v}_{N}\right\rangle
\right\rangle =\left\langle \left\langle \mathsf{v}_{1}|\mathsf{v}%
_{2}|\mathsf{v}_{3}|\ldots|\mathbf{T}^{\star_{N-1,N}}\mathsf{v}_{N}%
\right\rangle \right\rangle ,\nonumber\\
&  \left\langle \left\langle \mathsf{v}_{1}|\mathsf{v}_{2}|\mathsf{v}%
_{3}|\ldots\mathsf{v}_{N-1}|\mathbf{T}\mathsf{v}_{N}\right\rangle
\right\rangle =\left\langle \left\langle \mathbf{T}^{\star_{N,1}}%
\mathsf{v}_{1}|\mathsf{v}_{2}|\mathsf{v}_{3}|\ldots|\mathsf{v}_{N}%
\right\rangle \right\rangle ,\ \ \mathsf{v}_{i}\in\mathsf{V.} \label{tv}%
\end{align}

\end{definition}

In what follows, for the composition we will use the notation $\left(
\mathbf{T}^{\star_{ij}}\right)  ^{\star_{kl}\ldots}\equiv\mathbf{T}%
^{\star_{ij}\star_{kl}\ldots}$. We have from (\ref{tv}) the $N$ relations%
\begin{align}
\mathbf{T}^{\star_{12}\star_{23}\star_{34}\ldots\star_{N-1,N}\star_{N,1}}  &
=\mathbf{T,}\nonumber\\
\mathbf{T}^{\star_{23}\star_{34}\ldots\star_{N-1,N}\star_{N,1}\star_{12}}  &
=\mathbf{T,}\nonumber\\
&  \vdots\nonumber\\
\mathbf{T}^{\star_{N,1}\star_{12}\star_{23}\star_{34}\ldots\star_{N-1,N}}  &
=\mathbf{T,} \label{tn}%
\end{align}
which are called \textit{multistar cycles}.

\begin{definition}
We call the set of adjoint mappings $\left(  \bullet^{\star_{ij}}\right)
:\mathbf{T}\rightarrow\mathbf{T}^{\star_{ij}}$ a \textit{polyadic involution},
if they satisfy the multistar cycles (\ref{tn}).
\end{definition}

If the inner pairing $\left\langle \left\langle \bullet|\ldots|\bullet
\right\rangle \right\rangle $ has more than two places $N\geq3$, we have some
additional structural issues, which do not exist in the binary case.

\textsl{First}, we observe that the set of the adjointness relations
(\ref{tv}) can be described in the framework of the associativity quiver
approach introduced in \cite{dup2012} for polyadic representations. That is,
for general $N\geq3$ in addition to (\ref{tv}) which corresponds to the so
called Post-like associativity quiver (they will be called the
\textit{Post-like adjointness relations}), there also exist other sets. It is
cumbersome to write additional general formulas like (\ref{tv}) for other
non-Post-like cases, while instead we give a clear example for $N=4$.

\begin{example}
The polyadic adjointness relations for $N=4$ consist of the sets corresponding
to different associativity quivers%
\begin{equation}%
\begin{array}
[c]{c}%
\text{{\small 1) }\emph{Post-like adjointness relations}}\\
\left\langle \left\langle \mathbf{T}\mathsf{v}_{1}|\mathsf{v}_{2}%
|\mathsf{v}_{3}|\mathsf{v}_{4}\right\rangle \right\rangle =\left\langle
\left\langle \mathsf{v}_{1}|\mathbf{T}^{\star_{12}}\mathsf{v}_{2}%
|\mathsf{v}_{3}|\mathsf{v}_{4}\right\rangle \right\rangle ,\\
\left\langle \left\langle \mathsf{v}_{1}|\mathbf{T}\mathsf{v}_{2}%
|\mathsf{v}_{3}|\mathsf{v}_{4}\right\rangle \right\rangle =\left\langle
\left\langle \mathsf{v}_{1}|\mathsf{v}_{2}|\mathbf{T}^{\star_{23}}%
\mathsf{v}_{3}|\mathsf{v}_{4}\right\rangle \right\rangle ,\\
\left\langle \left\langle \mathsf{v}_{1}|\mathsf{v}_{2}|\mathbf{T}%
\mathsf{v}_{3}|\mathsf{v}_{4}\right\rangle \right\rangle =\left\langle
\left\langle \mathsf{v}_{1}|\mathsf{v}_{2}|\mathsf{v}_{3}|\mathbf{T}%
^{\star_{34}}\mathsf{v}_{4}\right\rangle \right\rangle ,\\
\left\langle \left\langle \mathsf{v}_{1}|\mathsf{v}_{2}|\mathsf{v}%
_{3}|\mathbf{T}\mathsf{v}_{4}\right\rangle \right\rangle =\left\langle
\left\langle \mathbf{T}^{\star_{41}}\mathsf{v}_{1}|\mathsf{v}_{2}%
|\mathsf{v}_{3}|\mathsf{v}_{4}\right\rangle \right\rangle ,
\end{array}
\ \ \
\begin{array}
[c]{c}%
\text{{\small 2) }\emph{Non-Post-like adjointness relations}}\\
\left\langle \left\langle \mathbf{T}\mathsf{v}_{1}|\mathsf{v}_{2}%
|\mathsf{v}_{3}|\mathsf{v}_{4}\right\rangle \right\rangle =\left\langle
\left\langle \mathsf{v}_{1}|\mathsf{v}_{2}|\mathsf{v}_{3}|\mathbf{T}%
^{\star_{14}}\mathsf{v}_{4}\right\rangle \right\rangle ,\\
\left\langle \left\langle \mathsf{v}_{1}|\mathsf{v}_{2}|\mathsf{v}%
_{3}|\mathbf{T}\mathsf{v}_{4}\right\rangle \right\rangle =\left\langle
\left\langle \mathsf{v}_{1}|\mathsf{v}_{2}|\mathbf{T}^{\star_{43}}%
\mathsf{v}_{3}|\mathsf{v}_{4}\right\rangle \right\rangle ,\\
\left\langle \left\langle \mathsf{v}_{1}|\mathsf{v}_{2}|\mathbf{T}%
\mathsf{v}_{3}|\mathsf{v}_{4}\right\rangle \right\rangle =\left\langle
\left\langle \mathsf{v}_{1}|\mathbf{T}^{\star_{32}}\mathsf{v}_{2}%
|\mathsf{v}_{3}|\mathsf{v}_{4}\right\rangle \right\rangle ,\\
\left\langle \left\langle \mathsf{v}_{1}|\mathbf{T}\mathsf{v}_{2}%
|\mathsf{v}_{3}|\mathsf{v}_{4}\right\rangle \right\rangle =\left\langle
\left\langle \mathbf{T}^{\star_{21}}\mathsf{v}_{1}|\mathsf{v}_{2}%
|\mathsf{v}_{3}|\mathsf{v}_{4}\right\rangle \right\rangle ,
\end{array}
\end{equation}
and the corresponding multistar cycles%
\begin{equation}%
\begin{array}
[c]{c}%
\text{{\small 1) }\emph{Post-like multistar cycles}}\\
\mathbf{T}^{\star_{12}\star_{23}\star_{34}\star_{41}}=\mathbf{T},\\
\mathbf{T}^{\star_{23}\star_{34}\star_{41}\star_{12}}=\mathbf{T},\\
\mathbf{T}^{\star_{34}\star_{41}\star_{12}\star_{23}}=\mathbf{T},\\
\mathbf{T}^{\star_{41}\star_{12}\star_{23}\star_{34}}=\mathbf{T},
\end{array}
\ \ \
\begin{array}
[c]{c}%
\text{{\small 2) }\emph{Non-Post-like multistar cycles}}\\
\mathbf{T}^{\star_{14}\star_{43}\star_{32}\star_{21}}=\mathbf{T,}\\
\mathbf{T}^{\star_{43}\star_{32}\star_{21}\star_{14}}=\mathbf{T,}\\
\mathbf{T}^{\star_{32}\star_{21}\star_{14}\star_{43}}=\mathbf{T,}\\
\mathbf{T}^{\star_{21}\star_{14}\star_{43}\star_{32}}=\mathbf{T.}%
\end{array}
\end{equation}

\end{example}

Thus, if the inner pairing has no symmetry, then both the Post-like and
non-Post-like adjoints and corresponding multistar involutions are different.

\textsl{Second}, in the case $N\geq3$ any symmetry of the multiplace inner
pairing restricts the polyadic adjoint sets and multistar involutions considerably.

\begin{theorem}
\label{theo-sym}If the inner pairing with $N\geq3$ has the full permutation
symmetry%
\begin{equation}
\left\langle \left\langle \mathsf{v}_{1}|\mathsf{v}_{2}|\ldots|\mathsf{v}%
_{N}\right\rangle \right\rangle =\left\langle \left\langle \sigma
\mathsf{v}_{1}|\sigma\mathsf{v}_{2}|\ldots|\sigma\mathsf{v}_{N}\right\rangle
\right\rangle ,\ \ \ \ \forall\sigma\in\mathfrak{S}_{N},
\end{equation}
where $\mathfrak{S}_{N}$ is the symmetric group of $N$ elements, then:

\begin{enumerate}
\item All the multistars coincide $\left(  \star_{ij}\right)  =\left(
\star_{kl}\right)  :=\left(  \ast\right)  $ for any allowed $i,j,k,l=1,\ldots
,N$;

\item All the operators are self-adjoint $\mathbf{T}=\mathbf{T}^{\ast}$.
\end{enumerate}
\end{theorem}

\begin{proof}
\noindent

\begin{enumerate}
\item In each adjointness relation from (\ref{tv}) we place the operator
$\mathbf{T}$ on the l.h.s. to the first position and its multistar adjoint
$\mathbf{T}^{\star_{ij}}$ to the second position, using the full permutation
symmetry, which together with (\ref{t1t2}) gives the equality of all multistar operations.

\item We place the operator $\mathbf{T}$ on the l.h.s. to the first position
and apply the derivation of the involution in the binary case to increasing
cycles of size $i\leq N$ recursively, that is:

For $i=2$%
\begin{align}
\left\langle \left\langle \mathbf{T}\mathsf{v}_{1}|\mathsf{v}_{2}%
|\mathsf{v}_{3}|\ldots|\mathsf{v}_{N}\right\rangle \right\rangle  &
=\left\langle \left\langle \mathsf{v}_{1}|\mathbf{T}^{\ast}\mathsf{v}%
_{2}|\mathsf{v}_{3}|\ldots|\mathsf{v}_{N}\right\rangle \right\rangle
=\left\langle \left\langle \mathbf{T}^{\ast}\mathsf{v}_{2}|\mathsf{v}%
_{1}|\mathsf{v}_{3}|\ldots|\mathsf{v}_{N}\right\rangle \right\rangle
\nonumber\\
&  =\left\langle \left\langle \mathsf{v}_{2}|\mathbf{T}^{\ast\ast}%
\mathsf{v}_{1}|\mathsf{v}_{3}|\ldots|\mathsf{v}_{N}\right\rangle \right\rangle
=\left\langle \left\langle \mathbf{T}^{\ast\ast}\mathsf{v}_{1}|\mathsf{v}%
_{2}|\mathsf{v}_{3}|\ldots|\mathsf{v}_{N}\right\rangle \right\rangle ,
\end{align}
then, using (\ref{t1t2}) we get%
\begin{equation}
\mathbf{T}=\mathbf{T}^{\ast\ast}, \label{ttt}%
\end{equation}
as in the standard binary case. However, for $N\geq3$ we have $N$ higher
cycles in addition.

For $i=3$%
\begin{align}
\left\langle \left\langle \mathbf{T}\mathsf{v}_{1}|\mathsf{v}_{2}%
|\mathsf{v}_{3}|\ldots|\mathsf{v}_{N}\right\rangle \right\rangle  &
=\left\langle \left\langle \mathsf{v}_{1}|\mathbf{T}^{\ast}\mathsf{v}%
_{2}|\mathsf{v}_{3}|\ldots|\mathsf{v}_{N}\right\rangle \right\rangle
=\left\langle \left\langle \mathbf{T}^{\ast}\mathsf{v}_{2}|\mathsf{v}%
_{3}|\mathsf{v}_{1}|\ldots|\mathsf{v}_{N}\right\rangle \right\rangle
\nonumber\\
&  =\left\langle \left\langle \mathsf{v}_{2}|\mathbf{T}^{\ast\ast}%
\mathsf{v}_{3}|\mathsf{v}_{1}|\ldots|\mathsf{v}_{N}\right\rangle \right\rangle
=\left\langle \left\langle \mathbf{T}^{\ast\ast}\mathsf{v}_{3}|\mathsf{v}%
_{1}|\mathsf{v}_{2}|\ldots|\mathsf{v}_{N}\right\rangle \right\rangle
\nonumber\\
&  =\left\langle \left\langle \mathsf{v}_{3}|\mathbf{T}^{\ast\ast\ast
}\mathsf{v}_{1}|\mathsf{v}_{2}|\ldots|\mathsf{v}_{N}\right\rangle
\right\rangle =\left\langle \left\langle \mathbf{T}^{\ast\ast\ast}%
\mathsf{v}_{1}|\mathsf{v}_{2}|\mathsf{v}_{3}|\ldots|\mathsf{v}_{N}%
\right\rangle \right\rangle ,
\end{align}
which together with (\ref{t1t2}) gives%
\begin{equation}
\mathbf{T}=\mathbf{T}^{\ast\ast\ast},
\end{equation}
and after using (\ref{ttt})%
\begin{equation}
\mathbf{T}=\mathbf{T}^{\ast}. \label{ttt1}%
\end{equation}

Similarly, for an arbitrary length of the cycle $i$ we obtain $\mathbf{T}%
=\mathbf{T}^{\overset{i}{\overbrace{\ast\ast\ldots\ast}}}$, which should be
valid for \textsl{each} cycle recursively with $i=2,3,\ldots,N$. Therefore,
for any $N\geq3$ all the operators $\mathbf{T}$ are self-adjoint (\ref{ttt1}),
while $N=2$ is an exceptional case, when we have $\mathbf{T}=\mathbf{T}%
^{\ast\ast}$ (\ref{ttt}) only.
\end{enumerate}
\end{proof}

Now we show that imposing a partial symmetry on the polyadic inner pairing
will give more interesting properties to the adjoint operators. Recall, that
one of possible binary commutativity generalizations of (semi)groups to the
polyadic case is the semicommutativity concept, when in the multiplication
only the first and last elements are exchanged. Similarly, we introduce

\begin{definition}
The polyadic inner pairing is called \textit{semicommutative}, if%
\begin{equation}
\left\langle \left\langle \mathsf{v}_{1}|\mathsf{v}_{2}|\mathsf{v}_{3}%
|\ldots|\mathsf{v}_{N}\right\rangle \right\rangle =\left\langle \left\langle
\mathsf{v}_{N}|\mathsf{v}_{2}|\mathsf{v}_{3}|\ldots|\mathsf{v}_{1}%
\right\rangle \right\rangle ,\ \ \ \mathsf{v}_{i}\in\mathsf{V}. \label{sm}%
\end{equation}

\end{definition}

\begin{proposition}
If the polyadic inner pairing is semicommutative, then for any operator
$\mathbf{T}$ (satisfying Post-like adjointness (\ref{tv})) the last multistar
operation $\left(  \star_{N,1}\right)  $ is a \textsl{binary involution} and
is a composition of all the previous multistars%
\begin{align}
\mathbf{T}^{\star_{N,1}}  &  =\mathbf{T}^{\star_{12}\star_{23}\star_{34}%
\ldots\star_{N-1,N}},\label{tn1}\\
\mathbf{T}^{\star_{N,1}\star_{N,1}}  &  =\mathbf{T}. \label{tn11}%
\end{align}

\end{proposition}

\begin{proof}
It follows from (\ref{tv}) and (\ref{sm}), that%
\begin{align}
\left\langle \left\langle \mathsf{v}_{1}|\mathsf{v}_{2}|\mathsf{v}_{3}%
|\ldots|\mathbf{T}\mathsf{v}_{N}\right\rangle \right\rangle  &  =\left\langle
\left\langle \mathbf{T}\mathsf{v}_{N}|\mathsf{v}_{2}|\mathsf{v}_{3}%
|\ldots|\mathsf{v}_{1}\right\rangle \right\rangle =\left\langle \left\langle
\mathsf{v}_{N}|\mathsf{v}_{2}|\mathsf{v}_{3}|\ldots|\mathbf{T}^{\star
_{12}\star_{23}\star_{34}\ldots\star_{N-1,N}}\mathsf{v}_{1}\right\rangle
\right\rangle \\
&  =\left\langle \left\langle \mathbf{T}^{\star_{12}\star_{23}\star_{34}%
\ldots\star_{N-1,N}}\mathsf{v}_{1}|\mathsf{v}_{2}|\mathsf{v}_{3}%
|\ldots|\mathsf{v}_{N}\right\rangle \right\rangle =\left\langle \left\langle
\mathbf{T}^{\star_{N,1}}\mathsf{v}_{1}|\mathsf{v}_{2}|\mathsf{v}_{3}%
|\ldots|\mathsf{v}_{N}\right\rangle \right\rangle ,
\end{align}
which using (\ref{t1t2}) gives (\ref{tn1}), (\ref{tn11}) follows from the
first multistar cocycle in (\ref{tn}).
\end{proof}

The adjointness relations (\ref{tv}) (of all kinds) together with (\ref{tl})
and (\ref{mtm}) allows us to fix the arity shape of the polyadic operator
algebra $\mathcal{A}_{T}$. We will assume that the arity of the operator
multiplication in $\mathcal{A}_{T}$ coincides with the number of places of the
inner pairing $N$ (\ref{aa})%
\begin{equation}
n_{T}=N, \label{ntn}%
\end{equation}
because it is in agreement with (\ref{tv}). Thus, the arity shape of the
polyadic operator algebra becomes
\begin{equation}
\mathcal{A}_{T}=\left\langle K;\left\{  \mathbf{T}\right\}  \mid
\mathbf{\sigma}_{m_{K}},\mathbf{\kappa}_{n_{K}};\mathbf{\eta}_{m_{T}=m_{V}%
},\mathbf{\omega}_{n_{T}=N}\mid\mathbf{\theta}_{k_{F}=k_{\rho}=1}\right\rangle
,
\end{equation}

\begin{definition}
We call the operator algebra $\mathcal{A}_{T}$ which has the arity $n_{T}=N$ a
\textit{nonderived polyadic operator algebra}.
\end{definition}

Let us investigate some structural properties of $\mathcal{A}_{T}$ and types
of polyadic operators.

\begin{remark}
\label{rem-oper}We can only \textsl{define}, but not derive, as in the binary
case, the action of any multistar $\left(  \star_{ij}\right)  $ on the product
of operators, because in the nonderived $n_{T}$-ary algebra we have a fixed
number of operators in a product and sum, that is $\ell^{\prime}\left(
n_{T}-1\right)  +1$ and $\ell^{\prime\prime}\left(  m_{T}-1\right)  +1$,
correspondingly, where $\ell^{\prime}$ is the number of $n_{T}$-ary
multiplications and $\ell^{\prime}$ is the number of $m_{T}$-ary additions.
Therefore, we cannot transfer (one at a time) all the polyadic operators from
one place in the inner pairing to another place, as is done in the standard
proof in the binary case.
\end{remark}

Taking this into account, as well as consistency under the multistar cycles
(\ref{tn}), we arrive at

\begin{definition}
The fixed multistar operation acts on the $\ell=1$ product of $n_{T}$ polyadic
operators, depending on the \textit{sequential number of the multistar}
$\left(  \star_{ij}\right)  $ (for the Post-like adjointness relations
(\ref{tv}))%
\begin{equation}
s_{ij}:=\left\{
\begin{array}
[c]{c}%
\dfrac{i+j-1}{2},\ \ \ \ \text{if }3\leq i+j\leq2N-1\\
N,\ \ \ \ \text{if }i\ j=N,
\end{array}
\right.  \ \ \ \ \ s_{ij}=1,2,\ldots,N-1,N, \label{sij}%
\end{equation}
in the following way%
\begin{equation}
\left(  \mathbf{\omega}_{n_{T}}\left[  \mathbf{T}_{1},\mathbf{T}_{2}%
,\ldots,\mathbf{T}_{n_{T}-1},\mathbf{T}_{n_{T}}\right]  \right)  ^{\star_{ij}%
}=\left\{
\begin{array}
[c]{c}%
\mathbf{\omega}_{n_{T}}\left[  \mathbf{T}_{n_{T}}^{\star_{ij}},\mathbf{T}%
_{n_{T}-1}^{\star_{ij}},\ldots,\mathbf{T}_{2}^{\star_{ij}},\mathbf{T}%
_{1}^{\star_{ij}}\right]  ,\ \ \ \ \text{if }s_{ij}\text{ is odd,}\\[5pt]%
\mathbf{\omega}_{n_{T}}\left[  \mathbf{T}_{1}^{\star_{ij}},\mathbf{T}%
_{2}^{\star_{ij}},\ldots,\mathbf{T}_{n_{T}-1}^{\star_{ij}},\mathbf{T}_{n_{T}%
}^{\star_{ij}}\right]  ,\ \ \ \ \text{if }s_{ij}\text{ is even.}%
\end{array}
\right.  \label{wt}%
\end{equation}

\end{definition}

A rule similar to (\ref{wt}) holds also for non-Post-like adjointness
relations, but their concrete form depends of the corresponding non-Post-like
associative quiver.

Sometimes, to shorten notation, it will be more convenient to mark a multistar
by the sequential number (\ref{sij}), such that $\left(  \star_{ij}\right)
\Rightarrow\left(  \star_{s_{ij}}\right)  $, e.g. $\left(  \star_{23}\right)
\Rightarrow\left(  \star_{2}\right)  $, $\left(  \star_{N,1}\right)
\Rightarrow\left(  \star_{N}\right)  $, etc. Also, in examples, for the
ternary multiplication we will use the square brackets without the name of
operation, if it is clear from the context, e.g. $\mathbf{\omega}_{3}\left[
\mathbf{T}_{1},\mathbf{T}_{2},\mathbf{T}_{3}\right]  \Rightarrow\left[
\mathbf{T}_{1},\mathbf{T}_{2},\mathbf{T}_{3}\right]  $, etc.

\begin{example}
\label{ex-ter}In the lowest ternary case $N=3$ we have%
\begin{equation}%
\begin{array}
[c]{c}%
\text{{\small 1) }\emph{Post-like adjointness relations}}\\
\left\langle \left\langle \mathbf{T}\mathsf{v}_{1}|\mathsf{v}_{2}%
|\mathsf{v}_{3}\right\rangle \right\rangle =\left\langle \left\langle
\mathsf{v}_{1}|\mathbf{T}^{\star_{1}}\mathsf{v}_{2}|\mathsf{v}_{3}%
\right\rangle \right\rangle ,\\
\left\langle \left\langle \mathsf{v}_{1}|\mathbf{T}\mathsf{v}_{2}%
|\mathsf{v}_{3}\right\rangle \right\rangle =\left\langle \left\langle
\mathsf{v}_{1}|\mathsf{v}_{2}|\mathbf{T}^{\star_{2}}\mathsf{v}_{3}%
\right\rangle \right\rangle ,\\
\left\langle \left\langle \mathsf{v}_{1}|\mathsf{v}_{2}|\mathbf{T}%
\mathsf{v}_{3}\right\rangle \right\rangle =\left\langle \left\langle
\mathbf{T}^{\star_{3}}\mathsf{v}_{1}|\mathsf{v}_{2}|\mathsf{v}_{3}%
\right\rangle \right\rangle ,
\end{array}
\ \ \
\begin{array}
[c]{c}%
\text{{\small 2) }\emph{Non-Post-like adjointness relations}}\\
\left\langle \left\langle \mathbf{T}\mathsf{v}_{1}|\mathsf{v}_{2}%
|\mathsf{v}_{3}\right\rangle \right\rangle =\left\langle \left\langle
\mathsf{v}_{1}|\mathsf{v}_{2}|\mathbf{T}^{\star_{3}}\mathsf{v}_{3}%
\right\rangle \right\rangle ,\\
\left\langle \left\langle \mathsf{v}_{1}|\mathsf{v}_{2}|\mathbf{T}%
\mathsf{v}_{3}\right\rangle \right\rangle =\left\langle \left\langle
\mathsf{v}_{1}|\mathbf{T}^{\star_{2}}\mathsf{v}_{2}|\mathsf{v}_{3}%
\right\rangle \right\rangle ,\\
\left\langle \left\langle \mathsf{v}_{1}|\mathbf{T}\mathsf{v}_{2}%
|\mathsf{v}_{3}\right\rangle \right\rangle =\left\langle \left\langle
\mathbf{T}^{\star_{1}}\mathsf{v}_{1}|\mathsf{v}_{2}|\mathsf{v}_{3}%
\right\rangle \right\rangle ,
\end{array}
\label{ta3}%
\end{equation}
and the corresponding multistar cycles%
\begin{equation}%
\begin{array}
[c]{c}%
\text{{\small 1) }\emph{Post-like multistar cycles}}\\
\mathbf{T}^{\star_{1}\star_{2}\star_{3}}=\mathbf{T},\\
\mathbf{T}^{\star_{2}\star_{3}\star_{1}}=\mathbf{T},\\
\mathbf{T}^{\star_{3}\star_{1}\star_{2}}=\mathbf{T},
\end{array}
\ \ \
\begin{array}
[c]{c}%
\text{{\small 2) }\emph{Non-Post-like multistar cycles}}\\
\mathbf{T}^{\star_{3}\star_{2}\star_{1}}=\mathbf{T,}\\
\mathbf{T}^{\star_{2}\star_{1}\star_{3}}=\mathbf{T,}\\
\mathbf{T}^{\star_{1}\star_{3}\star_{2}}=\mathbf{T.}%
\end{array}
\label{tc3}%
\end{equation}
Using (\ref{wt}) we obtain the following ternary conjugation rules%
\begin{align}
\left(  \left[  \mathbf{T}_{1},\mathbf{T}_{2},\mathbf{T}_{3}\right]  \right)
^{\star_{1}}  &  =\left[  \mathbf{T}_{3}^{\star_{1}},\mathbf{T}_{2}^{\star
_{1}},\mathbf{T}_{1}^{\star_{1}}\right]  ,\\
\left(  \left[  \mathbf{T}_{1},\mathbf{T}_{2},\mathbf{T}_{3}\right]  \right)
^{\star_{2}}  &  =\left[  \mathbf{T}_{1}^{\star_{2}},\mathbf{T}_{2}^{\star
_{2}},\mathbf{T}_{3}^{\star_{2}}\right]  ,\\
\left(  \left[  \mathbf{T}_{1},\mathbf{T}_{2},\mathbf{T}_{3}\right]  \right)
^{\star_{3}}  &  =\left[  \mathbf{T}_{3}^{\star_{3}},\mathbf{T}_{2}^{\star
_{3}},\mathbf{T}_{1}^{\star_{3}}\right]  ,
\end{align}
which are common for both Post-like and non-Post-like adjointness relations
(\ref{ta3}).
\end{example}

\begin{definition}
A polyadic operator $\mathbf{T}$ is called \textit{self-adjoint}, if all
multistar operations are identities, i.e. $\left(  \star_{ij}\right)
=\operatorname*{id}$, $\forall i,j$.
\end{definition}

\subsection{Polyadic isometry and projection}

Now we introduce polyadic analogs for the important types of operators:
isometry, unitary, and (orthogonal) projection. Taking into account
\emph{Remark} \ref{rem-oper}, we again cannot move operators singly, and
instead of proving the operator relations, as it is usually done in the binary
case, we can only exploit some mnemonic rules to \textsl{define} the
corresponding relations between polyadic operators.

If the polyadic operator algebra $\mathcal{A}_{T}$ contains a unit
$\mathbf{I}$ and zero $\mathbf{Z}$ (see (\ref{iv})--(\ref{zv})) we define the
conditions of polyadic isometry and orthogonality:

\begin{definition}
\label{def-isom}A polyadic operator $\mathbf{T}$ is called a \textit{polyadic
isometry}, if it preserves the polyadic inner pairing%
\begin{equation}
\left\langle \left\langle \mathbf{T}\mathsf{v}_{1}|\mathbf{T}\mathsf{v}%
_{2}|\mathbf{T}\mathsf{v}_{3}|\ldots|\mathbf{T}\mathsf{v}_{N}\right\rangle
\right\rangle =\left\langle \left\langle \mathsf{v}_{1}|\mathsf{v}%
_{2}|\mathsf{v}_{3}|\ldots|\mathsf{v}_{N}\right\rangle \right\rangle ,
\label{tvv}%
\end{equation}
and satisfies%
\begin{align}
\mathbf{\omega}_{n_{T}}\left[  \mathbf{T}^{\star_{N-1,N}},\mathbf{T}%
^{\star_{N-2,N-1}\star_{N-1,N}},\ldots\mathbf{T}^{\star_{23}\star_{34}%
\ldots\star_{N-2,N-1}\star_{N-1,N}},\mathbf{T}^{\star_{12}\star_{23}\star
_{34}\ldots\star_{N-2,N-1}\star_{N-1,N}},\mathbf{T}\right]   &  =\mathbf{I,}%
\nonumber\\
+\left(  N-1\right)
\ cycle\ permutations\ of\ multistars\ in\ the\ first\ \left(  N-1\right)
\ terms  &  . \label{wti}%
\end{align}

\end{definition}

\begin{remark}
If the multiplication in $\mathcal{A}_{T}$ is derived and all multistars are
equal, then the polyadic isometry operators satisfy some kind of
$N$-regularity \cite{dup/mar7} or regular $N$-cocycle condition
\cite{dup/mar2001a}.
\end{remark}

\begin{proposition}
The polyadic isometry operator $\mathbf{T}$ preserves the polyadic norm%
\begin{equation}
\left\Vert \mathbf{T}\mathsf{v}\right\Vert _{N}=\left\Vert \mathsf{v}%
\right\Vert _{N},\ \ \ \ \ \forall\mathsf{v}\in\mathsf{V}. \label{tv1}%
\end{equation}

\end{proposition}

\begin{proof}
It follows from (\ref{kv}) and (\ref{tvv}), that%
\begin{equation}
\mathbf{\kappa}_{n_{K}}\left[  \overset{n_{K}}{\overbrace{\left\Vert
\mathbf{T}\mathsf{v}\right\Vert _{N},\left\Vert \mathbf{T}\mathsf{v}%
\right\Vert _{N},\ldots,\left\Vert \mathbf{T}\mathsf{v}\right\Vert _{N}}%
}\right]  =\mathbf{\kappa}_{n_{K}}\left[  \overset{n_{K}}{\overbrace
{\left\Vert \mathsf{v}\right\Vert _{N},\left\Vert \mathsf{v}\right\Vert
_{N},\ldots,\left\Vert \mathsf{v}\right\Vert _{N}}}\right]  ,
\end{equation}
which gives (\ref{tv1}), when $n_{K}=N$.
\end{proof}

\begin{definition}
If for $N$ polyadic operators $\mathbf{T}_{i}$ we have%
\begin{equation}
\left\langle \left\langle \mathbf{T}_{1}\mathsf{v}_{1}|\mathbf{T}%
_{2}\mathsf{v}_{2}|\mathbf{T}_{3}\mathsf{v}_{3}|\ldots|\mathbf{T}%
_{N}\mathsf{v}_{N}\right\rangle \right\rangle =\mathsf{z}_{K},\ \ \ \forall
\mathsf{v}_{i}\in\mathsf{V},
\end{equation}
where $\mathsf{z}_{K}\in\mathsf{V}$ is the zero of the underlying polyadic
field $\mathbb{K}_{m_{K},n_{K}}$, then we say that $\mathbf{T}_{i}$ are
\textit{(polyadically) orthogonal}, and they satisfy%
\begin{align}
\mathbf{\omega}_{n_{T}}\left[  \mathbf{T}_{1}^{\star_{N-1,N}},\mathbf{T}%
_{2}^{\star_{N-2,N-1}\star_{N-1,N}},\ldots\mathbf{T}_{3}^{\star_{23}\star
_{34}\ldots\star_{N-2,N-1}\star_{N-1,N}},\mathbf{T}_{N-1}^{\star_{12}%
\star_{23}\star_{34}\ldots\star_{N-2,N-1}\star_{N-1,N}},\mathbf{T}_{N}\right]
&  =\mathbf{Z,}\label{ort}\\
+\left(  N-1\right)
\ cycle\ permutations\ of\ multistars\ in\ the\ first\ \left(  N-1\right)
\ terms  &  .
\end{align}

\end{definition}

The polyadic analog of projection is given by

\begin{definition}
If a polyadic operator $\mathbf{P}\in\mathcal{A}_{T}$ satisfies the polyadic
idempotency condition%
\begin{equation}
\mathbf{\omega}_{n_{T}}\left[  \overset{n_{T}}{\overbrace{\mathbf{P}%
,\mathbf{P},\ldots\mathbf{P}}}\right]  =\mathbf{P,} \label{pp}%
\end{equation}
then it is called a \textit{polyadic projection}.
\end{definition}

By analogy with the binary case, polyadic projections can be constructed from
polyadic isometry operators in a natural way.

\begin{proposition}
If $\mathbf{T}\in\mathcal{A}_{T}$ is a polyadic isometry, then%
\begin{align}
\mathbf{P}_{\mathbf{T}}^{\left(  1\right)  }  &  =\mathbf{\omega}_{n_{T}%
}\left[  \mathbf{T},\mathbf{T}^{\star_{N-1,N}},\mathbf{T}^{\star
_{N-2,N-1}\star_{N-1,N}},\ldots\mathbf{T}^{\star_{23}\star_{34}\ldots
\star_{N-2,N-1}\star_{N-1,N}},\mathbf{T}^{\star_{12}\star_{23}\star_{34}%
\ldots\star_{N-2,N-1}\star_{N-1,N}}\right]  ,\nonumber\\
&  +\left(  N-1\right)
\ cycle\ permutations\ of\ multistars\ in\ the\ last\ \left(  N-1\right)
\ terms. \label{pw}%
\end{align}
are the corresponding polyadic projections $\mathbf{P}_{\mathbf{T}}^{\left(
k\right)  }$, $k=1,\ldots,N$, satisfying (\ref{pp}).
\end{proposition}

\begin{definition}
A polyadic operator $\mathbf{T}\in\mathcal{A}_{T}$ is called \textit{normal},
if%
\begin{align}
&  \mathbf{\omega}_{n_{T}}\left[  \mathbf{T}^{\star_{N-1,N}},\mathbf{T}%
^{\star_{N-2,N-1}\star_{N-1,N}},\ldots\mathbf{T}^{\star_{23}\star_{34}%
\ldots\star_{N-2,N-1}\star_{N-1,N}},\mathbf{T}^{\star_{12}\star_{23}\star
_{34}\ldots\star_{N-2,N-1}\star_{N-1,N}},\mathbf{T}\right]  =\nonumber\\
&  \mathbf{\omega}_{n_{T}}\left[  \mathbf{T},\mathbf{T}^{\star_{N-1,N}%
},\mathbf{T}^{\star_{N-2,N-1}\star_{N-1,N}},\ldots\mathbf{T}^{\star_{23}%
\star_{34}\ldots\star_{N-2,N-1}\star_{N-1,N}},\mathbf{T}^{\star_{12}\star
_{23}\star_{34}\ldots\star_{N-2,N-1}\star_{N-1,N}}\right]  ,\nonumber\\
&  +\left(  N-1\right)
\ cycle\ permutations\ of\ multistars\ in\ the\ \left(  N-1\right)  \ terms.
\end{align}

\end{definition}

\begin{proof}
Insert (\ref{pw}) into (\ref{pp}) and use (\ref{wti}) together with $n_{T}%
$-ary associativity.
\end{proof}

\begin{definition}
If all the polyadic projections (\ref{pw}) are equal to unity $\mathbf{P}%
_{\mathbf{T}}^{\left(  k\right)  }=\mathbf{I}$, then the corresponding
polyadic isometry operator $\mathbf{T}$ is called a \textit{polyadic unitary
operator}.
\end{definition}

It can be shown, that each polyadic unitary operator is querable
(\textquotedblleft polyadically invertible\textquotedblright), such that it
has a querelement in $\mathcal{A}_{T}$.

\subsection{Towards polyadic analog of $C^{\ast}$-algebras}

Let us, first, generalize the operator binary norm (\ref{tm}) to the polyadic
case. This can be done, provided that a binary ordering on the underlying
polyadic field $\mathbb{K}_{m_{K},n_{K}}$ can be introduced.

\begin{definition}
The polyadic operator norm $\left\Vert \bullet\right\Vert _{T}:\left\{
\mathbf{T}\right\}  \rightarrow K$ is defined by%
\begin{equation}
\left\Vert \mathbf{T}\right\Vert _{T}=\inf\left\{  M\in K\mid\left\Vert
\mathbf{T}\mathsf{v}\right\Vert _{N}\leq\mu_{n_{K}}\left[  \overset{n_{K}%
-1}{\overbrace{M,\ldots,M}},\left\Vert \mathsf{v}\right\Vert _{N}\right]
,\forall\mathsf{v}\in\mathsf{V}\right\}  , \label{tti}%
\end{equation}
where $\left\Vert \bullet\right\Vert _{N}$ is the polyadic norm in the inner
pairing space $\mathcal{H}_{m_{K},n_{K},m_{V},k_{\rho}=1,N}$ and $\mu_{n_{K}}$
is the $n_{K}$-ary multiplication in $\mathbb{K}_{m_{K},n_{K}}$.
\end{definition}

\begin{definition}
The polyadic operator norm is called \textit{submultiplicative}, if%
\begin{align}
\left\Vert \mathbf{\omega}_{n_{T}}\left[  \mathbf{T}_{1},\mathbf{T}_{2}%
,\ldots,\mathbf{T}_{n_{T}}\right]  \right\Vert _{T}  &  \leq\mu_{n_{K}}\left[
\left\Vert \mathbf{T}_{1}\right\Vert _{T},\left\Vert \mathbf{T}_{2}\right\Vert
_{T},\ldots,\left\Vert \mathbf{T}_{n_{K}}\right\Vert _{T}\right]  ,\\
n_{T}  &  =n_{K}.
\end{align}

\end{definition}

\begin{definition}
The polyadic operator norm is called \textit{subadditive}, if%
\begin{align}
\left\Vert \mathbf{\eta}_{m_{T}}\left[  \mathbf{T}_{1},\mathbf{T}_{2}%
,\ldots,\mathbf{T}_{n_{T}}\right]  \right\Vert _{T}  &  \leq\nu_{m_{K}}\left[
\left\Vert \mathbf{T}_{1}\right\Vert _{T},\left\Vert \mathbf{T}_{2}\right\Vert
_{T},\ldots,\left\Vert \mathbf{T}_{m_{K}}\right\Vert _{T}\right]  ,\\
m_{T}  &  =m_{K}.
\end{align}

\end{definition}

By analogy with the binary case, we have

\begin{definition}
The polyadic operator algebra $\mathcal{A}_{T}$ equipped with the
submultiplicative norm $\left\Vert \bullet\right\Vert _{T}$ a \textit{polyadic
Banach algebra} of operators $\mathcal{B}_{T}$.
\end{definition}

The connection between the polyadic norms of operators and their polyadic
adjoints is given by

\begin{proposition}
For polyadic operators in the inner pairing space $\mathcal{H}_{m_{K}%
,n_{K},m_{V},k_{\rho}=1,N}$

\begin{enumerate}
\item The following $N$ \textit{multi-}$C^{\ast}$\textit{-relations}%
\begin{align}
&  \left\Vert \mathbf{\omega}_{n_{T}}\left[  \mathbf{T}^{\star_{N-1,N}%
},\mathbf{T}^{\star_{N-2,N-1}\star_{N-1,N}},\ldots\mathbf{T}^{\star_{23}%
\star_{34}\ldots\star_{N-2,N-1}\star_{N-1,N}},\mathbf{T}^{\star_{12}\star
_{23}\star_{34}\ldots\star_{N-2,N-1}\star_{N-1,N}},\mathbf{T}\right]
\right\Vert _{N}\nonumber\\
&  =\mu_{n_{K}}\left[  \overset{n_{K}}{\overbrace{\left\Vert \mathbf{T}%
\right\Vert _{T},\left\Vert \mathbf{T}\right\Vert _{T},\ldots,\left\Vert
\mathbf{T}\right\Vert _{T}}}\right]  \mathbf{,}\nonumber\\
&  +\left(  N-1\right)  \ cycle\ permutations\ of\ \left(  N-1\right)
\ terms\ with\ multistars, \label{nc}%
\end{align}
take place, if $n_{T}=n_{K}$.

\item The polyadic norms of operator and its all adjoints coincide%
\begin{equation}
\left\Vert \mathbf{T}^{\star_{i,j}}\right\Vert _{T}=\left\Vert \mathbf{T}%
\right\Vert _{T},\ \forall i,j\in1,\ldots,N.
\end{equation}

\end{enumerate}
\end{proposition}

\begin{proof}
Both statements follow from (\ref{tv}) and the definition of the polyadic
operator norm (\ref{tti}).
\end{proof}

Therefore, we arrive to

\begin{definition}
The operator Banach algebra $\mathcal{B}_{T}$ satisfying the multi-$C^{\ast}%
$-relations is called a polyadic operator \textit{multi-}$C^{\ast}%
$\textit{-algebra}.
\end{definition}

The first example of a multi-$C^{\ast}$-algebra (as in the binary case) can be
constructed from one isometry operator (see \textbf{Definition \ref{def-isom}}).

\begin{definition}
A polyadic algebra generated by one isometry operator $\mathbf{T}$ satisfying
(\ref{wti}) on the inner pairing space $\mathcal{H}_{m_{K},n_{K},m_{V}%
,k_{\rho}=1,N}$ represents a \textit{polyadic Toeplitz algebra} $\mathcal{T}%
_{m_{T},n_{T}}$ and has the arity shape $m_{T}=m_{V}$, $n_{T}=N$.
\end{definition}

\begin{example}
The ternary Toeplitz algebra $\mathcal{T}_{3,3}$ is represented by the
operator $\mathbf{T}$ and relations%
\begin{equation}%
\begin{array}
[c]{c}%
\left[  \mathbf{T}^{\star_{1}},\mathbf{T}^{\star_{3}\star_{1}},\mathbf{T}%
\right]  =\mathbf{I},\\
\left[  \mathbf{T}^{\star_{2}},\mathbf{T}^{\star_{1}\star_{2}},\mathbf{T}%
\right]  =\mathbf{I},\\
\left[  \mathbf{T}^{\star_{3}},\mathbf{T}^{\star_{2}\star_{3}},\mathbf{T}%
\right]  =\mathbf{I}.
\end{array}
\end{equation}

\end{example}

\begin{example}
If the inner pairing is semicommutative (\ref{sm}), then $\left(  \star
_{3}\right)  $ can be eliminated by%
\begin{align}
\mathbf{T}^{\star_{3}}  &  =\mathbf{T}^{\star_{1}\star_{2}},\label{t3}\\
\mathbf{T}^{\star_{3}\star_{3}}  &  =\mathbf{T,}%
\end{align}
and the corresponding relations representing $\mathcal{T}_{3,3}$ become%
\begin{equation}%
\begin{array}
[c]{c}%
\left[  \mathbf{T}^{\star_{1}},\mathbf{T}^{\star_{1}},\mathbf{T}\right]
=\mathbf{I},\\
\left[  \mathbf{T}^{\star_{2}},\mathbf{T}^{\star_{1}\star_{2}},\mathbf{T}%
\right]  =\mathbf{I},\\
\left[  \mathbf{T}^{\star_{1}\star_{2}},\mathbf{T}^{\star_{2}},\mathbf{T}%
\right]  =\mathbf{I}.
\end{array}
\end{equation}

\end{example}

Let us consider $M$ polyadic operators $\mathbf{T}_{1}\mathbf{T}_{2}%
\ldots\mathbf{T}_{M}\in\mathcal{B}_{T}$ and the related partial (in the usual
sense) isometries (\ref{pp}) which are mutually orthogonal (\ref{ort}). In the
binary case, the algebra generated by $M$ operators, such that the sum of the
related orthogonal partial projections is unity, represents the Cuntz algebra
$\mathcal{O}_{M}$ \cite{cun2}.

\begin{definition}
A polyadic algebra generated by $M$ polyadic isometric operators
$\mathbf{T}_{1}\mathbf{T}_{2}\ldots\mathbf{T}_{M}\in\mathcal{B}_{T}$
satisfying%
\begin{equation}
\mathbf{\eta}_{m_{T}}^{\left(  \ell_{a}\right)  }\left[  \mathbf{P}%
_{\mathbf{T}_{1}}^{\left(  k\right)  },\mathbf{P}_{\mathbf{T}_{2}}^{\left(
k\right)  }\ldots\mathbf{P}_{\mathbf{T}_{M}}^{\left(  k\right)  }\right]
=\mathbf{I},\ \ \ k=1,\ldots,N,
\end{equation}
where $\mathbf{P}_{\mathbf{T}_{i}}^{\left(  k\right)  }$ are given by
(\ref{pw})and $\mathbf{\eta}_{m_{T}}^{\left(  \ell_{a}\right)  }$ is a
\textquotedblleft long polyadic addition\textquotedblright\ (\ref{et}),
represents a \textit{polyadic Cuntz algebra} $p\mathcal{O}_{M\mid m_{T},n_{T}%
}$, which has the arity shape%
\begin{equation}
M=\ell_{a}\left(  m_{T}-1\right)  +1,
\end{equation}
where $\ell_{a}$ is number of \textquotedblleft$m_{T}$-ary
additions\textquotedblright.
\end{definition}

Below we use the same notations, as in \emph{Example }\ref{ex-ter}, also the
ternary addition will be denoted by $\left(  +_{3}\right)  $ as follows:
$\mathbf{\eta}_{3}\left[  \mathbf{T}_{1},\mathbf{T}_{2},\mathbf{T}_{3}\right]
\equiv\mathbf{T}_{1}+_{3}\mathbf{T}_{2}+_{3}\mathbf{T}_{3}$.

\begin{example}
In the ternary case $m_{T}=n_{T}=3$ and one ternary addition $\ell_{a}=1$, we
have $M=3$ mutually orthogonal isometries $\mathbf{T}_{1},\mathbf{T}%
_{2},\mathbf{T}_{3}\in\mathcal{B}_{T}$ and $N=3$ multistars $\left(  \star
_{i}\right)  $. In case of the Post-like multistar cocycles (\ref{tc3}) they
satisfy%
\begin{equation}%
\begin{array}
[c]{c}%
\text{Isometry\emph{ conditions}}\\
\left[  \mathbf{T}_{i}^{\star_{1}},\mathbf{T}_{i}^{\star_{3}\star_{1}%
},\mathbf{T}_{i}\right]  =\mathbf{I},\\
\left[  \mathbf{T}_{i}^{\star_{2}},\mathbf{T}_{i}^{\star_{1}\star_{2}%
},\mathbf{T}_{i}\right]  =\mathbf{I},\\
\left[  \mathbf{T}_{i}^{\star_{3}},\mathbf{T}_{i}^{\star_{2}\star_{3}%
},\mathbf{T}_{i}\right]  =\mathbf{I},\\
i=1,2,3,
\end{array}
\ \ \
\begin{array}
[c]{c}%
\text{\emph{Orthogonality conditions}}\\
\left[  \mathbf{T}_{i}^{\star_{1}},\mathbf{T}_{j}^{\star_{3}\star_{1}%
},\mathbf{T}_{k}\right]  =\mathbf{Z},\\
\left[  \mathbf{T}_{i}^{\star_{2}},\mathbf{T}_{j}^{\star_{1}\star_{2}%
},\mathbf{T}_{k}\right]  =\mathbf{Z},\\
\left[  \mathbf{T}_{i}^{\star_{3}},\mathbf{T}_{j}^{\star_{2}\star_{3}%
},\mathbf{T}_{k}\right]  =\mathbf{Z},\\
i,j,k=1,2,3,\ \ i\neq j\neq k,
\end{array}
\end{equation}
and the (sum of projections) relations%
\begin{align}
\left[  \mathbf{T}_{1},\mathbf{T}_{1}^{\star_{1}},\mathbf{T}_{1}^{\star
_{3}\star_{1}}\right]  +_{3}\left[  \mathbf{T}_{2},\mathbf{T}_{2}^{\star_{1}%
},\mathbf{T}_{2}^{\star_{3}\star_{1}}\right]  +_{3}\left[  \mathbf{T}%
_{3},\mathbf{T}_{3}^{\star_{1}},\mathbf{T}_{3}^{\star_{3}\star_{1}}\right]
&  =\mathbf{I},\\
\left[  \mathbf{T}_{1},\mathbf{T}_{1}^{\star_{2}},\mathbf{T}_{1}^{\star
_{1}\star_{2}}\right]  +_{3}\left[  \mathbf{T}_{2},\mathbf{T}_{2}^{\star_{2}%
},\mathbf{T}_{2}^{\star_{1}\star_{2}}\right]  +_{3}\left[  \mathbf{T}%
_{3},\mathbf{T}_{3}^{\star_{2}},\mathbf{T}_{3}^{\star_{1}\star_{2}}\right]
&  =\mathbf{I},\\
\left[  \mathbf{T}_{1},\mathbf{T}_{1}^{\star_{3}},\mathbf{T}_{1}^{\star
_{2}\star_{3}}\right]  +_{3}\left[  \mathbf{T}_{2},\mathbf{T}_{2}^{\star_{3}%
},\mathbf{T}_{2}^{\star_{2}\star_{3}}\right]  +_{3}\left[  \mathbf{T}%
_{3},\mathbf{T}_{3}^{\star_{3}},\mathbf{T}_{3}^{\star_{2}\star_{3}}\right]
&  =\mathbf{I},
\end{align}
which represent the ternary Cuntz algebra $p\mathcal{O}_{3\mid3,3}$.
\end{example}

\begin{example}
In the case where the inner pairing is semicommutative (\ref{sm}), we can
eliminate the multistar $\left(  \star_{3}\right)  $ by (\ref{t3}) and
represent the \textit{two-multistar ternary analog of the Cuntz algebra}
$p\mathcal{O}_{3\mid3,3}$ by%
\begin{equation}%
\begin{array}
[c]{c}%
\left[  \mathbf{T}_{i}^{\star_{1}},\mathbf{T}_{i}^{\star_{2}},\mathbf{T}%
_{i}\right]  =\mathbf{I},\\
\left[  \mathbf{T}_{i}^{\star_{2}},\mathbf{T}_{i}^{\star_{1}\star_{2}%
},\mathbf{T}_{i}\right]  =\mathbf{I},\\
\left[  \mathbf{T}_{i}^{\star_{1}\star_{2}},\mathbf{T}_{i}^{\star_{2}%
},\mathbf{T}_{i}\right]  =\mathbf{I},\\
i=1,2,3,
\end{array}
\ \ \
\begin{array}
[c]{c}%
\left[  \mathbf{T}_{i}^{\star_{1}},\mathbf{T}_{j}^{\star_{2}},\mathbf{T}%
_{k}\right]  =\mathbf{Z},\\
\left[  \mathbf{T}_{i}^{\star_{1}},\mathbf{T}_{j}^{\star_{1}\star_{2}%
},\mathbf{T}_{k}\right]  =\mathbf{Z},\\
\left[  \mathbf{T}_{i}^{\star_{1}\star_{2}},\mathbf{T}_{j}^{\star_{2}%
},\mathbf{T}_{k}\right]  =\mathbf{Z},\\
i,j,k=1,2,3,\ \ i\neq j\neq k,
\end{array}
\end{equation}%
\begin{align}
\left[  \mathbf{T}_{1},\mathbf{T}_{1}^{\star_{1}},\mathbf{T}_{1}^{\star_{1}%
}\right]  +_{3}\left[  \mathbf{T}_{2},\mathbf{T}_{2}^{\star_{1}}%
,\mathbf{T}_{2}^{\star_{1}}\right]  +_{3}\left[  \mathbf{T}_{3},\mathbf{T}%
_{3}^{\star_{1}},\mathbf{T}_{3}^{\star_{1}}\right]   &  =\mathbf{I},\\
\left[  \mathbf{T}_{1},\mathbf{T}_{1}^{\star_{2}},\mathbf{T}_{1}^{\star
_{1}\star_{2}}\right]  +_{3}\left[  \mathbf{T}_{2},\mathbf{T}_{2}^{\star_{2}%
},\mathbf{T}_{2}^{\star_{1}\star_{2}}\right]  +_{3}\left[  \mathbf{T}%
_{3},\mathbf{T}_{3}^{\star_{2}},\mathbf{T}_{3}^{\star_{1}\star_{2}}\right]
&  =\mathbf{I},\\
\left[  \mathbf{T}_{1},\mathbf{T}_{1}^{\star_{1}\star_{2}},\mathbf{T}%
_{1}^{\star_{2}}\right]  +_{3}\left[  \mathbf{T}_{2},\mathbf{T}_{2}^{\star
_{1}\star_{2}},\mathbf{T}_{2}^{\star_{2}}\right]  +_{3}\left[  \mathbf{T}%
_{3},\mathbf{T}_{3}^{\star_{1}\star_{2}},\mathbf{T}_{3}^{\star_{2}}\right]
&  =\mathbf{I}.
\end{align}

\end{example}

\section{\label{sec-residue}Congruence classes as polyadic rings}

Here we will show that the inner structure of the residue classes (congruence
classes) over integers is naturally described by polyadic rings
\cite{cel,cro2,lee/but}, and then study some special properties of them.

Denote a residue class (congruence class) of an integer $a$, modulo $b$
by\footnote{We use for the residue class the notation $\left[  \left[
a\right]  \right]  _{b}$ , because the standard notations by one square
bracket $\left[  a\right]  _{b}$ or $\bar{a}_{b}$ are busy by the $n$-ary
operations and querelements, respectively.}%
\begin{equation}
\left[  \left[  a\right]  \right]  _{b}=\left\{  \left\{  a+bk\right\}  \mid
k\in\mathbb{Z},\ \ a\in\mathbb{Z}_{+},\ b\in\mathbb{N},\ 0\leq a\leq
b-1\right\}  . \label{ab}%
\end{equation}
A representative element of the class $\left[  \left[  a\right]  \right]
_{b}$ will be denoted by $x_{k}=x_{k}^{\left(  a,b\right)  }=a+bk$. Here we do
not consider the addition and multiplication of the residue classes
(congruence classes). Instead, we consider the \textsl{fixed} congruence class
$\left[  \left[  a\right]  \right]  _{b}$, and note that, for arbitrary $a$
and $b$, it is \textsl{not closed} under binary operations. However, it can be
\textsl{closed with respect to} \textsl{polyadic operations}.

\subsection{Polyadic ring on integers}

Let us introduce the $m$-ary addition and $n$-ary multiplication of
representatives of the fixed congruence class $\left[  \left[  a\right]
\right]  _{b}$ by%
\begin{align}
\nu_{m}\left[  x_{k_{1}},x_{k_{2}},\ldots,x_{k_{m}}\right]   &  =x_{k_{1}%
}+x_{k_{2}}+\ldots+x_{k_{m}},\label{nu}\\
\mu_{n}\left[  x_{k_{1}},x_{k_{2}},\ldots,x_{k_{n}}\right]   &  =x_{k_{1}%
}x_{k_{2}}\ldots x_{k_{n}},\ \ \ x_{k_{i}}\in\left[  \left[  a\right]
\right]  _{b},\ k_{i}\in\mathbb{Z}, \label{mu}%
\end{align}
where on the r.h.s. the operations are the ordinary binary addition and binary
multiplication in $\mathbb{Z}$.

\begin{remark}
The polyadic operations (\ref{nu})--(\ref{mu}) are \textsl{not derived} (see,
e.g., \cite{gla/mic84,mic88}), because on the set $\left\{  x_{k_{i}}\right\}
$ one cannot define the \textsl{binary} semigroup structure with respect to
ordinary addition and multiplication. Derived polyadic rings which consist of
the repeated binary sums and binary products were considered in \cite{lee/but}.
\end{remark}

\begin{lemma}
\label{lem-mgr}In case%
\begin{equation}
\left(  m-1\right)  \dfrac{a}{b}=I^{\left(  m\right)  }\left(  a,b\right)
=I=\operatorname{integer} \label{m1}%
\end{equation}
the algebraic structure $\left\langle \left[  \left[  a\right]  \right]
_{b}\mid\nu_{m}\right\rangle $ is a commutative $m$-ary group.
\end{lemma}

\begin{proof}
The closure of the operation (\ref{nu}) can be written as $x_{k_{1}}+x_{k_{2}%
}+\ldots+x_{k_{m}}=x_{k_{0}}$, or $ma+b\left(  k_{1}+k_{2}+\ldots
+k_{m}\right)  =a+bk_{0}$, and then $k_{0}=\left(  m-1\right)  a/b+\left(
k_{1}+k_{2}+\ldots+k_{m}\right)  $, from (\ref{m1}). The (total) associativity
and commutativity of $\nu_{m}$ follows from those of the addition in the
binary $\mathbb{Z}$. Each element $x_{k}$ has its \textsl{unique querelement}
$\tilde{x}=x_{\tilde{k}}$ determined by the equation $\left(  m-1\right)
x_{k}+x_{\tilde{k}}=x_{k}$, which (uniquely, for any $k\in\mathbb{Z}$) gives%
\begin{equation}
\tilde{k}=bk\left(  2-m\right)  -\left(  m-1\right)  \dfrac{a}{b}\ .
\end{equation}
Thus, each element is \textquotedblleft querable\textquotedblright\ (polyadic
invertible), and so $\left\langle \left[  \left[  a\right]  \right]  _{b}%
\mid\nu_{m}\right\rangle $ is a $m$-ary group.
\end{proof}

\begin{example}
For $a=2$, $b=7$ we have $8$-ary group, and the querelement of $x_{k}$ is
$\tilde{x}=x_{\left(  -2-12k\right)  }$.
\end{example}

\begin{proposition}
\label{prop-mary}The $m$-ary commutative group $\left\langle \left[  \left[
a\right]  \right]  _{b}\mid\nu_{m}\right\rangle $:

\begin{enumerate}
\item has an infinite number of neutral sequences for each element;

\item if $a\neq0$, it has no \textquotedblleft unit\textquotedblright\ (which
is actually zero, because $\nu_{m}$ plays the role of \textquotedblleft
addition\textquotedblright);

\item in case of the zero congruence class $\left[  \left[  0\right]  \right]
_{b}$ the zero is $x_{k}=0$.
\end{enumerate}
\end{proposition}

\begin{proof}
\noindent

\begin{enumerate}
\item The (additive) neutral sequence $\mathbf{\tilde{n}}_{m-1}$ of the length
$\left(  m-1\right)  $ is defined by $\nu_{m}\left[  \mathbf{\tilde{n}}%
_{m-1},x_{k}\right]  =x_{k}$. Using (\ref{nu}), we take $\mathbf{\tilde{n}%
}_{m-1}=x_{k_{1}}+x_{k_{2}}+\ldots+x_{k_{m-1}}=0$ and obtain the equation
\begin{equation}
\left(  m-1\right)  a+b\left(  k_{1}+k_{2}+\ldots+k_{m-1}\right)  =0.
\label{mak}%
\end{equation}
Because of (\ref{m1}), we obtain%
\begin{equation}
k_{1}+k_{2}+\ldots+k_{m-1}=-I^{\left(  m\right)  }\left(  a,b\right)  ,
\end{equation}
and so there is an infinite number of sums satisfying this condition.

\item The polyadic \textquotedblleft unit\textquotedblright/zero $z=x_{k_{0}%
}=a+bk_{0}$ satisfies $\nu_{m}\left[  \overset{m-1}{\overbrace{z,z,\ldots,z}%
},x_{k}\right]  =x_{k}$ for all $x_{k}\in\left[  \left[  a\right]  \right]
_{b}$ (the neutral sequence $\mathbf{\tilde{n}}_{m-1}$ consists of one element
$z$ only), which gives $\left(  m-1\right)  \left(  a+bk_{0}\right)  =0$
having no solutions with $a\neq0$, since $a<b$.

\item In the case $a=0$, the only solution is $z=x_{k=0}=0$.
\end{enumerate}
\end{proof}

\begin{example}
In case $a=1$, $b=2$ we have $m=3$ and $I^{\left(  3\right)  }\left(
1,2\right)  =1$, and so from (\ref{mak}) we get $k_{1}+k_{2}=-1$, thus the
infinite number of neutral sequences consists of 2 elements $\mathbf{\tilde
{n}}_{2}=x_{k}+x_{-1-k}$, with arbitrary $k\in\mathbb{Z}$.
\end{example}

\begin{lemma}
\label{lem-nsemigr}If%
\begin{equation}
\dfrac{a^{n}-a}{b}=J^{\left(  n\right)  }\left(  a,b\right)
=J=\operatorname{integer}, \label{an}%
\end{equation}
then $\left\langle \left[  \left[  a\right]  \right]  _{b}\mid\mu
_{n}\right\rangle $ is a commutative $n$-ary semigroup.
\end{lemma}

\begin{proof}
It follows from (\ref{mu}), that the closeness of the operation $\mu_{n}$ is
$x_{k_{1}}x_{k_{2}}\ldots x_{k_{n}}=x_{k_{0}}$, which can be written as
$a^{n}+b\left(  \operatorname{integer}\right)  =a+bk_{0}$ leading to
(\ref{an}). The (total) associativity and commutativity of $\mu_{n}$ follows
from those of the multiplication in $\mathbb{Z}$.
\end{proof}

\begin{definition}
A unique pair of integers $\left(  I,J\right)  $ is called a
(\textit{polyadic})\textit{ shape invariants} of the congruence class $\left[
\left[  a\right]  \right]  _{b}$.
\end{definition}

\begin{theorem}
\label{theor-rz}The algebraic structure of the \textsl{fixed} congruence class
$\left[  \left[  a\right]  \right]  _{b}$ is a polyadic $\left(  m,n\right)
$-ring%
\begin{equation}
\mathcal{R}_{m,n}^{\left[  a,b\right]  }=\left\langle \left[  \left[
a\right]  \right]  _{b}\mid\nu_{m},\mu_{n}\right\rangle , \label{rz}%
\end{equation}
where the arities $m$ and $n$ are \textsl{minimal} positive integers (more or
equal $2$), for which the congruences%
\begin{align}
ma  &  \equiv a\left(  \operatorname{mod}b\right)  ,\label{maa}\\
a^{n}  &  \equiv a\left(  \operatorname{mod}b\right)  \label{ana}%
\end{align}
take place \textsl{simultaneously}, fixating its polyadic shape invariants
$\left(  I,J\right)  $.
\end{theorem}

\begin{proof}
By Lemma \ref{lem-mgr}, \ref{lem-nsemigr} the set $\left[  \left[  a\right]
\right]  _{b}$ is a $m$-ary group with respect to \textquotedblleft$m$-ary
addition\textquotedblright\ $\nu_{m}$ and a $n$-ary semigroup with respect to
\textquotedblleft$n$-ary multiplication\textquotedblright\ $\mu_{n}$, while
the polyadic distributivity (\ref{dis1})--(\ref{dis3}) follows from (\ref{nu})
and (\ref{mu}) and the binary distributivity in $\mathbb{Z}$.
\end{proof}

\begin{remark}
For a fixed $b\geq2$ there are $b$ congruence classes $\left[  \left[
a\right]  \right]  _{b}$, $0\leq a\leq b-1$, and therefore exactly $b$
corresponding polyadic $\left(  m,n\right)  $-rings $\mathcal{R}%
_{m,n}^{\left[  a,b\right]  }$, each of them is \textsl{infinite-dimensional}.
\end{remark}

\begin{corollary}
In case $\gcd\left(  a,b\right)  =1$ and $b$ is prime, there exists the
solution $n=b$.
\end{corollary}

\begin{proof}
Follows from (\ref{ana}) and Fermat's little theorem.
\end{proof}

\begin{remark}
\label{rem-a0}We exclude from consideration the zero congruence class $\left[
\left[  0\right]  \right]  _{b}$, because the arities of operations $\nu_{m}$
and $\mu_{n}$ cannot be fixed up by (\ref{maa})--(\ref{ana}) becoming
identities for any $m$ and $n$. Since the arities are uncertain, their minimal
values can be chosen $m=n=2$, and therefore, it follows from (\ref{nu}) and
(\ref{mu}), $\mathcal{R}_{2,2}^{\left[  0,b\right]  }=\mathbb{Z}$. Thus, in
what follows we always imply that $a\neq0$ (without using a special notation,
e.g. $\mathcal{R}^{\ast}$, etc.).
\end{remark}

In \textsc{Table \ref{T1}} we present the allowed (by (\ref{maa}%
)--(\ref{ana})) arities of the polyadic ring $\mathcal{R}_{m,n}^{\left[
a,b\right]  }$ and the corresponding polyadic shape invariants $\left(
I,J\right)  $ for $b\leq10$.

Let us investigate the properties of $\mathcal{R}_{m,n}^{\left[  a,b\right]
}$ in more detail. First, we consider equal arity polyadic rings and find the
relation between the corresponding congruence classes.

\begin{proposition}
The residue (congruence) classes $\left[  \left[  a\right]  \right]  _{b}$ and
$\left[  \left[  a^{\prime}\right]  \right]  _{b^{\prime}}$ which are
described by the polyadic rings of the same arities $\mathcal{R}%
_{m,n}^{\left[  a,b\right]  }$ and $\mathcal{R}_{m,n}^{\left[  a^{\prime
},b^{\prime}\right]  }$ are related by%
\begin{align}
\dfrac{b^{\prime}I^{\prime}}{a^{\prime}} &  =\dfrac{bI}{a},\label{bi}\\
a^{\prime}+b^{\prime}J^{\prime} &  =\left(  a+bJ\right)  ^{\log_{a}a^{\prime}%
}.\label{aj}%
\end{align}

\end{proposition}

\begin{proof}
Follows from (\ref{m1}) and (\ref{an}).
\end{proof}

For instance, in \textsc{Table \ref{T1}} the congruence classes $\left[
\left[  2\right]  \right]  _{5}$, $\left[  \left[  3\right]  \right]  _{5}$,
$\left[  \left[  2\right]  \right]  _{10}$, and $\left[  \left[  8\right]
\right]  _{10}$ are $\left(  6,5\right)  $-rings. If, in addition,
$a=a^{\prime}$, then the polyadic shapes satisfy%
\begin{equation}
\dfrac{I}{J}=\dfrac{I^{\prime}}{J^{\prime}}.
\end{equation}

\subsection{Limiting cases}

The limiting cases $a\equiv\pm1\left(  \operatorname{mod}b\right)  $ are
described by

\begin{corollary}
The polyadic ring of the fixed congruence class $\left[  \left[  a\right]
\right]  _{b}$ is: \textbf{1)} multiplicative binary, if $a=1$; \textbf{2)}
multiplicative ternary, if $a=b-1$; \textbf{3)} additive $\left(  b+1\right)
$-ary in both cases. That is, the limiting cases contain the rings
$\mathcal{R}_{b+1,2}^{\left[  1,b\right]  }$ and $\mathcal{R}_{b+1,3}^{\left[
b-1,b\right]  }$, respectively. They correspond to the first row and the main
diagonal of \textsc{Table \ref{T1}}. Their intersection consists of the
$\left(  3,2\right)  $-ring $\mathcal{R}_{3,2}^{\left[  1,2\right]  }$.
\end{corollary}

\begin{definition}
The congruence classes $\left[  \left[  1\right]  \right]  _{b}$ and $\left[
\left[  b-1\right]  \right]  _{b}$ are called the \textit{limiting classes},
and the corresponding polyadic rings are named the \textit{limiting polyadic
rings} of a fixed congruence class.
\end{definition}

\begin{proposition}
\label{prop-lim}In the limiting cases $a=1$ and $a=b-1$ the $n$-ary semigroup
$\left\langle \left[  \left[  a\right]  \right]  _{b}\mid\mu_{n}\right\rangle
$:

\begin{enumerate}
\item has the neutral sequences of the form $\mathbf{\bar{n}}_{n-1}=x_{k_{1}%
}x_{k_{2}}\ldots x_{k_{n-1}}=1$, where $x_{k_{i}}=\pm1$;

\item has a) the unit $e=x_{k=1}=1$, for the limiting class $\left[  \left[
1\right]  \right]  _{b}$, b) the unit $e^{-}=x_{k=-1}=-1$, if $n$ is odd, for
$\left[  \left[  b-1\right]  \right]  _{b}$, c) the class $\left[  \left[
1\right]  \right]  _{2}$ contains both polyadic units $e$ and $e^{-}$;

\item has the set of \textquotedblleft querable\textquotedblright\ (polyadic
invertible) elements which consist of $\bar{x}=x_{\bar{k}}=\pm1$;

\item has in the \textquotedblleft intersecting\textquotedblright\ case $a=1$,
$b=2$ and $n=2$ the binary subgroup $\mathbb{Z}_{2}=\left\{  1,-1\right\}  $,
while other elements have no inverses.
\end{enumerate}
\end{proposition}

\begin{proof}
\noindent

\begin{enumerate}
\item The (multiplicative) neutral sequence $\mathbf{\bar{n}}_{n-1}$ of length
$\left(  n-1\right)  $ is defined by $\mu_{n}\left[  \mathbf{\bar{n}}%
_{n-1},x_{k}\right]  =x_{k}$. It follows from (\ref{mu}) and cancellativity in
$\mathbb{Z}$, that $\mathbf{\bar{n}}_{n-1}=x_{k_{1}}x_{k_{2}}\ldots
x_{k_{n-1}}=1$ which is
\begin{equation}
\left(  a+bk_{1}\right)  \left(  a+bk_{2}\right)  \ldots\left(  a+bk_{n-1}%
\right)  =1. \label{abk}%
\end{equation}
The solution of this equation in integers is the following: a) all multipliers
are $a+bk_{i}=1$, $i=1,\ldots,n-1$; b) an even number of multipliers can be
$a+bk_{i}=-1$, while the others are $1$.

\item If the polyadic unit $e=x_{k_{1}}=a+bk_{1}$ exists, it should satisfy
$\mu_{m}\left[  \overset{n-1}{\overbrace{e,e,\ldots,e}},x_{k}\right]  =x_{k}$
$\forall x_{k}\in\left\langle \left[  \left[  a\right]  \right]  _{b}\mid
\mu_{n}\right\rangle $, such that the neutral sequence $\mathbf{\bar{n}}%
_{n-1}$ consists of one element $e$ only, and this leads to $\left(
a+bk_{1}\right)  ^{n-1}=1$. For any $n$ this equation has the solution
$a+bk_{1}=1$, which uniquely gives $a=1$ and $k_{1}=0$, thus $e=x_{k_{1}=0}%
=1$. If $n$ is odd, then there exists a \textquotedblleft negative
unit\textquotedblright\ $e^{-}=x_{k_{1}=-1}=-1$, such that $a+bk_{1}=-1$,
which can be uniquely solved by $k_{1}=-1$ and $a=b-1$. The neutral sequence
becomes $\mathbf{\bar{n}}_{n-1}=\overset{n-1}{\overbrace{e^{-},e^{-}%
,\ldots,e^{-}}}=1$, as a product of an even number of $e^{-}=-1$. The
intersection of limiting classes consists of one class $\left[  \left[
1\right]  \right]  _{2}$, and therefore it contains both polyadic units $e$
and $e^{-}$.

\item An element $x_{k}$ in $\left\langle \left[  \left[  a\right]  \right]
_{b}\mid\mu_{n}\right\rangle $ is \textquotedblleft querable\textquotedblright%
, if there exists its querelement $\bar{x}=x_{\bar{k}}$ such that $\mu
_{n}\left[  \overset{n-1}{\overbrace{x_{k},x_{k},\ldots,x_{k}}},\bar
{x}\right]  =x_{k}$. Using (\ref{mu}) and the cancellativity in $\mathbb{Z}$,
we obtain the equation $\left(  a+bk\right)  ^{n-2}\left(  a+b\bar{k}\right)
=1$, which in integers has 2 solutions: a) $\left(  a+bk\right)  ^{n-2}=1$ and
$\left(  a+b\bar{k}\right)  =1$, the last relation fixes up the class $\left[
\left[  1\right]  \right]  _{b}$, and the arity of multiplication $n=2$, and
therefore the first relation is valid for all elements in the class, each of
them has the same querelement $\bar{x}=1$. This means that all elements in
$\left[  \left[  1\right]  \right]  _{b}$ are \textquotedblleft
querable\textquotedblright, but only one element $x=1$ has an inverse, which
is also $1$; b) $\left(  a+bk\right)  ^{n-2}=-1$ and $\left(  a+b\bar
{k}\right)  =-1$. The second relation fixes the class $\left[  \left[
b-1\right]  \right]  _{b}$, and from the first relation we conclude that the
arity $n$ should be odd. In this case only one element $-1$ is
\textquotedblleft querable\textquotedblright, which has $\bar{x}=-1$, as a querelement.

\item The \textquotedblleft intersecting\textquotedblright\ class $\left[
\left[  1\right]  \right]  _{2}$ contains 2 \textquotedblleft
querable\textquotedblright\ elements $\pm1$ which coincide with their
inverses, which means that $\left\{  +1,-1\right\}  $ is a binary subgroup
(that is $\mathbb{Z}_{2}$) of the binary semigroup $\left\langle \left[
\left[  1\right]  \right]  _{2}\mid\mu_{2}\right\rangle $.
\end{enumerate}
\end{proof}

\begin{corollary}
\label{cor-nonlim}In the non-limiting cases $a\neq1,b-1$, the $n$-ary
semigroup $\left\langle \left[  \left[  a\right]  \right]  _{b}\mid\mu
_{n}\right\rangle $ contains no \textquotedblleft querable\textquotedblright%
\ (polyadic invertible) elements at all.
\end{corollary}

\begin{proof}
It follows from $\left(  a+bk\right)  \neq\pm1$ for any $k\in\mathbb{Z}$ or
$a\neq\pm1\left(  \operatorname{mod}b\right)  $.
\end{proof}

\begin{table}[h]
\caption{The polyadic ring $\mathcal{R}_{m,n}^{\mathbb{Z}\left(  a,b\right)
}$ of the fixed residue class $\left[  \left[  a\right]  \right]  _{b}$: arity
shape. }%
\label{T1}
\begin{center}
{\tiny
\begin{tabular}
[c]{||c||c|c|c|c|c|c|c|c|c||}\hline\hline
$a\setminus b$ & 2 & 3 & 4 & 5 & 6 & 7 & 8 & 9 & 10\\\hline\hline
1 & $%
\begin{array}
[c]{c}%
m=\mathbf{3}\\
n=\mathbf{2}\\
I=1\\
J=0
\end{array}
$ & $%
\begin{array}
[c]{c}%
m=\mathbf{4}\\
n=\mathbf{2}\\
I=1\\
J=0
\end{array}
$ & $%
\begin{array}
[c]{c}%
m=\mathbf{5}\\
n=\mathbf{2}\\
I=1\\
J=0
\end{array}
$ & $%
\begin{array}
[c]{c}%
m=\mathbf{6}\\
n=\mathbf{2}\\
I=1\\
J=0
\end{array}
$ & $%
\begin{array}
[c]{c}%
m=\mathbf{7}\\
n=\mathbf{2}\\
I=1\\
J=0
\end{array}
$ & $%
\begin{array}
[c]{c}%
m=\mathbf{8}\\
n=\mathbf{2}\\
I=1\\
J=0
\end{array}
$ & $%
\begin{array}
[c]{c}%
m=\mathbf{9}\\
n=\mathbf{2}\\
I=1\\
J=0
\end{array}
$ & $%
\begin{array}
[c]{c}%
m=\mathbf{10}\\
n=\mathbf{2}\\
I=1\\
J=0
\end{array}
$ & $%
\begin{array}
[c]{c}%
m=\mathbf{11}\\
n=\mathbf{2}\\
I=1\\
J=0
\end{array}
$\\\hline
2 &  & $%
\begin{array}
[c]{c}%
m=\mathbf{4}\\
n=\mathbf{3}\\
I=2\\
J=2
\end{array}
$ &  & $%
\begin{array}
[c]{c}%
m=\mathbf{6}\\
n=\mathbf{5}\\
I=2\\
J=6
\end{array}
$ & $%
\begin{array}
[c]{c}%
m=\mathbf{4}\\
n=\mathbf{3}\\
I=1\\
J=1
\end{array}
$ & $%
\begin{array}
[c]{c}%
m=\mathbf{8}\\
n=\mathbf{4}\\
I=2\\
J=2
\end{array}
$ &  & $%
\begin{array}
[c]{c}%
m=\mathbf{10}\\
n=\mathbf{7}\\
I=2\\
J=14
\end{array}
$ & $%
\begin{array}
[c]{c}%
m=\mathbf{6}\\
n=\mathbf{5}\\
I=1\\
J=3
\end{array}
$\\\hline
3 &  &  & $%
\begin{array}
[c]{c}%
m=\mathbf{5}\\
n=\mathbf{3}\\
I=3\\
J=6
\end{array}
$ & $%
\begin{array}
[c]{c}%
m=\mathbf{6}\\
n=\mathbf{5}\\
I=3\\
J=48
\end{array}
$ & $%
\begin{array}
[c]{c}%
m=\mathbf{3}\\
n=\mathbf{2}\\
I=1\\
J=1
\end{array}
$ & $%
\begin{array}
[c]{c}%
m=\mathbf{8}\\
n=\mathbf{7}\\
I=3\\
J=312
\end{array}
$ & $%
\begin{array}
[c]{c}%
m=\mathbf{9}\\
n=\mathbf{3}\\
I=3\\
J=3
\end{array}
$ &  & $%
\begin{array}
[c]{c}%
m=\mathbf{11}\\
n=\mathbf{5}\\
I=3\\
J=24
\end{array}
$\\\hline
4 &  &  &  & $%
\begin{array}
[c]{c}%
m=\mathbf{6}\\
n=\mathbf{3}\\
I=4\\
J=12
\end{array}
$ & $%
\begin{array}
[c]{c}%
m=\mathbf{4}\\
n=\mathbf{2}\\
I=2\\
J=2
\end{array}
$ & $%
\begin{array}
[c]{c}%
m=\mathbf{8}\\
n=\mathbf{4}\\
I=4\\
J=36
\end{array}
$ &  & $%
\begin{array}
[c]{c}%
m=\mathbf{10}\\
n=\mathbf{4}\\
I=4\\
J=28
\end{array}
$ & $%
\begin{array}
[c]{c}%
m=\mathbf{6}\\
n=\mathbf{3}\\
I=2\\
J=6
\end{array}
$\\\hline
5 &  &  &  &  & $%
\begin{array}
[c]{c}%
m=\mathbf{7}\\
n=\mathbf{3}\\
I=5\\
J=20
\end{array}
$ & $%
\begin{array}
[c]{c}%
m=\mathbf{8}\\
n=\mathbf{7}\\
I=5\\
J=11160
\end{array}
$ & $%
\begin{array}
[c]{c}%
m=\mathbf{9}\\
n=\mathbf{3}\\
I=5\\
J=15
\end{array}
$ & $%
\begin{array}
[c]{c}%
m=\mathbf{10}\\
n=\mathbf{7}\\
I=5\\
J=8680
\end{array}
$ & $%
\begin{array}
[c]{c}%
m=\mathbf{3}\\
n=\mathbf{2}\\
I=1\\
J=2
\end{array}
$\\\hline
6 &  &  &  &  &  & $%
\begin{array}
[c]{c}%
m=\mathbf{8}\\
n=\mathbf{3}\\
I=6\\
J=30
\end{array}
$ &  &  & $%
\begin{array}
[c]{c}%
m=\mathbf{6}\\
n=\mathbf{2}\\
I=3\\
J=3
\end{array}
$\\\hline
7 &  &  &  &  &  &  & $%
\begin{array}
[c]{c}%
m=\mathbf{9}\\
n=\mathbf{3}\\
I=7\\
J=42
\end{array}
$ & $%
\begin{array}
[c]{c}%
m=\mathbf{10}\\
n=\mathbf{4}\\
I=7\\
J=266
\end{array}
$ & $%
\begin{array}
[c]{c}%
m=\mathbf{11}\\
n=\mathbf{5}\\
I=7\\
J=1680
\end{array}
$\\\hline
8 &  &  &  &  &  &  &  & $%
\begin{array}
[c]{c}%
m=\mathbf{10}\\
n=\mathbf{3}\\
I=8\\
J=56
\end{array}
$ & $%
\begin{array}
[c]{c}%
m=\mathbf{6}\\
n=\mathbf{5}\\
I=4\\
J=3276
\end{array}
$\\\hline
9 &  &  &  &  &  &  &  &  & $%
\begin{array}
[c]{c}%
m=\mathbf{11}\\
n=\mathbf{3}\\
I=9\\
J=72
\end{array}
$\\\hline\hline
\end{tabular}
}
\end{center}
\end{table}

Based on the above statements, consider in the properties of the polyadic
rings $\mathcal{R}_{m,n}^{\left[  a,b\right]  }$ ($a\neq0$) describing
non-zero congruence classes (see \emph{Remark} \ref{rem-a0}).

\begin{definition}
\label{def-polint}The infinite set of representatives of the congruence
(residue) class $\left[  \left[  a\right]  \right]  _{b}$ having fixed arities
and form the $\left(  m,n\right)  $-ring $\mathcal{R}_{m,n}^{\left[
a,b\right]  }$ is called the set of (\textit{polyadic})\textit{ }$\left(
m,n\right)  $-\textit{integers (numbers)} and denoted $\mathbb{Z}_{\left(
m,n\right)  }$.
\end{definition}

Just obviously, for ordinary integers $\mathbb{Z=Z}_{\left(  2,2\right)  }$,
and they form the binary ring $\mathcal{R}_{2,2}^{\left[  0,1\right]  }$.

\begin{proposition}
\label{prop-dom}The polyadic ring $\mathcal{R}_{m,n}^{\left[  a,b\right]  }$
is a $\left(  m,n\right)  $-integral domain.
\end{proposition}

\begin{proof}
It follows from the definitions (\ref{nu})--(\ref{mu}), the condition $a\neq
0$, and commutativity and cancellativity in $\mathbb{Z}$.
\end{proof}

\begin{lemma}
There are no such congruence classes which can be described by polyadic
$\left(  m,n\right)  $-field.
\end{lemma}

\begin{proof}
Follows from \textbf{Proposition \ref{prop-lim}} and \textbf{Corollary
\ref{cor-nonlim}}.
\end{proof}

This statement for the limiting case $\left[  \left[  1\right]  \right]  _{2}$
appeared in \cite{dup/wer}, while studying the ideal structure of the
corresponding $\left(  3,2\right)  $-ring.

\begin{proposition}
In the limiting case $a=1$ the polyadic ring $\mathcal{R}_{b+1,2}^{\left[
1,b\right]  }\mathcal{\ }$can be embedded into a $\left(  b+1,2\right)  $-ary field.
\end{proposition}

\begin{proof}
Because the polyadic ring $\mathcal{R}_{b+1,2}^{\left[  1,b\right]  }$ of the
congruence class $\left[  \left[  1\right]  \right]  _{b}$ is an $\left(
b+1,2\right)  $-integral domain by \textbf{Proposition \ref{prop-dom}}, we can
construct in a standard way the correspondent $\left(  b+1,2\right)
$-quotient ring which is a $\left(  b+1,2\right)  $-ary field up to
isomorphism, as was shown in \cite{cro/tim}. By analogy, it can be called the
field of \textsl{polyadic rational numbers} which have the form%
\begin{equation}
x=\dfrac{1+bk_{1}}{1+bk_{2}},\ \ \ k_{i}\in\mathbb{Z}.
\end{equation}

Indeed, they form a $\left(  b+1,2\right)  $-field, because each element has
its inverse under multiplication (which is obvious) and additively
\textquotedblleft querable\textquotedblright, such that the equation for the
querelement $\bar{x}$ becomes $\nu_{b+1}\left[  \overset{b}{\overbrace
{x,x,\ldots,x}},\bar{x}\right]  =x$ which can be solved for any $x$, giving
uniquely $\bar{x}=-\left(  b-1\right)  \dfrac{1+bk_{1}}{1+bk_{2}}$.
\end{proof}

The introduced polyadic inner structure of the residue (congruence) classes
allows us to extend various number theory problems by considering the polyadic
$\left(  m,n\right)  $-integers $\mathbb{Z}_{\left(  m,n\right)  }$ instead of
$\mathbb{Z}$.

\section{Equal sums of like powers Diophantine equation over polyadic
integers}

First, recall the standard binary version of the equal sums of like powers
Diophantine equation \cite{lan/par/sel,ekl}. Take the fixed non-negative
integers $p,q,l\in\mathbb{N}^{0}$, $p\leq q$, and the positive integer
unknowns $u_{i},v_{j}\in\mathbb{Z}_{+}$, $i=1,\ldots p+1$, $j=1,1,\ldots q+1$,
then the Diophantine equation is%
\begin{equation}
\sum\limits_{i=1}^{p+1}u_{i}^{l+1}=\sum\limits_{j=1}^{q+1}v_{j}^{l+1}.
\label{uu}%
\end{equation}

The trivial case, when $u_{i}=0$, $v_{j}=0$, for all $i,j$ is not considered.
We mark the solutions of (\ref{uu}) by the triple $\left(  l\mid p,q\right)
_{r}$ showing quantity of operations\footnote{In the binary case, the
solutions of (\ref{uu}) are usually denoted by $\left(  l+1\mid
p+1,q+1\right)  _{r}$, which shows the number of summands on both sides and
powers of elements \cite{lan/par/sel}. But in the polyadic case (see below),
the number of summands and powers do not coincide with $l+1$, $p+1$, $q+1$, at
all.}, where $r$ (if it is used) is the order of the solution (ranked by the
value of the sum) and the unknowns $u_{i},v_{j}$ are placed in ascending order
$u_{i}\leq u_{i+1},v_{j}\leq v_{j+1}$.

Let us recall the \textsl{Tarry-Escott problem} (or multigrades problem)
\cite{dor/bro}: to find the solutions to (\ref{uu}) for an equal number of
summands on both sides of $p=q$ and $s$ equations simultaneously, such that
$l=0,\ldots,s$. Known solutions exist for powers until $s=10$, which are
bounded such that $s\leq p$ (in our notations), see, also, \cite{ngu2016}. The
solutions with highest powers $s=p$ are the most interesting and called the
\textsl{ideal solutions} \cite{bor2002}.

\begin{theorem}
[Frolov \cite{fro1889}]\label{theor-fro}If the set of $s$ Diophantine
equations (\ref{uu}) with $p=q$ for $l=0,\ldots,s$ has a solution $\left\{
u_{i},v_{i},i=1,\ldots p+1\right\}  $, then it has the solution $\left\{
a+bu_{i},a+bv_{i},i=1,\ldots p+1\right\}  $, where $a,b\in\mathbb{Z}$ are
arbitrary and fixed.
\end{theorem}

In the simplest case $\left(  1\mid0,1\right)  $, one term in l.h.s., one
addition on the r.h.s. and one multiplication, the (coprime) positive numbers
satisfying (\ref{uu}) are called a (primitive) \textsl{Pythagorean triple}.
For the \textsl{Fermat's triple} $\left(  l\mid0,1\right)  $ with one addition
on the r.h.s. and more than one multiplication $l\geq2$, there are no
solutions of (\ref{uu}) , which is known as \textsl{Fermat's last theorem}
proved in \cite{wil95}. There are many solutions known with more than one
addition on both sides, where the highest number of multiplications till now
is $31$ (S.~Chase, 2012).

Before generalizing (\ref{uu}) for polyadic case we note the following.

\begin{remark}
\label{rem-power}The notations in (\ref{uu}) are chosen in such a way that $p$
and $q$ are \textsl{numbers of binary additions} on both sides, while $l$ is
the \textsl{number of binary multiplications} in each term, which is natural
for using polyadic powers \cite{dup2012}.
\end{remark}

\subsection{Polyadic analog of the Lander-Parkin-Selfridge conjecture}

In \cite{lan/par/sel}, a generalization of Fermat's last theorem was
conjectured, that the solutions of (\ref{uu}) exist for small powers only,
which can be formulated in terms of the numbers of operations as

\begin{conjecture}
[Lander-Parkin-Selfridge \cite{lan/par/sel}]\label{conj-lps}There exist
solutions of (\ref{uu}) in positive integers, if the number of multiplications
is less than or equal than the total number of additions plus one%
\begin{equation}
3\leq l\leq l_{LSP}=p+q+1,
\end{equation}
where $p+q\geq2$.
\end{conjecture}

\begin{remark}
If the equation (\ref{uu}) is considered over the \textsl{binary ring} of
integers $\mathbb{Z}$, such that $u_{i},v_{j}\in\mathbb{Z}$, it leads to a
straightforward reformulation: for even powers it is obvious, but for odd
powers all negative terms can be rearranged and placed on the other side.
\end{remark}

Let us consider the Diophantine equation (\ref{uu}) over polyadic integers
$\mathbb{Z}_{\left(  m,n\right)  }$ (i.e. over the polyadic $\left(
m,n\right)  $-ary ring $\mathcal{R}_{m,n}^{\mathbb{Z}}$) such that
$u_{i},v_{j}\in\mathcal{R}_{m,n}^{\mathbb{Z}}$. We use the \textquotedblleft
long products\textquotedblright\ $\mu_{n}^{\left(  l\right)  }$ and $\nu
_{m}^{\left(  l\right)  }$ containing $l$ operations, and also the
\textquotedblleft polyadic power\textquotedblright\ for an element
$x\in\mathcal{R}_{m,n}^{\mathbb{Z}}$ with respect to $n$-ary multiplication
\cite{dup2012}%
\begin{equation}
x^{\left\langle l\right\rangle _{n}}=\mu_{n}^{\left(  l\right)  }\left[
\overset{l\left(  n-1\right)  +1}{\overbrace{x,x,\ldots,x}}\right]  .
\label{xm}%
\end{equation}
In the binary case, $n=2$, the polyadic power coincides with $\left(
l+1\right)  $ power of an element $x^{\left\langle l\right\rangle _{2}%
}=x^{l+1}$, which explains \textsc{Remark \ref{rem-power}}. In this notation
the polyadic analog of the equal sums of like powers Diophantine equation has
the form%
\begin{equation}
\nu_{m}^{\left(  p\right)  }\left[  u_{1}^{\left\langle l\right\rangle _{n}%
},u_{2}^{\left\langle l\right\rangle _{n}},\ldots,u_{p\left(  m-1\right)
+1}^{\left\langle l\right\rangle _{n}}\right]  =\nu_{m}^{\left(  q\right)
}\left[  v_{1}^{\left\langle l\right\rangle _{n}},v_{2}^{\left\langle
l\right\rangle _{n}},\ldots,v_{q\left(  m-1\right)  +1}^{\left\langle
l\right\rangle _{n}}\right]  , \label{uv}%
\end{equation}
where $p$ and $q$ are number of $m$-ary additions in l.h.s. and r.h.s.
correspondingly. The solutions of (\ref{uv}) will be denoted by $\left\{
u_{1},u_{2},\ldots,u_{p\left(  m-1\right)  +1};v_{1},v_{2},\ldots,v_{q\left(
m-1\right)  +1}\right\}  $. In the binary case $m=2,n=2$, (\ref{uv}) reduces
to (\ref{uu}). Analogously, we mark the solutions of (\ref{uv}) by the
\textit{polyadic triple} $\left(  l\mid p,q\right)  _{r}^{\left(  m,n\right)
}$. Now the \textit{polyadic Pythagorean triple} $\left(  1\mid0,1\right)
^{\left(  m,n\right)  }$, having one term on the l.h.s., one $m$-ary addition
on the r.h.s. and one $n$-ary multiplication (elements are in the first
polyadic power $\left\langle 1\right\rangle _{n}$), becomes%
\begin{equation}
u_{1}^{\left\langle 1\right\rangle _{n}}=\nu_{m}\left[  v_{1}^{\left\langle
1\right\rangle _{n}},v_{2}^{\left\langle 1\right\rangle _{n}},\ldots
,v_{m}^{\left\langle 1\right\rangle _{n}}\right]  . \label{pt}%
\end{equation}

\begin{definition}
The equation (\ref{pt}) solved by minimal $u_{1},v_{i}\in\mathbb{Z}$,
$i=1,\ldots,m$ can be named the \textit{polyadic Pythagorean theorem}.
\end{definition}

The \textit{polyadic Fermat's triple} $\left(  l\mid0,1\right)  ^{\left(
m,n\right)  }$ has one term in l.h.s., one $m$-ary addition in r.h.s. and $l$
($n$-ary) multiplications%
\begin{equation}
u_{1}^{\left\langle l\right\rangle _{n}}=\nu_{m}\left[  v_{1}^{\left\langle
l\right\rangle _{n}},v_{2}^{\left\langle l\right\rangle _{n}},\ldots
,v_{m}^{\left\langle l\right\rangle _{n}}\right]  . \label{pf}%
\end{equation}

One may be interested in whether the polyadic analog of Fermat's last theorem
is valid, and if not, in which cases the analogy with the binary case can be sustained.

\begin{conjecture}
[Polyadic analog of Fermat's Last Theorem]\label{conj-pf} The polyadic
Fermat's triple (\ref{pf}) has no solutions over the polyadic $\left(
m,n\right)  $-ary ring $\mathcal{R}_{m,n}^{\mathbb{Z}}$, if $l\geq2$, i.e.
there are more than one $n$-ary multiplications.
\end{conjecture}

Its straightforward generalization leads to the polyadic version of the
Lander-Parkin-Selfridge conjecture, as

\begin{conjecture}
[Polyadic Lander-Parkin-Selfridge conjecture]\label{conj-plps}There exist
solutions of the polyadic analog of the equal sums of like powers Diophantine
equation (\ref{uv}) in integers, if the number of $n$-ary multiplications is
less than or equal than the total number of $m$-ary additions plus one%
\begin{equation}
3\leq l\leq l_{pLPS}=p+q+1.
\end{equation}

\end{conjecture}

Below we will see a counterexample to both of the above conjectures.

\begin{example}
\label{ex-pyth32}Let us consider the $\left(  3,2\right)  $-ring
$\mathcal{R}_{3,2}^{\mathbb{Z}}=\left\langle \mathbb{Z\mid\nu}_{3},\mu
_{2}\right\rangle $, where%
\begin{align}
\nu_{3}\left[  x,y,z\right]   &  =x+y+z+2,\label{vu3}\\
\mu_{2}\left[  x,y\right]   &  =xy+x+y. \label{mu2a}%
\end{align}
Note that this exotic polyadic ring is commutative and cancellative, having
unit $0$, no multiplicative inverses, and for any $x\in\mathcal{R}%
_{3,2}^{\mathbb{Z}}$ its additive querelement $\tilde{x}=-x-2$, therefore
$\left\langle \mathbb{Z\mid\nu}_{3}\right\rangle $ is a ternary group (as it
should be). The polyadic power of any element is%
\begin{equation}
x^{\left\langle l\right\rangle _{2}}=\left(  x+1\right)  ^{l+1}-1. \label{xl}%
\end{equation}

1) For $\mathcal{R}_{3,2}^{\mathbb{Z}}$ the polyadic Pythagorean triple
$\left(  1\mid0,1\right)  ^{\left(  3,2\right)  }$ in (\ref{pt}) now is%
\begin{equation}
u^{\left\langle 1\right\rangle _{2}}=\nu_{3}\left[  x^{\left\langle
1\right\rangle _{2}},y^{\left\langle 1\right\rangle _{2}},z^{\left\langle
1\right\rangle _{2}}\right]  ,
\end{equation}
which, using (\ref{xm}), (\ref{mu2a}) and (\ref{xl}), becomes the (shifted)
Pythagorean quadruple \cite{spi62}%
\begin{equation}
\left(  u+1\right)  ^{2}=\left(  x+1\right)  ^{2}+\left(  y+1\right)
^{2}+\left(  z+1\right)  ^{2},
\end{equation}
and it has infinite number of solutions, among which two minimal ones
$\left\{  u=2;x=0,y=z=1\right\}  $ and $\left\{  u=14;x=1,y=9,z=10\right\}  $
give $3^{2}=1^{2}+2^{2}+2^{2}$ and $15^{2}=2^{2}+10^{2}+11^{2}$, correspondingly.

2) For this $\left(  3,2\right)  $-ring $\mathcal{R}_{3,2}^{\mathbb{Z}}$ the
polyadic Fermat's triple $\left(  l\mid0,1\right)  ^{\left(  3,2\right)  }$
becomes%
\begin{equation}
\left(  u+1\right)  ^{l+1}=\left(  x+1\right)  ^{l+1}+\left(  y+1\right)
^{l+1}+\left(  z+1\right)  ^{l+1}. \label{u}%
\end{equation}

If the polyadic analog of Fermat's last theorem\textbf{ \ref{conj-pf}} holds,
then there are no solutions to (\ref{u}) for more than one $n$-ary
multiplication $l\geq2$. But this is the particular case, $p=0$, $q=2$, of the
\textsl{binary} Lander-Parkin-Selfridge \textbf{Conjecture }\ref{conj-lps}
which now takes the form: the solutions to (\ref{u}) exist, if $l\leq3$. Thus,
as a \textsl{counterexample} to the polyadic analog of Fermat's last theorem,
we have two possible solutions with numbers of multiplications: $l=2,3$. In
the case of $l=2$ there exist two solutions: one well-known solution $\left\{
u=5;x=2,y=3,z=4\right\}  $ giving $6^{3}=3^{3}+4^{3}+5^{3}$ and another one
giving $709^{3}=193^{3}+461^{3}+631^{3}$ (J.-C. Meyrignac, 2000), while for
$l=3$ there exist an infinite number of solutions, and one of them (minimal)
gives $422481^{4}=95800^{4}+217519^{4}+414560^{4}$ \cite{elk88}.

3) The general polyadic triple $\left(  l\mid p,q\right)  ^{\left(
3,2\right)  }$, using (\ref{uv}), can be presented in the standard binary form
(as (\ref{uu}))%
\begin{equation}
\sum\limits_{i=1}^{2p+1}\left(  u_{i}+1\right)  ^{l+1}=\sum\limits_{j=1}%
^{2q+1}\left(  v_{i}+1\right)  ^{l+1},\ \ \ u_{i},v_{j}\in\mathbb{Z}.
\label{uv32}%
\end{equation}
Let us apply the \textsl{polyadic} Lander-Parkin-Selfridge \textbf{Conjecture}
\ref{conj-plps} for this case: the solutions to (\ref{uv32}) exist, if $3\leq
l\leq l_{pLSP}=p+q+1$. But the \textsl{binary} Lander-Parkin-Selfridge
\textbf{Conjecture} \ref{conj-lps}, applied directly, gives $3\leq l\leq
l_{LSP}=2p+2q+1$. So we should have \textsl{counterexamples} to the polyadic
Lander-Parkin-Selfridge conjecture, when $l_{pLSP}$ $<l\leq l_{LSP}$. For
instance, for $p=q=1$, we have $l_{pLSP}=3$, while the (minimal)
counterexample with $l=5$ is $\left\{  u_{1}=3,u_{2}=18,u_{3}=21,v_{1}%
=9,v_{2}=14,v_{3}=22\right\}  $ giving $3^{6}+19^{6}+22^{6}=10^{6}%
+15^{6}+23^{6}$ \cite{sub34}.
\end{example}

As it can be observed from \emph{Example }\ref{ex-pyth32}, the arity shape of
the polyadic ring $\mathcal{R}_{m,n}^{\mathbb{Z}}$ is crucial in constructing
polyadic analogs of the equal sums of like powers conjectures. We can make
some general estimations assuming a special (more or less natural) form of its
operations over integers.

\begin{definition}
\label{def-stand}We call $\mathcal{R}_{m,n}^{\mathbb{Z}}$ the \textit{standard
polyadic ring}, if the \textquotedblleft leading terms\textquotedblright\ of
its $m$-ary addition and $n$-ary multiplication are%
\begin{align}
\nu_{m}\left[  \overset{m}{\overbrace{x,x,\ldots,x}}\right]   &  \sim
mx,\label{nu1}\\
\mu_{n}\left[  \overset{n}{\overbrace{x,x,\ldots,x}}\right]   &  \sim
x^{n},\ \ \ x\in\mathbb{Z}. \label{mu1}%
\end{align}

\end{definition}

The polyadic ring $\mathcal{R}_{3,2}^{\mathbb{Z}}$ from \textit{Example
}\ref{ex-pyth32}, as well as the congruence class polyadic ring $\mathcal{R}%
_{m,n}^{\left[  a,b\right]  }$ (\ref{rz}) are both standard.

Using (\ref{xm}), we obtain approximate behavior of the polyadic power in the
standard polyadic ring%
\begin{equation}
x^{\left\langle l\right\rangle _{n}}\sim x^{l\left(  n-1\right)
+1},\ \ \ x\in\mathbb{Z},\ \ l\in\mathbb{N},\ \ n\geq2. \label{xln}%
\end{equation}
So increasing the arity of multiplication leads to higher powers, while
increasing arity of addition gives more terms in sums. Thus, the estimation
for the polyadic analog of the equal sums of like powers Diophantine equation
(\ref{uv}) becomes%
\begin{equation}
\left(  p\left(  m-1\right)  +1\right)  x^{l\left(  n-1\right)  +1}\sim\left(
q\left(  m-1\right)  +1\right)  x^{l\left(  n-1\right)  +1},\ \ \ x\in
\mathbb{Z}. \label{pqm}%
\end{equation}
Now we can apply the \textsl{binary} Lander-Parkin-Selfridge
\textbf{Conjecture }\ref{conj-lps} in the form: the solutions to (\ref{pqm})
can exist if $3\leq l\leq l_{LPS}$, where $l_{LPS}$ is an integer solution of%
\begin{equation}
\left(  n-1\right)  l_{LPS}=\left(  p+q\right)  \left(  m-1\right)  +1.
\label{npq}%
\end{equation}
On the other hand the \textsl{polyadic} Lander-Parkin-Selfridge
\textbf{Conjecture} \ref{conj-plps} gives: the solutions to (\ref{pqm}) can
exist if $3\leq l\leq l_{pLPS}=p+q+1$. Note that $\left(  p+q\right)  \geq2$ now.

An interesting question arises: which arities give the same limit, that is,
when $l_{pLPS}=l_{LPS}$?

\begin{proposition}
For any fixed number of additions in both sizes of the polyadic analog of the
equal sums of like powers Diophantine equation (\ref{uv}) $p+q\geq2$, there
exist limiting arities $m_{0}$ and $n_{0}$ (excluding the trivial binary case
$m_{0}=n_{0}=2$), for which the binary and polyadic Lander-Parkin-Selfridge
conjectures coincide $l_{pLPS}=l_{LPS}$, such that%
\begin{align}
m_{0}  &  =3+p+q+\left(  p+q+1\right)  k,\label{m2}\\
n_{0}  &  =2+p+q+\left(  p+q\right)  k,\ \ \ k\in\mathbb{N}^{0}. \label{n0}%
\end{align}

\end{proposition}

\begin{proof}
To equate $l_{LPS}=l_{pLPS}=p+q+1$ we use (\ref{npq}) and solve in integers
the equation%
\begin{equation}
\left(  n_{0}-1\right)  \left(  p+q+1\right)  =\left(  p+q\right)  \left(
m_{0}-1\right)  +1.
\end{equation}
In the trivial case, $m_{0}=n_{0}=2$, this is an identity, while the other
solutions can be found from $n_{0}\left(  p+q+1\right)  =\left(  p+q\right)
m_{0}+2$, which gives (\ref{m2})--(\ref{n0}).
\end{proof}

\begin{corollary}
In the limiting case $l_{pLPS}=l_{LPS}$ the arity of multiplication always
exceeds the arity of addition%
\begin{equation}
m_{0}-n_{0}=k+1,\ \ \ k\in\mathbb{N}^{0},
\end{equation}
and they start from $m_{0}\geq5$, $n_{0}\geq4$.
\end{corollary}

The first allowed arities $m_{0}$ and $n_{0}$ are presented in \textsc{Table
\ref{T2}}. Their meaning is the following.

\begin{corollary}
For the polyadic analog of the equal sums of like powers equation over the
standard polyadic ring $\mathcal{R}_{m,n}^{\mathbb{Z}}$ (with fixed $p+q\geq
2$) the polyadic Lander-Parkin-Selfridge conjecture becomes weaker than the
binary one $l_{pLPS}\geq l_{LPS}$, if: 1) the arity of multiplication exceeds
its limiting value $n_{0}$ with fixed arity of the addition; 2) the arity of
addition is lower, than its limiting value $m_{0}$ with the fixed arity of multiplication.
\end{corollary}

\begin{table}[h]
\caption{The limiting arities $m_{0}$ and $n_{0}$, which give $l_{pLPS}%
=l_{LPS}$ in (\ref{pqm}).}%
\label{T2}
\begin{center}%
\begin{tabular}
[c]{||c|c||c|c||c|c||}\hline\hline
\multicolumn{2}{||c||}{$p+q=2$} & \multicolumn{2}{c||}{$p+q=3$} &
\multicolumn{2}{c||}{$p+q=4$}\\\hline\hline
$m_{0}$ & $n_{0}$ & $m_{0}$ & $n_{0}$ & $m_{0}$ & $n_{0}$\\\hline
$5$ & $4$ & $6$ & $5$ & $7$ & $6$\\
$8$ & $6$ & $10$ & $8$ & $12$ & $10$\\
$11$ & $8$ & $14$ & $11$ & $17$ & $14$\\
$14$ & $10$ & $18$ & $14$ & $22$ & $18$\\\hline\hline
\end{tabular}
\end{center}
\end{table}

\begin{example}
Consider the standard polyadic ring $\mathcal{R}_{m,n}^{\mathbb{Z}}$ and fix
the arity of addition $m_{0}=12$, then take in (\ref{pqm}) the total number of
additions $p+q=4$ (the last column in \textsc{Table \ref{T2}}). We observe
that the arity of multiplication $n=16$, which exceeds the limiting arity
$n_{0}=10$ (corresponding to $m_{0}$). Thus, we obtain $l_{pLPS}=5$ and
$l_{LPS}=3$ by solving (\ref{npq}) in integers, and therefore the polyadic
Lander-Parkin-Selfridge conjecture becomes now weaker than the binary one, and
we do not obtain counterexamples to it, as in \emph{Example} \ref{ex-pyth32}
(where the situation was opposite $l_{pLPS}=3$ and $l_{LPS}=5$, and they
cannot be equal).
\end{example}

A concrete example of the standard polyadic ring (\textbf{Definition
\ref{def-stand}}) is the polyadic ring of the fixed congruence class
$\mathcal{R}_{m,n}^{\left[  a,b\right]  }$ considered in \textsc{Section}%
\textbf{ }\ref{sec-residue}, because its operations (\ref{nu})--(\ref{mu})
have the same straightforward behavior (\ref{nu1})--(\ref{mu1}). Let us
formulate the polyadic analog of the equal sums of like powers Diophantine
equation (\ref{uv}) over $\mathcal{R}_{m,n}^{\left[  a,b\right]  }$ in terms
of operations in $\mathbb{Z}$. Using (\ref{nu})--(\ref{mu}) and (\ref{xln})
for (\ref{uv}) we obtain%
\begin{equation}
\sum\limits_{i=1}^{p\left(  m-1\right)  +1}\left(  a+bk_{i}\right)  ^{l\left(
n-1\right)  +1}=\sum\limits_{j=1}^{q\left(  m-1\right)  +1}\left(
a+bk_{j}\right)  ^{l\left(  n-1\right)  +1},\ \ \ \ a,b,k_{i}\in\mathbb{Z}.
\label{pab}%
\end{equation}

It is seen that the leading power behavior of both sides in (\ref{pab})
coincides with the general estimation (\ref{pqm}). But now the arity shape
$\left(  m,n\right)  $ is fixed by (\ref{maa})--(\ref{ana}) and given in
\textsc{Table \ref{T1}}. Nevertheless, we can consider for (\ref{pab}) the
polyadic analog of Fermat's last theorem\textbf{ \ref{conj-pf}}, the
Lander-Parkin-Selfridge \textbf{Conjecture \ref{conj-lps}} (solutions exist
for $l\leq l_{LPS}$) and its polyadic version (\textbf{Conjecture
\ref{conj-plps}}, solutions exist for $l\leq l_{pLPS}$), as in the estimations
above. Let us consider some examples of solutions to (\ref{pab}).

\begin{example}
Let $\left[  \left[  2\right]  \right]  _{3}$ be the congruence class, which
is described by $\left(  4,3\right)  $-ring $\mathcal{R}_{4,3}^{\left[
2,3\right]  }$ (see \textsc{Table \ref{T1}}), and we consider the polyadic
Fermat's triple $\left(  l\mid0,5\right)  ^{\left(  4,3\right)  }$ (\ref{pf}).
Now the powers are $l_{LPS}=8$, $l_{pLPS}=6$, and for instance, if $l=2$, we
have solutions, because $l<l_{pLPS}<l_{LPS}$, and one of them is%
\begin{equation}
14^{5}=4\cdot\left(  -1\right)  ^{5}+7\cdot5^{5}+8^{5}+2\cdot11^{5}.
\end{equation}

\end{example}

\subsection{Frolov's Theorem and the Tarry-Escott problem}

A special set of solutions to the polyadic Lander-Parkin-Selfridge
\textbf{Conjecture \ref{conj-plps}} can be generated, if we put $p=q$ in
(\ref{pab}), which we call \textit{equal-summand solutions}\footnote{The term
\textquotedblleft symmetric solution\textquotedblright\ is already taken and
widely used \cite{bor2002}.}, by exploiting the Tarry-Escott problem approach
\cite{dor/bro} and Frolov's \textbf{Theorem \ref{theor-fro}}.

\begin{theorem}
If the set of integers $k_{i}\in\mathbb{Z}$ solves the Tarry-Escott problem%
\begin{equation}
\sum\limits_{i=1}^{p\left(  m-1\right)  +1}k_{i}^{r}=\sum\limits_{j=1}%
^{p\left(  m-1\right)  +1}k_{j}^{r},\ \ \ r=1,\ldots,s=l\left(  n-1\right)
+1, \label{pm}%
\end{equation}
then the polyadic equal sums of like powers equation with equal summands
(\ref{uv}) has a solution over the polyadic $\left(  m,n\right)  $-ring
$\mathcal{R}_{m,n}^{\left[  a,b\right]  }$ having the arity shape given by the
following relations:

\begin{enumerate}
\item Inequality%
\begin{equation}
l\left(  n-1\right)  +1\leq p\left(  m-1\right)  ; \label{n}%
\end{equation}

\item Equality%
\begin{equation}
p\left(  m-1\right)  =2^{l\left(  n-1\right)  +1}. \label{p}%
\end{equation}

\end{enumerate}
\end{theorem}

\begin{proof}
Using Frolov's \textbf{Theorem \ref{theor-fro}} applied to (\ref{pm}) we state
that%
\begin{equation}
\sum\limits_{i=1}^{p\left(  m-1\right)  +1}\left(  a+bk_{i}\right)  ^{r}%
=\sum\limits_{j=1}^{p\left(  m-1\right)  +1}\left(  a+bk_{j}\right)
^{r},\ \ \ r=1,\ldots,s=l\left(  n-1\right)  +1, \label{pmm}%
\end{equation}
for any fixed integers $a,b\in\mathbb{Z}$. This means, that (\ref{pmm}), with
$k_{i}$ (satisfying (\ref{pm})) corresponds to a solution to the polyadic
equal sums of like powers equation (\ref{uv}) for any congruence class
$\left[  \left[  a\right]  \right]  _{b}$. Nevertheless, the values $a$ and
$b$ are fixed by the restrictions on the arity shape and the relations
(\ref{m1}) and (\ref{an}).

\begin{enumerate}
\item It is known that the Tarry-Escott problem can have a solution only when
the powers are strongly less than the number of summands
\cite{bor2002,dor/bro}, that is $\left(  l\left(  n-1\right)  +1\right)
+1\leq p\left(  m-1\right)  +1$, which gives (\ref{n}).

\item A special kind of solutions, when number of summands is equal to 2 into
the number of powers, was found using the \textsl{Thue--Morse sequence}
\cite{all/sha}, which always satisfies the bound (\ref{n}), and in our
notation it is (\ref{p}).
\end{enumerate}

In both cases the relations (\ref{n}) and (\ref{p}) should be solved in
positive integers and with $m\geq2$ and $n\geq2$, which can lead to non-unique solutions.
\end{proof}

Let us consider some examples which give solutions to the polyadic equal sums
of like powers equation (\ref{uv}) with $p=q$ over the polyadic $\left(
m,n\right)  $-ring $\mathcal{R}_{m,n}^{\left[  a,b\right]  }$ of the fixed
congruence class $\left[  \left[  a\right]  \right]  _{b}$.

\begin{example}
1) One of the first ideal (non-symmetric) solutions to the Tarry-Escott
problem has 6 summands and 5 powers (A. Golden, 1944)%
\begin{equation}
0^{r}+19^{r}+25^{r}+57^{r}+62^{r}+86^{r}=2^{r}+11^{r}+40^{r}+42^{r}%
+69^{r}+85^{r},\ \ r=1,\ldots,5. \label{r6}%
\end{equation}
We compare with (\ref{pm}) and obtain%
\begin{align}
p\left(  m-1\right)   &  =5,\\
l\left(  n-1\right)   &  =4.
\end{align}
After ignoring binary arities we get $m=6$, $p=1$ and $n=3$, $l=2$. From
\textbf{Theorem \ref{theor-rz}} and \textsc{Table \ref{T1}} we observe the
minimal choice $a=4$ and $b=5$. It follows from Frolov's \textbf{Theorem
\ref{theor-fro}}, that all equations in (\ref{r6}) have symmetry
$k_{i}\rightarrow a+bk_{i}=4+5k_{i}$. Thus, we obtain the solution of the
polyadic equal sums of like powers equation (\ref{uv}) for the fixed
congruence class $\left[  \left[  4\right]  \right]  _{5}$ in the form
\begin{equation}
4^{5}+99^{5}+129^{5}+289^{5}+314^{5}+434^{5}=14^{5}+59^{5}+204^{5}%
+214^{5}+349^{5}+429^{5}.
\end{equation}
It is seen from \textsc{Table \ref{T1}} that the arity shape $\left(
m=6,n=3\right)  $ corresponds, e.g., to the congruence class $\left[  \left[
4\right]  \right]  _{10}$ as well. Using Frolov's theorem, we substitute in
(\ref{r6}) $k_{i}\rightarrow4+10k_{i}$ to obtain the solution in the
congruence class $\left[  \left[  4\right]  \right]  _{10}$%
\begin{equation}
4^{5}+194^{5}+254^{5}+574^{5}+624^{5}+864^{5}=24^{5}+114^{5}+404^{5}%
+424^{5}+694^{5}+854^{5}.
\end{equation}

2) To obtain the special kind of solutions to the Tarry-Escott problem we
start with the known one with 8 summands and 3 powers (see, e.g.,
\cite{leh47})%
\begin{equation}
0^{r}+3^{r}+5^{r}+6^{r}+9^{r}+10^{r}+12^{r}+15^{r}=1^{r}+2^{r}+4^{r}%
+7^{r}+8^{r}+11^{r}+13^{r}+14^{r},\ \ r=1,2,3.\label{r12}%
\end{equation}
So we have the concrete solution to the system (\ref{pm}) with the condition
(\ref{p}) which now takes the form $8=2^{3}$, and therefore%
\begin{align}
p\left(  m-1\right)   &  =7,\\
l\left(  n-1\right)   &  =2.
\end{align}
Excluding the trivial case containing binary arities, we have $m=8$, $p=1$ and
$n=3$, $l=1$. It follows from \textbf{Theorem \ref{theor-rz}} and
\textsc{Table \ref{T1}}, that $a=6$ and $b=7$, and so the polyadic ring is
$\mathcal{R}_{8,3}^{\left[  6,7\right]  }$. Using Frolov's \textbf{Theorem
\ref{theor-fro}}, we can substitute entries in (\ref{r12}) as $k_{i}%
\rightarrow a+bk_{i}=6+7k_{i}$ in the equation with highest power $r=3$ (which
is relevant to our task) and obtain the solution of (\ref{uv}) for $\left[
\left[  6\right]  \right]  _{7}$ as follows%
\begin{equation}
6^{3}+27^{3}+41^{3}+48^{3}+69^{3}+76^{3}+90^{3}+111^{3}=13^{3}+20^{3}%
+34^{3}+55^{3}+62^{3}+83^{3}+97^{3}+104^{3}.
\end{equation}

\end{example}

We conclude that consideration of the Tarry-Escott problem and Frolov's
theorem over polyadic rings gives the possibility of obtaining many nontrivial
solutions to the polyadic equal sums of like powers equation for fixed
congruence classes.

\bigskip

\subsection*{Acknowledgments}

The author would like to express his deep gratitude and sincere thankfulness
to Joachim Cuntz, Christopher Deninger, Grigorij Kurinnoj, Mike Hewitt, Jim
Stasheff, Alexander Voronov, and Wend Werner for discussions, and to
Dara~Shayda for Mathematica programming help.


\newpage

\mbox{}

\bigskip

\mbox{}
\bigskip
\listoftables

\end{document}